\documentclass[12pt]{amsbook}
\usepackage[all,cmtip]{xy}

\usepackage{a4wide}
\usepackage{pxfonts}
\usepackage{wasysym}
\usepackage{amssymb,amsxtra,mathrsfs}
\usepackage{eucal}
\usepackage{setspace,upgreek}
\usepackage{mathtools}
\usepackage{url}
\usepackage{hyperref}

\newcommand{\mf}[1]{\mathfrak{#1}}
\newcommand{\mc}[1]{\mathcal{#1}}
\newcommand{\mh}[1]{\mathscr{#1}}

\newcommand{\mr}[1]{\mathrm{#1}}

\newcommand{\ov}{\overline}

\newcommand{\lr}[2]{\langle #1,#2\rangle}

\newcommand{\up}{\upharpoonright}

\newcommand{\ep}{\varepsilon}

\newcommand{\ze}{\mathbf{0}} 
\newcommand{\un}{\mathbf{1}}

\newcommand{\K}{\mathbb{K}}
\newcommand{\R}{\mathbb{R}}

\newcommand{\C}{\mathbb{C}}
\newcommand{\Z}{\mathbb{Z}}

\newtheorem{theorem}{Theorem}
\newtheorem{definition}[theorem]{Definition}

\newtheorem{lemma}[theorem]{Lemma}
\newtheorem{proposition}[theorem]{Proposition}
\newtheorem{corollary}[theorem]{Corollary}

\newtheorem{assumption}[theorem]{Assumption}
\newtheorem{convention}[theorem]{Convention}

\newtheorem{notations}[theorem]{Notation}

\theoremstyle{definition}
\newtheorem{notation}[theorem]{Notation}
\newtheorem{introduction}
[theorem]{Introduction}

\newtheorem{remark}
[theorem]
{Remark}

\numberwithin{subsection}{section}
\numberwithin{theorem}{section}
\numberwithin{equation}{section}

\begin{document}
\title
{Operations on $\mh{M}\otimes_{\mc{A}(U)}\mr{Alt}^{\bullet}(U,M)$.
Vector-valued measures associated with locally integrable locally convex linear space valued forms.
Stokes theorem for locally convex linear space valued forms. 
Divergence theorem for locally convex linear space valued vector fields.}
\author{Benedetto Silvestri}
\email{bndsilvestri@proton.me}
\address{440 W Kelly Ave Apt 6 Jackson, WY 83001-8544 USA}
\address{PO BOX 14845 Jackson, WY 83002-4845 USA}
\begin{abstract}
We extend to $\mh{M}\otimes_{\mc{A}(U)}\mf{T}_{\bullet}^{\bullet}(U,M)$ the usual operations defined on 
$\mf{T}_{\bullet}^{\bullet}(U,M)$, 
then generalize to $\mh{M}\otimes_{\mc{A}(U)}\mr{Alt}^{\bullet}(U,M)$ the usual wedge product, insertion operator and exterior
differential defined on $\mr{Alt}^{\bullet}(U,M)$ and establish their properties.
Here $M$ is a smooth finite dimensional manifold, $U$ an open submanifold of $M$, $\mc{A}(U)$ the ring of smooth maps on $U$, 
$\mf{T}_{\bullet}^{\bullet}(U,M)$ the $\mc{A}(U)$-module of smooth tensor fields of $M$ defined on $U$, 
$\mr{Alt}^{\bullet}(U,M)$ the $\mc{A}(U)$-module of smooth alternating tensor fields of $M$ defined on $U$, 
while $\mh{M}$ is a $\mc{A}(U)$-module. $\mh{M}$ depending by the operation might be either general, or one of the following 
function spaces: the space of $G$-valued (respectively measurable $\mh{L}_{c}^{1}(U,G,\lambda)$) compactly supported scalarly 
integrable maps on $U$; the space of $G$-valued (respectively measurable $\mh{L}_{loc}^{1}(U,G,\lambda)$) scalarly locally 
integrable maps on $U$; the space $\mc{B}^{p}(U,G)$ of $G_{\R}$-valued $p$-times continuously differentiable maps defined 
on $U$; where $G$ is a real or complex Hausdorff locally convex space, $G_{\R}$ is the real Hausdorff locally convex space 
underlying $G$ and $p\in\Z_{>0}\cup\{+\infty\}$. If $H$ is a Hausdorff locally convex space, then a key contribution 
of this paper establishes that 
the tensor product induces well-defined $\mc{A}(U)$-bilinear maps 
from $\mh{L}_{c}^{1}(U,G,\lambda)\times\mc{B}^{p}(U,H)$ to $\mh{L}_{c}^{1}(U,G\widehat{\otimes}H,\lambda)$ and from
$\mh{L}_{loc}^{1}(U,G,\lambda)\times\mc{B}^{p}(U,H)$ to $\mh{L}_{loc}^{1}(U,G\widehat{\otimes}H,\lambda)$, where 
$G\widehat{\otimes}H$ is the completion of the projective tensor product of $G$ and $H$.
This result opens the possibility of extending the wedge product to 
$(\mh{M}\otimes_{\mc{A}(U)}\mr{Alt}^{\bullet}(U,M))\times(\mh{N}\otimes_{\mc{A}(U)}\mr{Alt}^{\bullet}(U,M))$ 
and the insertion operator to $\mh{M}\otimes_{\mc{A}(U)}\mr{Alt}^{\bullet}(U,M)$ 
where $\mh{M}$ and $\mh{N}$ are some of the $\mc{A}(U)$-modules of $G$-valued functions above listed.
This framenwork permits to construct the divergence of a $G$-valued $p$-times continuously differentiable vector field, 
to define the weak integral of $G$-valued compactly supported scalarly integrable maximal forms, 
to associate a $G$-valued measure with any locally integrable $G$-valued form, to establish a Stokes type theorem for 
$G$-valued compactly supported $p$-times continuously differentiable forms and to obtain a divergence type theorem for 
$G$-valued $p$-times continously differentiable vector fields.
\end{abstract}

\maketitle

\begin{notations}
\label{12152008}
If $A$ is a ring, then we convein that $A$-module stands always for left $A$-module.
Let $A-\mr{mod}$ be the category of $A$-modules and $A$-linear maps.
If $\rho:B\to A$ is a morphism of rings and $Z$ is a $A$-module, then we shall slightly modify the Bourbaki notation in 
\cite[II.30]{BourA1} and denote $\rho_{\star}(Z)$ (instead of $\rho_{\ast}(Z)$ as in Bourbaki)
the $B$-module $Z$ associated with $\rho$ where the law of action of $B$ on $\rho_{\star}(Z)$ 
is defined by $(b,z)\mapsto\rho(b)\cdot z$, where $\cdot$ is the law of action of $A$ on $Z$. Furthermore if $Y$ is a $A$-module 
and $\zeta:Y\to Z$ is a $A$-module morphism, then $\rho_{\star}(\zeta):\rho_{\star}(Y)\to\rho_{\star}(Z)$, $y\mapsto\zeta(y)$ is 
a $B$-module morphism. As in Bourbaki, $\zeta:\rho_{\star}(Y)\otimes_{B}\rho_{\star}(Z)\to\rho_{\star}(Y\otimes_{A}Z)$, 
is called the \emph{canonical map}, if $\zeta$ is the function that for every $(x,y)\in Y\times Z$ maps the element $x\otimes y$
of $\rho_{\star}(Y)\otimes_{B}\rho_{\star}(Z)$ into the element $x\otimes y$ of $\rho_{\star}(Y\otimes_{A}Z)$. 
Let $\K\in\{\R,\C\}$ and $\un_{\K}$ be the unit of the ring underlying $\K$; we denote by $\updelta^{\K}$ the morphism of rings 
such that $\updelta^{\K}:\R\to\K$, $t\mapsto t\un_{\K}$.
If $\{E_{i}\}_{i\in I}$ is a family of $A$-modules and $E=\bigoplus_{i\in I}E_{i}$, then for every $j\in I$ we let 
$\mf{i}_{j}^{E}:E_{j}\to E$ denote the canonical injection of $E_{j}$ into $E$, while let $\Pr_{j}^{E}:E\to E_{j}$ be the canonical
projection of $E$ onto $E_{j}$. If in addition $X\neq\emptyset$, $E$ is a $A$-module, then by letting $E_{x}=E$ for every 
$x\in X$ we let $E^{X}$ be the $A$-module product $\prod_{x\in X}E_{x}$ and $E^{(X)}$ be the $A$-module $\bigoplus_{x\in X}E_{x}$.
If $E$ is a $A$-module, then let $E^{\ast}$ be its $A$-dual, in addition if $S$ and $T$ are 
$A$-modules, then we let $\mr{bil}_{A}(E,S;T)$ denote the set of those maps defined on $E\times S$ and valued in $T$ which are 
$A$-bilinear. Next assume in the remaining of the present Notation section \ref{12152008} that $A$ is a \textbf{commutative} 
ring. Then $\mr{mor}_{A}(E,S)$ is a $A$-submodule of the $A$-module product $S^{E}$ and  $\mr{bil}_{A}(E,S;T)$ is a $A$-submodule 
of the $A$-module product $T^{E\times S}$. Let $r,s\in\Z_{\geq 0}$ and $E$ be a $A-$module, define 
$[E,r,s]$ to be such that $[E,0,0]\coloneqq A^{\ast}$, otherwise be the map on $[1,r+s]$ such that 
\begin{equation*}
\begin{aligned}
i\in[1,r]\cap\Z\Rightarrow[E,r,s]_{i}&\coloneqq E^{\ast},
\\
j\in[1,s]\cap\Z\Rightarrow[E,r,s]_{r+j}&\coloneqq E.
\end{aligned}
\end{equation*}
Let $\prod[E,0,0]\coloneqq[E,0,0]$ and and let $\prod[E,r,s]$ be the $A$-module product $\prod_{i=1}^{r+s}[E,r,s]_{i}$.
If $F$ is a $A$-module, then define $\mf{T}_{s}^{r}(E,F)$ be the $A$-module of $A$-multilinear maps from $\prod[E,r,s]$
into $F$ whose elements are called tensors on $E$ of type $(r,s)$ at values in $F$. Define 
\begin{equation*}
\mf{T}_{\bullet}^{\bullet}(E,F)\coloneqq\bigoplus_{(t,j)\in\Z_{\geq 0}\times\Z_{\geq 0}}\mf{T}_{j}^{t}(E,F)
\end{equation*}
moreover let $\mf{T}_{s}^{r}(E)\coloneqq\mf{T}_{s}^{r}(E,A)$ and identify $\mf{T}_{0}^{0}(E)$ with $A$.
Let $\mr{Alt}^{k}(E,F)$ be the $A$-submodule of the alternating maps in $\mf{T}_{k}^{0}(E,F)$ and let $\mr{Alt}^{k}(E)$ denote 
$\mr{Alt}^{k}(E,A)$. 
\par
Let $M$ be a smooth manifold with or without boundary, $n=\mr{dim}\,M$ and $U$ be an open set of $M$. 
A chart and an atlas of $M$ are understood smooth. Let $\mc{F}(M,\R)$ denote the ring $\R^{M}$ 
of maps defined on $M$ and valued in $\R$ and let $\un_{M}$ its unity, while let $\mc{C}(M,\R)$ be the subring of $\mc{F}(M,\R)$ 
made of continuous maps defined on $M$ and valued in $\R$. If $p\in\Z_{\geq 0}\cup\{+\infty\}$, then let 
$\mc{A}^{p}(M)\coloneqq\mc{C}^{p}(M,\R)$ be the subring of $\mc{F}(M,\R)$ made of $p$-times continuously differentiable real 
valued maps defined on $M$, with the usual convention of letting $\mc{C}^{0}(M,\R)=\mc{C}(M,\R)$.
Let $\mc{A}(M)$ denote $\mc{A}^{\infty}(M)$ moreover let $\mf{a}(M)$ denote the ideal of $\mc{A}(M)$ made 
of those $f\in\mc{A}(M)$ whose support is compact. Let $\uprho^{M}$ be the morphism of rings so defined 
$\uprho^{M}:\R\to\mc{A}(M),\,t\mapsto t\cdot\un_{M}$, where $\cdot$ is the law of action of $\R$ on the $\R$-linear space 
$\R^{M}$. Let $TM$ and $T^{\ast}M$ be the tangent and cotangent bundle of $M$ respectively provided with their usual real vector
bundle structure. If $\mc{V}$ is a smooth real vector bundle over $M$, then let $\Gamma_{0}(U,\mc{V})$, 
$\Gamma^{0}(U,\mc{V})$,
$\Gamma(U,\mc{V})$ and 
$\Gamma_{c}(U,\mc{V})$ 
be the 
$\mc{F}(U,\R)$-module of sections;
$\mc{C}(U,\R)$-module of continuous sections;
$\mc{A}(U)$-module of smooth sections; 
$\mc{A}(U)$-module of compactly supported smooth sections 
respectively of the restriction at $U$ of $\mc{V}$. 
If $r,s\in\Z_{\geq 0}$, then let $\mf{T}_{s}^{r}(U,M)\coloneqq\mf{T}_{s}^{r}(\Gamma(U,TM),\mc{A}(U))$ and 
$\mf{T}_{\bullet}^{\bullet}(U,M)\coloneqq\mf{T}_{\bullet}^{\bullet}(\Gamma(U,TM),\mc{A}(U))$.
Let $\mf{T}_{s}^{r}(TM)$ be the real vector bundle over $M$ whose fiber at $p$ equals $\mf{T}_{s}^{r}(T_{p}M)$ 
(recall that $T_{p}M$ is a real vector space), while if $k\in\Z_{\geq 0}$, then let
$\mr{Alt}^{k}(U,M)\coloneqq\mr{Alt}^{k}(\Gamma(U,M))$
and let
$\mr{Alt}^{k}(TM)$ be the real vector bundle over $M$ whose fiber at $p$ equals $\mr{Alt}^{k}(T_{p}M)$.
Set
$\mf{T}_{\bullet}^{\bullet}(TM)\coloneqq\bigoplus_{(r,s)\in\Z_{\geq 0}\times\Z_{\geq 0}}\mf{T}_{s}^{r}(TM)$;
while
$\mr{Alt}^{\bullet}(U,M)\coloneqq\bigoplus_{k\in\Z_{\geq 0}}\mr{Alt}^{k}(U,M)$
and 
$\mr{Alt}^{\bullet}(TM)\coloneqq\bigoplus_{k\in\Z_{\geq 0}}\mr{Alt}^{k}(TM)$.
We set $\Omega^{k}(U,M)\coloneqq\Gamma(U,\mr{Alt}^{k}(TM))$ and
$\Omega^{\bullet}(U,M)\coloneqq\bigoplus_{k\in\Z_{\geq 0}}\Omega^{k}(U,M)$.
Clearly $\mr{Alt}^{\bullet}(U,M)$, $\mr{Alt}^{\bullet}(TM)$ and $\Omega^{\bullet}(U,M)$ equal the direct sum over
$[1,n]\cap\Z$.
\par
We shall denote by $\mf{r}_{\R}^{U}$ or simply $\mf{r}^{U}$ 
the usual $\mc{A}(U)$-isomorphism from $\Gamma(U,\mf{T}_{\bullet}^{\bullet}(TM))$ onto $\mf{T}_{\bullet}^{\bullet}(U,M)$ 
and by $\mf{t}_{\R}^{U}$ or simply $\mf{t}^{U}$ the inverse of $\mf{r}_{\R}^{U}$.
\par
Given a chart $(U,\phi)$ of $M$, in order to keep the notation as light as possible we
convein to let $dx_{i}^{\phi}\in\Gamma(U,T^{\ast}M)$ denote also $\mf{t}(dx_{i}^{\phi})\in\Gamma(U,TM)^{\ast}$.
Moreover we let $\{(\otimes(b^{r,s,\phi})^{\ast})_{\alpha}\,\vert\,\alpha\in\Xi(b^{r,s,\phi})\}$
and $\{\mc{E}_{dx^{\phi}}(I)\,\vert\,I\in M(k,N,<)\}$ be the basis of $\mf{T}_{s}^{r}(U,M)$ and $\mr{Alt}^{k}(U,M)$
image via the isomorphism $\mf{r}$ of the basis of $\Gamma(U,\mf{T}_{s}^{r}(TM))$ and $\Gamma(U,\mr{Alt}^{k}(TM))$ 
associated with the chart $(U,\phi)$ respectively. 
\par
Let $\K\in\{\R,\C\}$ and let $G$ and $H$ be Hausdorff locally convex spaces on $\K$.
We let $G_{0}$ denote the linear space over $\R$ underlying $G$, while let 
$G_{\R}$ denote the Hausdorff locally convex space over $\R$ underlying $G$.
If $S$ is a set, then by abuse of language we let $G^{S}$ be the product $\K$-linear space of maps defined on $S$ and valued in 
the $\K$-linear space underlying $G$. If $\Gamma$ is a fundamental set of seminorms of $G$, then $\Gamma$ is a fundamental set 
of seminorms of $G_{\R}$. Let $\mc{L}(G,H)$ be the $\K$-linear space of continuous linear maps defined on $G$ and valued in $H$, 
set $G^{\prime}\coloneqq\mc{L}(G,\K)$ and then $(G^{\prime})^{\ast}$ is the algebraic dual of $G^{\prime}$ namely 
$(G^{\prime})^{\ast}=\mr{Mor}_{\K-\mr{mod}}(G^{\prime},\K)$.  
Let $\widehat{G}$ denote the Hausdorff locally convex completion of $G$ and let $\mf{j}_{G}\in\mc{L}(G,\widehat{G})$ denote 
the canonical map of $G$ into $\widehat{G}$. 
If $T\in\mc{L}(G,H)$, then we let $\widehat{T}\in\mc{L}(\widehat{G},\widehat{H})$ denote the unique continuous extension of 
$T$, namely the unique map $\widehat{T}$ such that $\widehat{T}\circ\mf{j}_{G}=\mf{j}_{H}\circ T$. 
We have $\widehat{S\circ T}=\widehat{S}\circ\widehat{T}$. 
Next let $G\otimes H$ denote the projective topological tensor product of $G$ and $H$ \cite[p.32 Def.2]{gro} which is a 
locally convex linear space on $\K$ and let $G\widehat{\otimes}H$ denote $\widehat{G\otimes H}$. If $G_{1}$ and $H_{1}$ are 
Hausdorff locally convex spaces on $\K$, and $\phi\in\mc{L}(G,H)$ and $\psi\in\mc{L}(G_{1},H_{1})$, then 
$\phi\otimes\psi\in\mc{L}(G\otimes H,G_{1}\otimes H_{1})$ indeed the map 
$\otimes\circ(\phi\times\psi):G\times H\mapsto G_{1}\otimes H_{1}$ is bilinear and continuous since by construction the topology
on $G_{1}\otimes H_{1}$ is the finest locally convex topology on $G_{1}\otimes H_{1}$ rendering 
$\otimes:G_{1}\times H_{1}\to G_{1}\otimes H_{1}$ continuous. We let $\phi\widehat{\otimes}\psi$ denote 
$\widehat{\phi\otimes\psi}$.
\par
If $A,B,C$ are sets and $f:A\to B$, $g:A\to C$, then $f\boxtimes g:A\to B\times C$, $a\mapsto(f(a),g(a))$.
If $X$ is a topological space, then $\mr{Cl}(X)$ denotes the set of closed sets of $X$, while $\mc{C}(X,G)$ denotes 
the $\K$-linear space of continuous maps defined on $X$ and valued in $G$, while $\mc{C}_{c}(X,G)$ denotes the $\K$-linear 
subspace of $\mc{C}(X,G)$ made of maps compactly supported.
If $W$ is an open set of $\R^{n}$, then let $\mc{C}^{0}(W,G_{\R})=\mc{C}(W,G_{\R})$, while if $k\in\Z_{>0}\cup\{+\infty\}$, then 
for the moment let $\mc{C}^{k}(W,G_{\R})$ be the set (for the linear structure see bolow) of maps defined on $W$, valued in 
$G_{\R}$ and $k$-times continuously differentiable in the sense of bounded differential calculus see \cite{yam} and also 
\cite{kll}. Since $W\subseteq\R^{n}$ this ammounts to be equivalent to say that $\mc{C}^{k}(W,G_{\R})$ is the set of maps defined 
on $W$, valued in $G_{\R}$ and $k$-times continuously derivable where a map is derivable in the usual sense with the limit of 
the incremental ratio performed with respect to the topology of $G_{\R}$. Notice that $\mc{C}^{\infty}(W,G_{\R})$ equals the set of 
smooth maps in the sense of Bastiani calculus for which we refer the reader primary to \cite{glo}, see also \cite{kll}.
If $\K=\C$, then the map $\eta_{G}:G_{\R}\to G_{\R}$, $x\mapsto ix$ is $\R$-linear and such that 
$\eta_{G}\circ\eta_{G}=-\mr{Id}_{G_{\R}}$ moreover $q\circ\eta_{G}=q$ for every seminorm on $G$ therefore, $\eta_{G}$ is an 
automorphism (with inverse $-\eta_{G}$) of the linear topological space $G_{\R}$. Note that 
\begin{equation}
\label{11241319a}
(\forall g\in G^{W})(ig=\eta_{G}\circ g)
\end{equation}
moreover since $\eta_{G}$ is $\R$-linear and continuous we have for every $p\in\Z_{>0}\cup\{+\infty\}$ that 
\begin{equation}
\label{11241319b}
f\in\mc{C}^{p}(W,G_{\R})\Leftrightarrow\eta_{G}\circ f\in\mc{C}^{p}(W,G_{\R})
\end{equation}
and for every $r\in[1,p]\cap\Z$ the $r$-differential of $\eta_{G}\circ f$ at every $x\in W$ equals $\eta_{G}\circ(f^{(r)}(x))$ 
where $f^{(r)}(x)$ is the $r$-differential of $f$ at $x$. Therefore by \eqref{11241319a} and \eqref{11241319b} we conclude that 
$\mc{C}^{p}(W,G_{\R})$ is a $\C$-linear subspace of $G^{W}$. By returning to the general case $\K\in\{\R,\C\}$ we then conclude 
that $\mc{C}^{p}(W,G_{\R})$ is a $\K$-linear subspace of $G^{W}$. 
Therefore we can provide the following 
\begin{convention}
\label{07071626}
\textbf{In the present work we shall consider $\mc{C}^{p}(W,G_{\R})$ a $\K$-linear subspace of $G^{W}$.}
\end{convention}
Next for every $(a,v)\in W\times\R^{n}$ and $f\in\mc{C}^{1}(W,G_{\R})$ we let $D_{v}^{W,G}\vert_{a}f$ denote the 
derivative of $f$ at $a$ in the direction $v$, and let $D_{v}^{W,G}f:W\ni a\mapsto D_{v}^{W,G}\vert_{a}f\in G$. 
\par
In this work we employ measure theory as developped in \cite{IntBourb}. We remark what follows.
Since the discussion in \cite[VI.40 end of $n^{\circ}$10]{IntBourb} and that in \cite[II.65]{EVT} every definition given in
\cite{IntBourb} regarding locally convex spaces over $\R$, can be generalized to the case of locally convex spaces over $\C$ 
therefore, when we quote \cite{IntBourb} we always assume implicitly quoted the discussion in 
\cite[VI.40 end of $n^{\circ}$10]{IntBourb} and that in \cite[II.65]{EVT}.
If $X$ is a topological space and $E$ is a Hausdorff locally convex space over $\K$,
then we let $\mc{H}(X,E)$ be the $\K$-linear space of compactly supported continous maps defined on $X$
and with values in $E$ provided with the usual locally convex topology, we let $\mc{H}(X)\coloneqq\mc{H}(X,\K)$.
If in addition $X$ is locally compact, then let $\mc{M}(X,E)$ be the $\K$-linear space of vector-measures on $X$ with 
values in $E$, namely the space of $\K$-linear and continuous maps from $\mc{H}(X)$ into $E$ \cite[VI.18 Def. 1]{IntBourb}.
Let $\mc{M}(X)$ denote $\mc{M}(X,\K)$ whose elements are called measures on $X$ 
\cite[Def. 2, $\S1$, $n^{\circ}3$, Ch. $3$]{IntBourb}. For any $\mu\in\mc{M}(X)$ and $p\in[1,\infty]$, let $\mf{L}^{p}(X,\mu)$ and 
respectively $\mf{L}_{loc}^{1}(X,\mu)$ be the $\K$-linear space of maps defined on $X$ at values in $\K$ which are $p$-th power 
$\mu$-integrable and respectively locally $\mu$-integrable (\cite[V.39 Def.1]{IntBourb}).
Given a Hausdorff locally convex space $G$ over $\K$ and a map $f:X\to G$, we say that $f$ is scalarly (respectively scalarly 
locally) $\mu$-integrable iff $\uppsi\circ f$ is $\mu$-integrable
(respectively locally $\mu$-integrable) for every $\uppsi\in G^{\prime}$. Furthermore if $f$ is scalarly $\mu$-integrable we say 
that the integral of $f$ belongs to $G$ iff there exists a necessarily unique $s\in G$ such that $\uppsi(s)=\int\uppsi\circ f$ 
for every $\uppsi\in G^{\prime}$ in which case we set $\int f\coloneqq s$.
\end{notations}
\begin{assumption}
\label{10041445}
In what follows we let $K\in\{\R,\C\}$ and let $G$, $H$, $G_{1}$ and $H_{1}$ be Hausdorff locally convex spaces over $\K$,
and let $M$ be a finite dimensional smooth manifold $M$, with or without boudary, such that $n\coloneqq\mr{dim}\,M>0$.
Let $\lambda$ be the Lebesgue measure on $\R^{n}$ and for every open set $A$ of $\R^{n}$ we let $\lambda_{A}$ be the restriction 
at $A$ of $\lambda$.
\end{assumption}
\begin{introduction}
In this Introduction we employ Notation \ref{12152008} and assume Assumption \ref{10041445}.
In section \ref{09172012} we define and study some properties of spaces of $G$-valued maps employed in the remaning of 
the paper. In particular the space $\mc{B}^{p}(M,G)$ (respectively $\mc{B}_{c}^{p}(M,G)$)
of $G$-valued (respectively compactly supported) $p$-times continuously differentiable maps on $M$
in Def.\ref{10021123}, then the space $\mf{L}_{loc}^{1}(M,G,\lambda)$ (respectively $\mf{L}_{c,loc}^{1}(M,G,\lambda)$)
of $G$-valued (respectively compactly supported) scalarly locally integrable maps on $M$ in Def.\ref{10111122}, 
the space $\mh{L}_{loc}^{1}(M,G,\lambda)$ (respectively $\mh{L}_{c,loc}^{1}(M,G,\lambda)$)
of $G$-valued (respectively compactly supported) measurable scalarly locally integrable maps on $M$ 
in Def.\ref{07152211}, the space $\mf{L}^{1}(M,G,\lambda)$ (respectively $\mf{L}_{c}^{1}(M,G,\lambda)$) of $G$-valued 
(respectively compactly supported) scalarly integrable maps on $M$ in Def.\ref{09281553} and the space 
$\mh{L}^{1}(M,G,\lambda)$ (respectively$\mh{L}_{c}^{1}(M,G,\lambda)$) of $G$-valued (respectively compactly supported) measurable 
scalarly integrable maps on $M$ in Def.\ref{07152212}. We prove the $\mc{A}^{p}(M)$-module structure of some of the above spaces 
in Cor.\ref{10111859}.
\par
In section \ref{09251630} we define and establish elementary properties of the $\mc{A}(U)$-tensor product of a general 
$\mc{A}(U)$-module $\mh{M}$ with the $\mc{A}(U)$-module and submodules of tensor fields of $M$ defined on $U$ with $U$ open of 
$M$. In particular we state various isomorphisms and the representation Thm.\ref{11151400} of 
$\mh{M}\otimes_{\mc{A}(U)}\mf{T}_{\bullet}^{\bullet}(U,M)$ in terms of $\mf{T}_{\bullet}^{\bullet}(\Gamma(U,M),\mh{M})$.
The bar operator over these tensor products is trivially extended.
After having collected under the symbol $\mh{Q}[r]_{U}^{M}(G)$ any of the following $\mc{A}(U)$-modules: $\mc{B}^{p}(U,G)$, 
$\mf{L}_{loc}^{1}(U,G,\lambda)$, $\mf{L}_{c}^{1}(U,G,\lambda)$, $\mh{L}_{loc}^{1}(U,G,\lambda)$ and $\mh{L}_{c}^{1}(U,G,\lambda)$,
and having verified that the pullback operator is well-defined on $\mh{Q}[r]_{U}^{M}(G)$, we extend it to 
$\mh{Q}[r]_{U}^{M}(G)\otimes_{\mc{A}(U)}\mf{T}_{\bullet}^{0}(U,M)$. Similarly pushforward via linear and continuous maps is defined 
on $\mh{Q}[r]_{U}^{M}(G)\otimes_{\mc{A}(U)}\mh{N}$ with $\mh{N}$ a $\mc{A}(U)$-module, commutation properties with pullback are then
established. We end the section by exhibiting a decomposition of $\uprho_{\star}^{U}(\mh{Q}[r]_{U}^{M}(G))$ when $G$ is a 
topological direct sum. 
\par
In section \ref{09172013} after exhibiting unique decompositions of 
$\mr{Alt}^{k}(U,M;\mh{M})\coloneqq\mh{M}\otimes_{\mc{A}(U)}\mr{Alt}^{k}(U,M)$ 
we establish Thm.\ref{09251928} an essential result of this paper.  
Let $\ov{\otimes}:G\times H\to G\widehat{\otimes}H$ be defined by $\ov{\otimes}=\mf{j}_{G\otimes H}\circ\otimes$ where 
$\otimes:G\times H\to G\otimes H$ is the bilinear and continuous tensor product map defined on $G\times H$ and 
valued in the projective tensor product $G\otimes H$.
Now under suitable hypothesis on the topologies of $G$ and $H$ we state in the key \textbf{Thm.\ref{09251928}} 
that $(f,g)\mapsto\ov{\otimes}\circ(f\boxtimes g)$ induces well-defined $\mc{A}(U)$-bilinear maps 
from $\mh{L}_{c}^{1}(U,G,\lambda)\times\mc{B}^{p}(U,H)$ to $\mh{L}_{c}^{1}(U,G\widehat{\otimes}H,\lambda)$ and from
$\mh{L}_{loc}^{1}(U,G,\lambda)\times\mc{B}^{p}(U,H)$ to $\mh{L}_{loc}^{1}(U,G\widehat{\otimes}H,\lambda)$.
This result is essential in extending the insertion operator and the wedge product.
\par
Indeed given an arbitrary operator $\upeta\in\mr{mor}_{\mc{A}(U)-\mr{mod}}(\mh{X}\otimes_{\mc{A}(U)}\mh{Y},\mh{Z})$ where 
$\mh{X},\mh{Y},\mh{Z}$ are $\mc{A}(U)$-modules, and via Thm.\ref{09251928} and a similar result for the 
differential cases, we state in \textbf{Prp.\ref{120020748}} that for every $(g,a)\in\mc{B}^{p}(U,H)\times\mh{X}$ 
there exists a unique extension of $\upeta$ as a $\mc{A}(U)$-linear map $\widehat{\upeta}_{r,g,a,p}^{G,H}$ from 
$\mh{Q}[r]_{U}^{M}(G)\otimes_{\mc{A}(U)}\mh{Y}$ into $\mh{Q}[x(r,p)]_{U}^{M}(G\widehat{\otimes}H)\otimes_{\mc{A}(U)}\mh{Z}$.
This preliminary result permits to obtain in \textbf{Prp.\ref{12020810}} the following uniqueness 
\begin{multline*}
\bigl(\exists\,!\widehat{\upeta}_{r,p}^{G,H}
\in
\\
\mr{mor}_{\mc{A}(U)-\mr{mod}}
\bigl(\mc{B}^{p}(U,H)\otimes_{\mc{A}(U)}\mh{X},
\mr{mor}_{\mc{A}(U)-\mr{mod}}
\bigl(\mh{Q}[r]_{U}^{M}(G)\otimes_{\mc{A}(U)}\mh{Y},\mh{Q}[x(r,p)]_{U}^{M}(G\widehat{\otimes}H)\otimes_{\mc{A}(U)}\mh{Z}\bigr)
\bigr)\bigr)
\\
(\forall(g,a)\in\mc{B}^{p}(U,H)\times\mh{X})
\bigl(\widehat{\upeta}_{r,p}^{G,H}(g\otimes a)=\widehat{\upeta}_{r,g,a,p}^{G,H}\bigr),
\end{multline*}
and in \textbf{Cor.\ref{12020811}} the corresponding bilinear map. 
\par
These two results allow to construct the \emph{extension of the insertion operator} in \textbf{Def.\ref{12022149}}, 
and the \emph{extension of the wedge product} in \textbf{Def.\ref{07191816}} respectively. Furthermore together a commutation 
property we can state in Prp.\ref{12051823} and Prp.\ref{09290547} that the above extensions exibith commutation properties with
the pushforward via linear and continuous maps. The commutation of the extension of the wedge product with pullback is provided 
in Prp.\ref{11300106}.
More specifically and by letting $\mr{Alt}^{\bullet}(U,M;\mh{M})\coloneqq\mh{M}\otimes_{\mc{A}(U)}\mr{Alt}^{\bullet}(U,M)$,
the wedge product extends from an operation on $\mr{Alt}^{\bullet}(U,M)$ to a $\mc{A}(U)$-bilinear map: 
\begin{equation*}
\wedge_{r,p}^{U,M,G,H}:
\mr{Alt}^{\bullet}(U,M;\mh{Q}[p]_{U}^{M}(H))\times\mr{Alt}^{\bullet}(U,M;\mh{Q}[r]_{U}^{M}(G))
\to
\mr{Alt}^{\bullet}(U,M;\mh{Q}[x(r,p)]_{U}^{M}(G\widehat{\otimes}H));
\end{equation*}
while the operation of insertion of a vector field generalizes to a map
\begin{multline*}
\mf{i}_{r,p}^{U,M,G,H}\in
\mr{mor}_{\mc{A}(U)-\mr{mod}}
\bigl(\mc{B}^{p}(U,H)\otimes_{\mc{A}(U)}\Gamma(U,TM),
\\
\mr{mor}_{\mc{A}(U)-\mr{mod}}
\bigl(\mr{Alt}^{\bullet}(U,M;\mh{Q}[r]_{U}^{M}(G)),\mr{Alt}^{\bullet}(U,M;\mh{Q}[x(r,p)]_{U}^{M}(G\widehat{\otimes}H))\bigr)
\bigr).
\end{multline*}
where $\mc{B}^{p}(U,H)\otimes_{\mc{A}(U)}\Gamma(U,TM)$ is the $\mc{A}(U)$-module of $H$-valued $p$-times continuously differentiable 
vector fields of $M$ defined on $U$.
\par
By employing the unique decomposition for the case $\mh{M}=\mc{B}^{p}(U,G)$, we introduce firstly in \textbf{Def.\ref{08281540}} 
for the special case of a domain of a chart, and then for the general case in \textbf{Thm.\ref{08281401}} the 
\emph{differential of $G$-valued $p$-times continuously differentiable alternating tensor fields of $M$ defined on $U$} and 
called $G$-differential in the remaining of this Introduction. After having stated that the pushforward via linear and continous 
functionals commutes with the $G$-differential, and with the help of the sequence of results from Rmk.\ref{11281505} to 
Prp.\ref{12062133}, we establish the properties of the $G$-differential in \textbf{Thm.\ref{09211057}}. 
\par
Next having at our disposition the extension of the insertion operator and the $G$-differential we can extend in 
\textbf{Cor.\ref{12011515}} to $H$-valued $p$-times continuously differentiable vector fields the definition of divergence 
associated with a volume form, that we call $H$-divergence, and in \textbf{Thm.\ref{12092034}} and its Corollary \ref{12131724} 
we generalize a well-known form of the divergence to our $H$-valued divergence. 
\par
Now the unique decomposition for the case $\mh{M}=\mf{L}_{c}^{1}(V,\K,\lambda)$, the concept of support as introduced in 
Def.\ref{09161456} and finally the pushforward via continuous and linear functionals permit to establish in 
\textbf{Def.\ref{09281748}} the fundamental definition of weak integral of compactly supported $G$-valued scalarly integrable 
maximal forms. \textbf{Prp.\ref{08291022}} states that if $G$ is quasi-complete, then the above mentioned 
\emph{weak integral belongs to $G$}, while \textbf{Thm.\ref{09171005}} proves that a vector-measure can be associated 
with any scalarly locally integrable $G$-valued form, measure that is $G$-valued in case $G$ is quasi-complete.
This result is the main reason of introducing the definition of scalarly locally integrable $G$-valued forms. 
Finally the main result of this paper \textbf{Thm.\ref{08281926}} establishes a 
\emph{Stokes theorem type for $G$-valued compactly supported $p$-times continuously differentiable forms} and as a consequence 
we obtain \textbf{Cor.\ref{12022324}} exhibiting a 
\emph{divergence theorem type for compactly supported $H$-valued $p$-times continuously differentiable vector fields}.
\end{introduction}
\section[$G$-valued scalarly locally integrable and differentiable maps]
{$G$-valued scalarly locally integrable maps and $G$-valued $p$-times continuously differentiable maps}
\label{09172012}
\begin{definition}
[$G$-valued $p$-times continuously differentiable maps on $M$]
\label{10021123}
Define $\mc{B}^{0}(M,G)$ to be the $\mc{A}(M)$-module $\mc{C}(M,G_{\R})$. Let $p\in\Z_{>0}\cup\{+\infty\}$. 
Define $\mc{B}^{p}(M,G)$ be the $\mc{A}(M)$-submodule of $\mc{C}(M,G_{\R})$ made of those maps $f:M\to G_{\R}$ such that 
$f\circ\imath_{U}^{M}\circ\phi^{-1}\in\mc{C}^{p}(\phi(U),G_{\R})$, for every chart $(U,\phi)$ of $M$. Define $\mc{B}_{c}^{p}(M,G)$ be 
the $\mc{A}(M)$-submodule of $\mc{B}^{p}(M,G)$ made of those maps in $\mc{B}^{p}(M,G)$ whose support is compact.
\end{definition}
\begin{convention}
Let $p\in\Z_{\geq 0}\cup\{+\infty\}$. We let $\mc{B}^{p}(M)=\mc{B}^{p}(M,\K)$, $\mc{B}_{c}^{p}(M)=\mc{B}_{c}^{p}(M,\K)$, 
$\mc{B}(M,G)=\mc{B}^{\infty}(M,G)$, $\mc{B}_{c}(M,G)=\mc{B}_{c}^{\infty}(M,G)$, $\mc{B}(M)=\mc{B}(M,\K)$, 
$\mc{B}_{c}(M)=\mc{B}_{c}(M,\K)$.
\end{convention}
Since Convention \ref{07071626} $\mc{B}^{p}(M,G)$ and $\mc{B}_{c}^{p}(M,G)$ are also $\K$-linear spaces moreover 
$\mc{A}^{p}(M)=\mc{B}^{p}(M,\R)$ for every $p\in\Z_{\geq 0}\cup\{+\infty\}$.
\begin{remark}
\label{09251444}
A standard argument proves that $f\in\mc{B}^{p}(M,G)$ is equivalent to state that for every $x\in M$ there exists a chart 
$(V,\beta)$ such that $V\ni x$ and $f\circ\imath_{V}^{M}\circ\beta^{-1}\in\mc{C}^{p}(\beta(V),G_{\R})$ 
for every $p\in\Z_{\geq 0}\cup\{+\infty\}$. As a result the usual gluing lemma via a covering of charts extends to 
$\mc{B}^{p}(M,G)$.
\end{remark}
\begin{definition}
\label{09241534}
Let $r\in\Z_{\geq 1}\cup\{+\infty\}$, $(U,\phi)$ be a chart of $M$. Define for all $i\in[1,n]\cap\Z$ 
the map $\partial_{i}^{\phi,G,r}:\mc{B}^{r}(U,G)\to\mc{B}^{r-1}(U,G)$ such that for every $f\in\mc{B}^{r}(U,G)$ we have 
\begin{equation*}
\partial_{i}^{\phi,G,r}
f
\coloneqq
\left(D_{e_{i}}^{\phi(U),G}(f\circ\phi^{-1})\right)\circ\phi.
\end{equation*}
\end{definition}
\begin{convention}
We convein to denote $\partial_{i}^{\phi,G,r}$ simply by $\partial_{i}^{\phi,G}$ whenever it is clear by the context which value $r$ 
is involved.
\end{convention}
\begin{definition}
[Germs of Vector Valued Smooth Maps]
\label{08251143}
Let $p\in M$, define $\mc{B}_{p}^{\infty}(M,G)$ to be the set of equivalence classes of the equivalence relation 
$\equiv_{p}^{M,G}$ in $\mf{c}_{p}^{\infty}(M,G)\coloneqq\bigcup_{\{U\in Op(M)\vert p\in U\}}\mc{B}^{\infty}(U,G)$ 
so defined $f\equiv_{p}^{M,G} g$ iff there exists an open set $V$ of $M$ such that 
$p\in V\subseteq\mr{D}(f)\cap\mr{D}(g)$ and $f\circ\imath_{V}^{\mr{D}(f)}=g\circ\imath_{V}^{\mr{D}(g)}$.
We let $[f]_{p}^{M,G}$ denote the equivalence class of $f$ w.r.t. $\equiv_{p}^{M,G}$,
moreover set $\uprho_{p}^{M,G}:\mc{B}^{\infty}(M,G)\ni f\mapsto[f]_{p}^{M,G}\in\mc{B}_{p}^{\infty}(M,G)$.
Let $\mf{c}_{p}^{\infty}(M)\coloneqq\mf{c}_{p}^{\infty}(M,\R)$.
We define the operations $+:\mc{B}_{p}^{\infty}(M,G)\times\mc{B}_{p}^{\infty}(M,G)\to\mc{B}_{p}^{\infty}(M,G)$ and 
$\cdot:\mf{c}_{p}^{\infty}(M)\times\mc{B}_{p}^{\infty}(M,G)\to\mc{B}_{p}^{\infty}(M,G)$ such that 
for any $f,h\in\mf{c}_{p}^{\infty}(M,G)$ and $g\in\mf{c}_{p}^{\infty}(M)$ said $U$, $W$ and $V$ the domains of $f$, $h$ and $g$ 
respectively, we have $[f]_{p}^{M,G}+[h]_{p}^{M,G}\coloneqq[f\circ\imath_{U\cap W}^{\mr{D}(f)}+h\circ\imath_{U\cap W}^{\mr{D}(h)}]_{p}^{M,G}$
and $[g]_{p}^{M,\K}\cdot[f]_{p}^{M,G}\coloneqq[(g\circ\imath_{U\cap V}^{\mr{D}(g)})\cdot(f\circ\imath_{U\cap V}^{\mr{D}(f)})]_{p}^{M,G}$.
In particular the $\mc{A}(M)$-module structure of $\mc{B}^{\infty}(M,G)$ extends on $\mc{B}_{p}^{\infty}(M,G)$. 
\end{definition}
We let $[g]_{p}$ denote $[g]_{p}^{M,\K}$ and let $[g]_{p}[f]_{p}^{M,G}$ denote $[g]_{p}^{M,\K}\cdot[f]_{p}^{M,G}$.
If $f$ is any map, then we let $f_{\natural}$ be the restriction of $f$ to its range, namely if $f:A\to B$, then 
$f_{\natural}:A\to f(A)$.
\begin{definition}
[Composition of germs]
\label{08251145}
Let $m\in M$ and $n\in G$, set $\mc{B}_{(n,m)}^{\infty}(M,G)$ be the subset of those $g\in\mc{B}^{\infty}(M,G)$ such that 
$m\in g^{-1}(\{n\})$. Next let $R$ be a manifold, define $\circ^{R}$ or simply $\circ$ to be the map 
\begin{equation*}
\circ\in
\prod_{(Q,W)\in\mr{Op}(R)\times\mr{Op}(M)}
\prod_{(p,q)\in W\times Q}
\mc{F}\bigl(\mc{B}_{p}^{\infty}(M,G)\times\mc{B}_{(p,q)}^{\infty}(Q,W),\mc{B}_{q}^{\infty}(R,G)\bigr);
\end{equation*}
such that for all $(Q,W)\in\mr{Op}(R)\times\mr{Op}(M)$, $(p,q)\in W\times Q$, $f\in\mf{c}_{p}^{\infty}(M,G)$ and
$g\in\mc{B}_{(p,q)}^{\infty}(Q,W)$ we have
\begin{equation*}
[f]_{p}^{M,G}\circ g
\coloneqq
[f\circ\imath_{g(g^{-1}(\mr{D}(f)\cap W))}^{\mr{D}(f)}
\circ
(g\circ\imath_{g^{-1}(\mr{D}(f)\cap W)}^{\mr{D}(g)})_{\natural}]_{q}^{R,G}.
\end{equation*}
For every $h\in\mc{B}_{(p,q)}^{\infty}(Q,W)$ define $h^{\dagger}:\mc{B}_{p}^{\infty}(M,G)\to\mc{B}_{q}^{\infty}(R,G)$, 
$[f]_{p}^{M,G}\mapsto [f]_{p}^{M,G}\circ h$.
\end{definition}
Clearly $h^{\dagger}$ is $\K$-linear.
\begin{proposition}
\label{08251004vect}
Let $R$ be a manifold, $(Q,W)\in\mr{Op}(R)\times\mr{Op}(M)$, $(p,q)\in W\times Q$, $f,l\in\mf{c}_{p}^{\infty}(M,G)$ and 
$h\in\mc{B}_{(p,q)}^{\infty}(Q,W)$, then $[f]_{p}^{M,G}=[l]_{p}^{M,G}\Rightarrow [f]_{p}^{M,G}\circ h=[l]_{p}^{M,G}\circ h$, in particular 
Def. \ref{08251145} is well-set.
\end{proposition}
In the following definition by abuse of language we let $\mc{B}_{p}^{\infty}(M,G)$ denote also the $\R$-linear space underlying
the $\K$-linear space $\mc{B}_{p}^{\infty}(M,G)$.
\begin{definition}
[$G$-valued Derivations]
Let $p\in M$. A derivation on $M$ at $p$ and at values in $G$ is a couple $(v,\mf{v})$ where
$v$ is a derivation on $M$ at $p$, $\mf{v}$ is a $\R$-linear map from $\mc{B}_{p}^{\infty}(M,G)$ to $G_{\R}$,
and for all $f\in\mf{c}_{p}^{\infty}(M,G)$ and $g\in\mf{c}_{p}^{\infty}(M)$ we have
\begin{equation*}
\mf{v}([g]_{p}[f]_{p}^{M,G})=v([g]_{p})f(p)+g(p)\mf{v}([f]_{p}^{M,G}).
\end{equation*}
The set of derivations of $M$ at $p$ and at values in $G$ is a linear subspace of the $\R$-linear space
$\mr{Mor}_{\R-\mr{mod}}(\mc{B}_{p}^{\infty}(M,\R),\R)\times\mr{Mor}_{\R-\mr{mod}}(\mc{B}_{p}^{\infty}(M,G),G_{\R})$, 
it is denoted by $T_{p}(M,G)$.
Finally define 
\begin{equation*}
\mh{T}_{p}(M,G)
\coloneqq
\left\{\mf{v}\in\mr{Mor}_{\R-\mr{mod}}(\mc{B}_{p}^{\infty}(M,G),G_{\R})\,\vert\,
(\exists\,v\in\mr{Mor}_{\R-\mr{mod}}(\mc{B}_{p}^{\infty}(M,\R),\R))((v,\mf{v})\in T_{p}(M,G))\right\}.
\end{equation*}
\end{definition}
\begin{definition}
\label{08251123}
Let $W$ be an open set of $\R^{n}$, $p\in W$ and $v\in\R^{n}$. Define $D_{v}^{W,G}\vert_{p}:\mc{B}_{p}^{\infty}(W,G)\to G_{\R}$ such 
that for all $f\in\mf{c}_{p}^{\infty}(W,G)$ we have 
\begin{equation*}
D_{v}^{W,G}\vert_{p}[f]_{p}^{W}\coloneqq D_{v}^{\mr{D}(f),G}\vert_{p}f.
\end{equation*}
\end{definition}
\begin{proposition}
\label{08251307}
Let $a\in\R^{n}$, thus $\{(D_{e_{i}}\vert_{a},D_{e_{i}}^{\R^{n},G}\vert_{a})\}_{i=1}^{n}$ is a basis of $T_{a}(\R^{n},G)$ and 
$(w,\mf{w})=\sum_{i=1}^{n}w([x^{i}]_{a}^{\R^{n}})(D_{e_{i}}\vert_{a},D_{e_{i}}^{\R^{n},G}\vert_{a})$ for any $(w,\mf{w})\in T_{a}(\R^{n},G)$,
in particular $\{D_{e_{i}}^{\R^{n},G}\vert_{a}\}_{i=1}^{n}$ is a basis of $\mh{T}_{a}(\R^{n},G)$. 
\end{proposition}
\begin{proof}
The proof of the first two sentences of the statement runs essentially as the one for scalar-valued maps with the modification 
of taking $[x\Pr_{i}]_{a}^{\R^{n},G}$ with $x\in\complement_{G}\{\ze\}$ instead of $[\Pr_{i}]_{a}^{\R^{n}}$. Next let 
$(w,\mf{w})\in T_{a}(\R^{n},G)$ thus, $\mf{w}=\sum_{i=1}^{n}w([x^{i}]_{a}^{\R^{n}})D_{e_{i}}^{\R^{n},G}\vert_{a}$ by the first sentence of 
the statement. 
Viceversa given a sequence $\{\lambda_{i}\}_{i=1}^{n}$ in $\R$, and letting $\mf{q}=\sum_{i=1}^{n}\lambda_{i}D_{e_{i}}^{\R^{n},G}\vert_{a}$,
we have $\mf{q}\in\mh{T}_{a}(\R^{n},G)$ indeed $(q,\mf{q})\in T_{a}(\R^{n},G)$ by the first sentence of the statement 
where $q=\sum_{i=1}^{n}\lambda_{i}D_{e_{i}}\vert_{a}$.
\end{proof}
Easily we can prove that for any chart $(U,\phi)$ of $M$ and $p\in U$ we have that 
\begin{equation*}
\begin{aligned}
\phi^{\dagger}&\in\mr{Is}_{\K-\mr{mod}}(\mc{B}_{\phi(p)}^{\infty}(\R^{n},G),\mc{B}_{p}^{\infty}(M,G)),
\\
(\phi^{-1})^{\dagger}&\in\mr{Is}_{\K-\mr{mod}}(\mc{B}_{p}^{\infty}(M,G),\mc{B}_{\phi(p)}^{\infty}(\R^{n},G)).
\end{aligned}
\end{equation*}
In particular with the obvious abuse of language we have 
\begin{equation}
\label{08251759}
\bigl((w,\mf{w})\mapsto(w\circ(\phi^{-1})^{\dagger},\mf{w}\circ(\phi^{-1})^{\dagger})\bigr)
\in\mr{Is}_{\K-\mr{mod}}(T_{\phi(p)}(\R^{n},G),T_{p}(M,G)).
\end{equation}
\begin{definition}
\label{03301348vect}
Let $p\in M$, and $(U,\phi)$ a chart of $M$ containing $p$, define for all $i\in[1,n]\cap\Z$ the map 
$\partial_{i}^{\phi,G}\vert_{p}:\mc{B}_{p}^{\infty}(M,G)\to G_{\R}$ such that 
\begin{equation*}
\partial_{i}^{\phi,G}\vert_{p}
\coloneqq
(D_{e_{i}}^{\R^{n},G}\vert_{\phi(p)})\circ(\phi^{-1})^{\dagger}
\end{equation*}
and by abuse of language we let $\partial_{i}^{\phi,G}\vert_{p}:\mc{B}^{\infty}(M,G)\to G_{\R}$ 
denote the map $(\partial_{i}^{\phi,G}\vert_{p})\circ\uprho_{p}^{M,G}$.
\end{definition}
\begin{proposition}
[Basis of $T_{p}(M,G)$]
\label{08251803}
Let $p\in M$ and $(U,\phi)$ be a chart of $M$ such that $p\in U$. 
Thus $\{(\partial_{i}^{\phi}\vert_{p},\partial_{i}^{\phi,G}\vert_{p})\}_{i=1}^{n}$ is a basis of $T_{p}(M,G)$ and for all 
$(v,\mf{v})\in T_{p}(M,G)$ we have 
\begin{equation*}
(v,\mf{v})=\sum_{i=1}^{n}v([x_{i}^{\phi}]_{p}^{M})(\partial_{i}^{\phi}\vert_{p},\partial_{i}^{\phi,G}\vert_{p}),
\end{equation*}
where for all $f\in\mf{c}_{p}^{\infty}(M,G)$
\begin{equation*}
\partial_{i}^{\phi,G}\vert_{p}[f]_{p}^{M,G}
=
D_{e_{i}}^{\phi(\mr{D}(f)\cap U),G}\vert_{\phi(p)}(f\circ\imath_{\mr{D}(f)\cap U}^{\mr{D}(f)}
\circ(\phi^{-1}\up_{\phi(\mr{D}(f)\cap U)})_{\natural}).
\end{equation*}
In particular $\{\partial_{i}^{\phi,G}\vert_{p}\}_{i=1}^{n}$ is a basis of $\mh{T}_{p}(M,G)$ and 
$\mr{dim}\,T_{p}(M,G)=\mr{dim}\,\mh{T}_{p}(M,G)=n$.
\end{proposition}
\begin{proof}
Since Prp. \ref{08251307} and \eqref{08251759}.
\end{proof}
\begin{definition}
[\textbf{$G$-valued scalarly locally integrable maps on $M$}]
\label{10111122}
Define $\mf{L}_{loc}^{1}(M,G,\lambda)$ to be the $\K$-linear subspace of $G^{M}$ of those maps $f:M\to G$ such that 
$f\circ\imath_{U}^{M}\circ\phi^{-1}$ is scalarly locally $\lambda_{\phi(U)}$-integrable, for every chart $(U,\phi)$ of $M$. 
Let $\mf{L}_{c,loc}^{1}(M,G,\lambda)$ be the $\K$-linear subspace of $\mf{L}_{loc}^{1}(M,G,\lambda)$ made of maps with compact 
support. 
\end{definition}
\begin{definition}
[\textbf{$G$-valued measurable scalarly locally integrable maps on $M$}]
\label{07152211}
Define $\mh{L}_{loc}^{1}(M,G,\lambda)$ to be the $\K$-linear subspace of $G^{M}$ of those maps $f:M\to G$ 
such that for every chart $(U,\phi)$ of $M$ we have that $f\circ\imath_{U}^{M}\circ\phi^{-1}$ is $\lambda_{\phi(U)}$-measurable 
and $\int^{\ast}|\uppsi\circ f\circ\imath_{U}^{M}\circ\phi^{-1}|\chi_{K}d\,\lambda_{\phi(U)}<\infty$
for every $K$ compact of $\phi(U)$ and $\uppsi\in G^{\prime}$.
Let $\mh{L}_{c,loc}^{1}(M,G,\lambda)$ be the $\K$-linear subspace of $\mh{L}_{loc}^{1}(M,G,\lambda)$ made of maps with compact 
support. 
\end{definition}
Note that if $(U,\phi)$ is a chart of $M$, $\uppsi\in G^{\prime}$ and $f\circ\imath_{U}^{M}\circ\phi^{-1}$ is 
$\lambda_{\phi(U)}$-measurable, then $\uppsi\circ f\circ\imath_{U}^{M}\circ\phi^{-1}$ is $\lambda_{\phi(U)}$-measurable since 
\cite[Thm.1 IV.63]{IntBourb}, therefore by definition of locally integrability we deduce that 
$\mh{L}_{loc}^{1}(M,G,\lambda)$ is the $\K$-linear subspace of those 
$f\in\mf{L}_{loc}^{1}(M,G,\lambda)$ such that $f\circ\imath_{U}^{M}\circ\phi^{-1}$ is $\lambda_{\phi(U)}$-measurable, 
for every chart $(U,\phi)$ of $M$. 
\begin{definition}
[\textbf{$G$-valued scalarly integrable maps on $M$}]
\label{09281553}
Define $\mf{L}^{1}(M,G,\lambda)$ to be the $\K$-linear subspace of $G^{M}$ of those maps $f:M\to G$ such that 
$f\circ\imath_{U}^{M}\circ\phi^{-1}$ is scalarly $\lambda_{\phi(U)}$-integrable, for every chart $(U,\phi)$ of $M$. 
Let $\mf{L}_{c}^{1}(M,G,\lambda)$ be the $\K$-linear subspace of $\mf{L}^{1}(M,G,\lambda)$ made of maps with compact 
support.
\end{definition}
\begin{definition}
[\textbf{$G$-valued measurable scalarly integrable maps on $M$}]
\label{07152212}
Define $\mh{L}^{1}(M,G,\lambda)$ to be the $\K$-linear subspace of $G^{M}$ of those maps $f:M\to G$ such that 
for every chart $(U,\phi)$ of $M$ we have that $f\circ\imath_{U}^{M}\circ\phi^{-1}$ is $\lambda_{\phi(U)}$-measurable and 
$\int^{\ast}|\uppsi\circ f\circ\imath_{U}^{M}\circ\phi^{-1}|d\,\lambda_{\phi(U)}<\infty$ for every $\uppsi\in G^{\prime}$.
Let $\mh{L}_{c}^{1}(M,G,\lambda)$ be the $\K$-linear subspace of $\mh{L}^{1}(M,G,\lambda)$ made of maps with compact 
support.
\end{definition}
As above said if $(U,\phi)$ is a chart of $M$, $\uppsi\in G^{\prime}$ and $f\circ\imath_{U}^{M}\circ\phi^{-1}$ is 
$\lambda_{\phi(U)}$-measurable, then $\uppsi\circ f\circ\imath_{U}^{M}\circ\phi^{-1}$ is $\lambda_{\phi(U)}$-measurable since 
\cite[Thm.1 IV.63]{IntBourb}, therefore by \cite[IV.71 Thm.5]{IntBourb} we deduce that 
$\mh{L}^{1}(M,G,\lambda)$ is the $\K$-linear subspace of those $f\in\mf{L}^{1}(M,G,\lambda)$ such that 
$f\circ\imath_{U}^{M}\circ\phi^{-1}$ is $\lambda_{\phi(U)}$-measurable, for every chart $(U,\phi)$ of $M$. 
The spaces $\mh{L}_{loc}^{1}(M,G,\lambda)$ and $\mh{L}^{1}(M,G,\lambda)$ are required in Thm.\ref{09251928}.
\begin{remark}
\label{11121558}
Clearly $\mf{L}_{loc}^{1}(M,\K,\lambda)=\mh{L}_{loc}^{1}(M,\K,\lambda)$ and $\mf{L}^{1}(M,\K,\lambda)=\mh{L}^{1}(M,\K,\lambda)$. 
Next let $O$ be an open set of $\R^{n}$. Then since $(O,\mr{Id}_{O})$ is evidently a chart of $O$, we conclude that 
$\mf{L}_{loc}^{1}(O,G,\lambda)$ equals the space of scalarly locally $\lambda_{O}$-integrable $G$-valued maps defined on $O$, while
$\mf{L}^{1}(O,G,\lambda)$ equals the space of scalarly $\lambda_{O}$-integrable $G$-valued maps defined on $O$. Therefore 
$\mf{L}_{loc}^{1}(O,\K,\lambda)=\mh{L}_{loc}^{1}(O,\K,\lambda)$ equals the space of locally $\lambda_{O}$-integrable $\K$-valued 
maps defined on $O$, while $\mf{L}^{1}(O,\K,\lambda)=\mh{L}^{1}(O,\K,\lambda)$ equals the space of $\lambda_{O}$-integrable 
$\K$-valued maps defined on $O$.
\end{remark}
Since for every locally compact space $X$ and any measure $\mu$ on $X$ we have $\mc{H}(X,\K)\subset\mf{L}^{\infty}(X,\mu)$ 
we obtain by the H\"{o}lder inequality that $\mf{L}^{1}(X,\mu)\subseteq\mf{L}_{loc}^{1}(X,\mu)$ and then 
\begin{equation}
\label{10111855}
\begin{aligned}
\mf{L}^{1}(M,G,\lambda)&\subseteq\mf{L}_{loc}^{1}(M,G,\lambda),
\\
\mh{L}^{1}(M,G,\lambda)&\subseteq\mh{L}_{loc}^{1}(M,G,\lambda).
\end{aligned}
\end{equation}
\begin{lemma}
\label{10111124}
We have 
\begin{equation*}
\begin{aligned}
(\forall g\in\mc{H}(M))(\forall f\in\mf{L}_{loc}^{1}(M,G,\lambda))(gf&\in\mf{L}_{c}^{1}(M,G,\lambda)),
\\
(\forall g\in\mc{H}(M))(\forall f\in\mh{L}_{loc}^{1}(M,G,\lambda))(gf&\in\mh{L}_{c}^{1}(M,G,\lambda)).
\end{aligned}
\end{equation*}
\end{lemma}
\begin{proof}
The second sentence of the statement follows directly by the first one and since the product of measurable maps is measurable
see for instance \cite[IV.64 Cor.5]{IntBourb}, therefore in what follows we prove the first sentence of the statement.
To this end let $(U,\phi)$ be a chart of $M$, $g\in\mc{H}(M)$, $f\in\mf{L}_{loc}^{1}(M,G,\lambda)$ and $\psi\in\mc{L}(G,\K)$ 
moreover let $h=g\circ\imath_{U}^{M}\circ\phi^{-1}$, $j=f\circ\imath_{U}^{M}\circ\phi^{-1}$ and $l=\psi\circ j$. By construction $l$ 
is $\lambda_{\phi(U)}$-locally integrable so $l$ is $\lambda_{\phi(U)}$-measurable by definition.
Next $h$ is continuous being $g$ so, in particular it is $\lambda_{\phi(U)}$-measurable therefore $hl$ is 
$\lambda_{\phi(U)}$-measurable since a product of measurable maps is measurable \cite[IV.64 Cor.5]{IntBourb}. 
Next since we have 
\begin{equation}
\label{10111212}
\mr{supp}(h)
=
\phi(\mr{supp}(g\circ\imath_{U}^{M})),
\end{equation}
while $\mr{supp}(g\circ\imath_{U}^{M})\subseteq\mr{supp}(g)$. Thus, 
$\mr{supp}(g\circ\imath_{U}^{M})$ is a closed set of the topological space $\mr{supp}(g)$ 
but then $\mr{supp}(g\circ\imath_{U}^{M})$ is compact since $\mr{supp}(g)$ is compact by hypothesis.
Next $\phi\in\mc{C}(U,\phi(U))$ then $\phi(\mr{supp}(g\circ\imath_{U}^{M}))$ is compact as topological subspace
of $\phi(U)$ therefore, $\mr{supp}(h)$ is compact by \eqref{10111212} and then there exists a compact $K\subseteq\phi(U)$, 
we can take $K=\mr{supp}(h)$, such that $h=h\chi_{K}^{\phi(U)}$. Henceforth
\begin{equation*}
|hl|=|h\chi_{K}^{\phi(U)}l|=|h\chi_{K}^{\phi(U)}\chi_{K}^{\phi(U)}l|
=|h\chi_{K}^{\phi(U)}|\,|\chi_{K}^{\phi(U)}l|
\leq\sup_{x\in K}|h(x)|\,|\chi_{K}^{\phi(U)}l|,
\end{equation*}
so, we have 
\begin{equation}
\label{10111229}
\int^{\ast}|hl|\,d\lambda_{\phi(U)}
\leq
\sup_{x\in K}|h(x)|
\int^{\ast}|\chi_{K}^{\phi(U)}l|
\,d\lambda_{\phi(U)}
<\infty;
\end{equation}
where the last inequality follows since $K$ is compact and since $l$ is by construction locally $\lambda_{\phi(U)}$-integrable.
The already proven fact that $hl$ is $\lambda_{\phi(U)}$-measurable and \eqref{10111229} yield that $hl$ is 
$\lambda_{\phi(U)}$-integrable but $\psi\circ(hj)=hl$ then $hj$ is scalarly $\lambda_{\phi(U)}$-integrable. Next by the above proven 
fact that the support $K$ of $h$ is compact, we deduce that the support of $hj$ is a subset of $K$, 
so it is a closed set of the topological space $K$ which is compact, thus the support of $hj$ is compact being a closed set of a 
compact and the statement follows.
\end{proof}
\begin{remark}
\label{09251444int}
A standard argument proves that $f\in\mf{L}_{loc}^{1}(M,G,\lambda)$ is equivalent to state that for every $x\in M$ there exists
a chart $(V,\beta)$ such that $V\ni x$ and $f\circ\imath_{V}^{M}\circ\beta^{-1}$ is
scalarly locally $\lambda_{\beta(V)}$-integrable.
Similarly $h\in\mf{L}^{1}(M,G,\lambda)$ is equivalent to state that for every $x\in M$ there exists
a chart $(V,\beta)$ such that $V\ni x$ and $h\circ\imath_{V}^{M}\circ\beta^{-1}$ is
scalarly $\lambda_{\beta(V)}$-integrable. Therefore we deduce that $f\in\mh{L}_{loc}^{1}(M,G,\lambda)$ is equivalent to state 
that for every $x\in M$ there exists a chart $(V,\beta)$ such that $V\ni x$ and $f\circ\imath_{V}^{M}\circ\beta^{-1}$ 
is $\lambda_{\beta(V)}$-measurable and scalarly locally $\lambda_{\beta(V)}$-integrable.
Furthermore $h\in\mh{L}^{1}(M,G,\lambda)$ is equivalent to state that for every $x\in M$ there exists
a chart $(V,\beta)$ such that $V\ni x$ and $h\circ\imath_{V}^{M}\circ\beta^{-1}$ is $\lambda_{\beta(V)}$-measurable and 
scalarly $\lambda_{\beta(V)}$-integrable. As a result the usual gluing lemma via a covering of charts extends to 
$\mf{L}_{loc}^{1}(M,G,\lambda)$, $\mf{L}^{1}(M,G,\lambda)$, $\mh{L}_{loc}^{1}(M,G,\lambda)$ and $\mh{L}^{1}(M,G,\lambda)$.
\end{remark}
\begin{remark}
Let $X$ be a locally compact space and $K$ be a compact set of $X$. 
The existence of a continuous partition of the unity subordinated to any finite open covering of any compact space, 
and the one-point compactification of any locally compact space, yield, for every finite covering of $K$ made by open sets of 
$X$, the existence of a real valued continuous map on $X$ whose support is contained in the union of the open cover and whose 
restriction to $K$ equals the constant $1$ \cite[III.4 Lemma 1]{IntBourb}.
Next there exists a fundamental set of neighbourhoods of $K$ made of compact sets \cite[I.65]{BourGT}.
Here a neighbourhood of a set $A$ of a topological space $Y$ is any set that contains an open set of $Y$ containing $A$; while a 
fundamental set of neighbourhoods of $A$ is a base of the filter of neighbourhoods of $A$.
Now $X$ is a neighbourhood of $K$ thus there exists a compact $Q$ and an open set $V$ of $X$ such that $Q\supset V\supset K$,
then by applying \cite[III.4 Lemma 1]{IntBourb} to the open cover $V$ of $K$ we deduce that there exists a real valued 
continuous map on $X$ whose restriction to $K$ equals the constant map $1$ and whose support is contained in $V$ so in $Q$ and 
so it is compact being closed in the compact $Q$. In other words by letting for any $A\subset X$ and $\lambda\in\R$ the map 
$c_{\lambda}^{A}:A\to\R$ be $x\mapsto\lambda$, we have 
\begin{equation}
\label{11121710}
(\exists g\in\mc{H}(X))(g\circ\imath_{K}^{X}=c_{1}^{K}).
\end{equation}
\end{remark}
\begin{remark}
\label{11241612} 
Since $\C_{\R}=\R^{2}$, where $\Pr_{1}$ is the real part and $\Pr_{2}$ is the imaginary part,  
we have that if $O$ is an open set of $\R^{n}$, $p\in\Z_{\geq 0}\cup\{+\infty\}$, then 
$f\in\mc{C}^{p}(O,\C_{\R})\Leftrightarrow(\forall k\in\{1,2\})((\Pr_{k}\circ f)\in\mc{C}^{p}(O,\R))$, as a result 
\begin{equation*}
\K=\C\Rightarrow\bigl(f\in\mc{B}^{p}(M)\Leftrightarrow(\forall k\in\{1,2\})((\Pr_{k}\circ f)\in\mc{A}^{p}(M))\bigr).
\end{equation*}
In particular if $E$ is a $\K$-linear space and $\mh{R}$ is a $\K$-linear subspace of $E^{M}$, then 
$\mh{R}$ is a $\mc{A}^{p}(M)$-module if and only if $\mh{R}$ is a $\mc{B}^{p}(M)$-module 
where the law of action of $\mc{A}^{p}(M)$ and $\mc{B}^{p}(M)$ on $\mh{R}$ is the pointwise multiplication.
\end{remark}
\begin{corollary} 
\label{10111859}
If $p\in\Z_{\geq 0}\cup\{+\infty\}$, then 
\begin{enumerate}
\item
$\mf{L}_{loc}^{1}(M,G,\lambda)$ 
is a $\mc{A}^{p}(M)$-module and then also a $\mc{B}^{p}(M)$-module,
\label{10111859a}
\item
$\mf{L}_{c}^{1}(M,G,\lambda)$ 
is a $\mc{A}^{p}(M)$-module and then also a $\mc{B}^{p}(M)$-module,
\label{10111859c}
\item
$\mh{L}_{loc}^{1}(M,G,\lambda)$ 
is a $\mc{A}^{p}(M)$-module and then also a $\mc{B}^{p}(M)$-module,
\label{10111859e}
\item
$\mh{L}_{c}^{1}(M,G,\lambda)$ 
is a $\mc{A}^{p}(M)$-module and then also a $\mc{B}^{p}(M)$-module.
\label{10111859f}
\end{enumerate}
\end{corollary}
\begin{proof}
Statements \eqref{10111859e} and \eqref{10111859f} follow by statements \eqref{10111859a} and \eqref{10111859c} respectively;
by the fact that if $g\in\mc{A}^{p}(M)$ and $(U,\phi)$ is a chart of 
$M$, then $g\circ\imath_{U}^{M}\circ\phi^{-1}$ is $\lambda_{\phi(U)}$-measurable being continuous; finally by the fact that the 
product of measurable maps is a measurable map see for instance \cite[IV.64 Cor.5]{IntBourb}.
Therefore let us prove statements \eqref{10111859a} and \eqref{10111859c}. 
Since Rmk.\ref{11241612} it is sufficient to prove the statements for the case $\mc{A}^{p}(M)$. To this end let 
$f\in\mf{L}_{loc}^{1}(M,G,\lambda)$, $h\in\mc{A}^{p}(M)$, $\psi\in G^{\prime}$ and $(V,\phi)$ be a chart of $M$, then 
$a=\psi\circ f\circ\imath_{V}^{M}\circ\phi^{-1}\in\mf{L}_{loc}^{1}(\phi(V),\K,\lambda_{\phi(V)})$ by Rmk.\ref{11121558}, while 
by definition $b=h\circ\imath_{V}^{M}\circ\phi^{-1}\in\mc{C}^{p}(\phi(V),\R)$ in particular for every $g\in\mc{H}(\phi(V))$ we 
have $gb\in\mc{H}(\phi(V))$ and therefore $gba\in\mf{L}^{1}(\phi(V),\K,\lambda_{\phi(V)})$ since Lemma \ref{10111124} namely $gba$ 
is $\lambda_{\phi(V)}$-integrable by Rmk.\ref{11121558}, which means that $ba$ is locally $\lambda_{\phi(V)}$-integrable and 
st.\eqref{10111859a} follows. Next let $f\in\mf{L}_{c}^{1}(M,G,\lambda)$ and let $l\in\mc{H}(M)$ be such that $l$ equals $1$ on 
$\mr{supp}(f)$ which exists since $M$ is locally compact and since \eqref{11121710}, then $f=lf$ therefore for any 
$h\in\mc{A}^{p}(M)$ we have $hf=hlf$, but $hl\in\mc{H}(M)$ since $\mc{A}^{p}(M)\subset\mc{C}(M,\R)$, then st.\eqref{10111859c} 
follows since \eqref{10111855} and Lemma \ref{10111124}.
\end{proof}
$\mc{C}^{p}(O,\K)\cdot\mc{H}(O)\subseteq\mc{H}(O)$, while $\mc{H}(O)\subset\mf{L}^{1}(O,\R,\lambda_{O})$, for every open set $O$
of $\R^{n}$ and $p\in\Z_{\geq 0}\cup\{+\infty\}$, in addition if $U$ is an open set of $M$, $g\in\mc{H}(U,G)\cup\mc{B}^{p}(U,G)$ 
and $(V,\phi)$ is a chart of $U$, then $g\circ\imath_{V}^{U}\circ\phi^{-1}$ is $\lambda_{\phi(V)}$-measurable being $g$ continuous, 
then $\mc{H}(U,G)\subset\mh{L}_{c}^{1}(U,G,\lambda)$ and $\mc{B}^{p}(U,G)\subset\mh{L}_{loc}^{1}(U,G,\lambda)$ so, 
by \eqref{10111855} and Cor.\ref{10111859} we obtain the following inclusions of $\mc{A}(U)$-modules
\begin{equation}
\label{11292215}
\begin{aligned}
\mc{B}_{c}^{p}(U,G)
&\subset
\mc{B}^{p}(U,G)
\subset
\mh{L}_{loc}^{1}(U,G,\lambda)
\subseteq
\mf{L}_{loc}^{1}(U,G,\lambda);
\\
\mc{B}_{c}^{p}(U,G)
&\subseteq
\mc{H}(U,G)
\subset
\mh{L}_{c}^{1}(U,G,\lambda)
\subseteq
\mf{L}_{c}^{1}(U,G,\lambda)
\subset
\mf{L}_{c,loc}^{1}(U,G,\lambda)
\subset
\mf{L}_{loc}^{1}(U,G,\lambda),
\end{aligned}
\end{equation}
where clearly $\mc{B}_{c}^{0}(U,G)=\mc{H}(U,G)$. 
\section{The $\mc{A}(U)$-modules 
$\mh{M}\otimes_{\mc{A}(U)}\mf{T}_{\bullet}^{\bullet}(U,M)$ and $\mf{T}_{\bullet}^{\bullet}(\Gamma(U,M),\mh{M})$.}
\label{09251630}
\begin{assumption}
For the remaining of the paper we let $U$ be an open set of $M$, and let $\mh{M}$ and $\mh{N}$ be  
$\mc{A}(U)$-modules. Furthermore let $E\in\{G,H,G_{1},H_{1},\K\}$ and let $\mh{P}$ be a $\K$-linear subspace of $E^{U}$ such that 
$\mh{P}$ is also a $\mc{A}(U)$-module where the law of action of $\mc{A}(U)$ on $\mh{P}$ is the pointwise multiplication, 
while let $\mh{P}_{c}$ be the $\mc{A}(U)$-submodule of $\mh{P}$ made of compactly supported maps. 
Finally let $\mc{V}$ be a smooth vector bundle over $U$, let $\Gamma(\mc{V})$ be the $\mc{A}(U)$-module of smooth sections of 
$\mc{V}$ and let $\Gamma_{c}(\mc{V})$ be the $\mc{A}(U)$-submodule of $\Gamma(\mc{V})$ made of compactly supported smooth 
sections of $\mc{V}$.
\end{assumption}
\begin{definition}
For every $(r,s)\in\Z_{\geq 0}\times\Z_{\geq 0}$ define 
\begin{equation*}
\Pr_{(r,s)}^{\mf{T}}:\mf{T}_{\bullet}^{\bullet}(U,M)\to\mf{T}_{s}^{r}(U,M),
\end{equation*}
to be the restriction at the domain of the projector map, and define 
\begin{equation*}
\imath_{(r,s)}^{\mf{T}}:\mf{T}_{s}^{r}(U,M)\to\mf{T}_{\bullet}^{\bullet}(U,M),
\end{equation*}
to be the restriction at the codomain of the injection map.
\end{definition}
\begin{definition}
Define the following $\mc{A}(U)$-submodule of $\mf{T}_{\bullet}^{\bullet}(U,M)$
\begin{equation*}
\mf{T}_{\bullet}^{\bullet}(U,M)_{c}
\coloneqq
\bigl\{\zeta\in\mf{T}_{\bullet}^{\bullet}(U,M)\,\vert\,\mr{supp}(\mf{t}_{\R}(\zeta))\in\mr{Cmp}(M)\bigr\},
\end{equation*}
moreover define $\mf{T}_{s}^{r}(U,M)_{c}\coloneqq\Pr_{(r,s)}^{\mf{T}}(\mf{T}_{\bullet}^{\bullet}(U,M)_{c})$.
Furthermore define the following $\mc{A}(U)$-submodule of $\mr{Alt}^{\bullet}(U,M)$
\begin{equation*}
\mr{Alt}^{\bullet}(U,M)_{c}\coloneqq\mf{T}_{\bullet}^{0}(U,M)_{c}\cap\mr{Alt}^{\bullet}(U,M),
\end{equation*}
and for any $k\in[0,n]\cap\Z$ define $\mr{Alt}^{k}(U,M)_{c}\coloneqq\Pr_{(0,k)}^{\mf{T}}(\mr{Alt}^{\bullet}(U,M)_{c})$.
\end{definition}
Next $\mf{a}(U)$ is a $\mc{A}(U)$-module being a $\mc{A}(U)$-ideal so, for every $r,s\in\Z_{\geq 0}$  we have that
$\mf{T}_{s}^{r}(\Gamma(U,TM),\mf{a}(U))$ is a $\mc{A}(U)$-module that is trivially isomorphic in the category of 
$\mc{A}(U)$-modules to the $\mc{A}(U)$-submodule of $\mf{T}_{s}^{r}(U,M)$ made of those $T$ such that $T(x)\in\mf{a}(U)$ 
for every $x\in\prod[\Gamma(U,TM),r,s]$. Therefore, we provide the following 
\begin{convention}
Often and without any further reference we shall identify the above mentioned two isomorphic modules.
\end{convention}
\begin{proposition}
\label{11140741}
$\mf{T}_{\bullet}^{\bullet}(U,M)_{c}\subseteq\mf{T}_{\bullet}^{\bullet}(\Gamma(U,TM),\mf{a}(U))$.
\end{proposition}
\begin{proof}
Let $f\in\Gamma(U,\mf{T}_{s}^{r}(M))$ and $x\in\prod[\Gamma(U,TM),r,s]$, then by construction of $\mf{r}_{\R}$ we deduce that
$\mr{supp}(\mf{r}_{\R}(f)(x))\subseteq\mr{supp}(f)$ so, the statement follows since 
$\mf{r}_{\R}\circ\mf{t}_{\R}=\mr{Id}_{\mf{T}_{\bullet}^{\bullet}(U,M)}$ and since a closed in a compact is compact.
\end{proof}
\begin{lemma}
\label{11140700}
We have 
\begin{equation*}
\begin{aligned}
\mh{P}_{c}\otimes_{\mc{A}(U)}\mf{T}_{\bullet}^{\bullet}(\Gamma(U,TM),\mf{a}(U))
&=
\mh{P}_{c}\otimes_{\mc{A}(U)}\mf{T}_{\bullet}^{\bullet}(U,M)
\\
&=
\mh{P}_{c}\otimes_{\mc{A}(U)}\mf{T}_{\bullet}^{\bullet}(U,M)_{c}
\\
&=
\mh{P}\otimes_{\mc{A}(U)}\mf{T}_{\bullet}^{\bullet}(U,M)_{c},
\end{aligned}
\end{equation*}
and for every $r,s\in\Z_{\geq 0}$
\begin{equation*}
\begin{aligned}
\mh{P}_{c}\otimes_{\mc{A}(U)}\mf{T}_{s}^{r}(\Gamma(U,TM),\mf{a}(U))
&=
\mh{P}_{c}\otimes_{\mc{A}(U)}\mf{T}_{s}^{r}(U,M)
\\
&=
\mh{P}_{c}\otimes_{\mc{A}(U)}\mf{T}_{s}^{r}(U,M)_{c}
\\
&=
\mh{P}\otimes_{\mc{A}(U)}\mf{T}_{s}^{r}(U,M)_{c},
\end{aligned}
\end{equation*}
while 
\begin{equation*}
\begin{aligned}
\mh{P}_{c}\otimes_{\mc{A}(U)}\mr{Alt}^{\bullet}(\Gamma(U,TM),\mf{a}(U))
&=
\mh{P}_{c}\otimes_{\mc{A}(U)}\mr{Alt}^{\bullet}(U,M)
\\
&=
\mh{P}_{c}\otimes_{\mc{A}(U)}\mr{Alt}^{\bullet}(U,M)_{c}
\\
&=
\mh{P}\otimes_{\mc{A}(U)}\mr{Alt}^{\bullet}(U,M)_{c},
\end{aligned}
\end{equation*}
and for every $k\in[0,n]\cap\Z$ 
\begin{equation*}
\begin{aligned}
\mh{P}_{c}\otimes_{\mc{A}(U)}\mr{Alt}^{k}(\Gamma(U,TM),\mf{a}(U))
&=
\mh{P}_{c}\otimes_{\mc{A}(U)}\mr{Alt}^{k}(U,M)
\\
&=
\mh{P}_{c}\otimes_{\mc{A}(U)}\mr{Alt}^{k}(U,M)_{c}
\\
&=
\mh{P}\otimes_{\mc{A}(U)}\mr{Alt}^{k}(U,M)_{c},
\end{aligned}
\end{equation*}
and 
\begin{equation*}
\begin{aligned}
\mh{P}_{c}\otimes_{\mc{A}(U)}\Omega^{k}(U,M)
&=
\mh{P}_{c}\otimes_{\mc{A}(U)}\Omega^{k}(U,M)_{c}
\\
&=
\mh{P}\otimes_{\mc{A}(U)}\Omega^{k}(U,M)_{c}.
\end{aligned}
\end{equation*}
\end{lemma}
\begin{proof}
We shall prove the 1th sequence of three equalities, the remaining ones follow by a similar proof.
Let $g\in\mh{P}_{c}$, $T\in\mf{T}_{\bullet}^{\bullet}(U,M)$ and $h\in\mc{H}(U)$ such that $h$ equals $1$ on $\mr{supp}(g)$, 
which exists since $U$ is locally compact and by \eqref{11121710}, then $g=hg$, so $g\otimes T=(hg)\otimes T=g\otimes(hT)$ but
$\mf{t}_{\R}(hT)=h\mf{t}_{\R}(T)$ so $\mr{supp}(\mf{t}_{\R}(hT))\in\mr{Cmp}(U)$ then, 
\begin{equation*}
\mh{P}_{c}\otimes_{\mc{A}(U)}\mf{T}_{\bullet}^{\bullet}(U,M)
\subseteq
\mh{P}_{c}\otimes_{\mc{A}(U)}\mf{T}_{\bullet}^{\bullet}(U,M)_{c}
\subseteq
\mh{P}\otimes_{\mc{A}(U)}\mf{T}_{\bullet}^{\bullet}(U,M)_{c}.
\end{equation*}
Next let $f\in\mh{P}$, $S\in\mf{T}_{\bullet}^{\bullet}(U,M)_{c}$ and $l\in\mc{H}(U)$ such that $l$ equals $1$ on $\mr{supp}(S)$, 
then $S=lS$, so $f\otimes S=f\otimes(lS)=lf\otimes S$ 
\begin{equation*}
\mh{P}\otimes_{\mc{A}(U)}\mf{T}_{\bullet}^{\bullet}(U,M)_{c}
\subseteq
\mh{P}_{c}\otimes_{\mc{A}(U)}\mf{T}_{\bullet}^{\bullet}(U,M),
\end{equation*}
so, the 2th and 3th equalities in the statement follow. The 1th equality follows by the 2th one and since Prp.\ref{11140741}.
\end{proof}
Recall that by \cite[II.56 $n^{\circ}5$]{BourA1} the tensor product w.r.t. a commutative ring $A$ of $A$-modules admits
a $A$-module structure, thus we can give the following 
\begin{definition}
Define the $\mc{A}(U)$-module
\begin{equation*}
\mf{Q}(\mh{N};\mh{M})\coloneqq\mh{M}\otimes_{\mc{A}(U)}\mh{N}.
\end{equation*}
\end{definition}
\begin{definition}
[\textbf{$\mh{M}$-valued tensor fields of $M$ defined on $U$}]
\label{11121824}
Let $r,s\in\Z_{\geq 0}$, define the $\mc{A}(U)$-modules 
\begin{equation*}
\begin{aligned}
\mf{Q}_{s}^{r}(U,M;\mh{M})&\coloneqq\mf{Q}(\mf{T}_{s}^{r}(U,M),\mh{M})
\\
\mf{Q}_{s}^{r}(U,M;\mh{M})^{c}&\coloneqq\mf{Q}(\mf{T}_{s}^{r}(\Gamma(U,TM),\mf{a}(U)),\mh{M})
\\
\mf{Q}_{s}^{r}(U,M;\mh{M})_{c}&\coloneqq\mf{Q}(\mf{T}_{s}^{r}(U,M)_{c},\mh{M})
\\
\mf{Q}_{s}^{r}(U,M;\mh{P})_{[c]}&\coloneqq\mf{Q}(\mf{T}_{s}^{r}(U,M),\mh{P}_{c})
\\
\mr{Q}_{s}^{r}(U,M;\mh{M})&\coloneqq\mf{T}_{s}^{r}(\Gamma(U,TM),\mh{M})
\\
\Gamma(\mc{V};\mh{M})&\coloneqq\mf{Q}(\Gamma(\mc{V}),\mh{M})
\\
\Gamma(\mc{V};\mh{M})_{c}&\coloneqq\mf{Q}(\Gamma_{c}(\mc{V}),\mh{M}).
\end{aligned}
\end{equation*}
Furthermore define the $\mc{A}(U)$-modules
\begin{equation*}
\begin{aligned}
\mf{Q}_{\bullet}^{\bullet}(U,M;\mh{M})&\coloneqq
\bigoplus_{(t,j)\in\Z_{\geq 0}\times\Z_{\geq 0}}\mf{Q}_{j}^{t}(U,M;\mh{M})
\\
\mf{Q}_{\bullet}^{\bullet}(U,M;\mh{M})^{c}&\coloneqq\bigoplus_{(t,j)\in\Z_{\geq 0}\times\Z_{\geq 0}}\mf{Q}_{j}^{t}(U,M;\mh{M})^{c}
\\
\mf{Q}_{\bullet}^{\bullet}(U,M;\mh{M})_{c}&\coloneqq\bigoplus_{(t,j)\in\Z_{\geq 0}\times\Z_{\geq 0}}\mf{Q}_{j}^{t}(U,M;\mh{M})_{c}
\\
\mf{Q}_{\bullet}^{\bullet}(U,M;\mh{P})_{[c]}&\coloneqq\bigoplus_{(t,j)\in\Z_{\geq 0}\times\Z_{\geq 0}}\mf{Q}_{j}^{t}(U,M;\mh{P})_{[c]}
\\
\mr{Q}_{\bullet}^{\bullet}(U,M;\mh{M})&\coloneqq\mf{T}_{\bullet}^{\bullet}(\Gamma(U,TM),\mh{M}).
\end{aligned}
\end{equation*}
We call the elements of $\mf{Q}_{\bullet}^{\bullet}(U,M;\mh{M})$ $\mh{M}$-valued tensor fields of $M$ defined on $U$. 
\end{definition}  
\begin{definition}
[\textbf{$\mh{M}$-valued alternating tensors of $M$ defined on $U$}]
\label{11121824alt}
Let $k\in[0,n]\cap\Z$, define the $\mc{A}(U)$-modules 
\begin{equation*}
\begin{aligned}
\mr{Alt}^{k}(U,M;\mh{M})&\coloneqq\mf{Q}(\mr{Alt}^{k}(U,M),\mh{M})
\\
\mr{Alt}^{k}(U,M;\mh{M})^{c}&\coloneqq\mf{Q}(\mr{Alt}^{k}(\Gamma(U,TM),\mf{a}(U)),\mh{M})
\\
\mr{Alt}^{k}(U,M;\mh{M})_{c}&\coloneqq\mf{Q}(\mr{Alt}^{k}(U,M)_{c},\mh{M})
\\
\mr{Alt}^{k}(U,M;\mh{P})_{[c]}&\coloneqq\mr{Alt}^{k}(U,M;\mh{P}_{c})
\\
\Omega^{k}(U,M;\mh{M})&\coloneqq\Gamma(\mr{Alt}^{k}(TM)\up U;\mh{M})
\\
\Omega^{k}(U,M;\mh{M})_{c}&\coloneqq\Gamma(\mr{Alt}^{k}(TM)\up U;\mh{M})_{c}
\\
\Omega^{k}(U,M;\mh{P})_{[c]}&\coloneqq\Omega^{k}(U,M;\mh{P}_{c}).
\end{aligned}
\end{equation*}
Furthermore define the $\mc{A}(U)$-modules
\begin{equation*}
\begin{aligned}
\mr{Alt}^{\bullet}(U,M;\mh{M})
&\coloneqq
\bigoplus_{j\in[0,n]\cap\Z}
\mr{Alt}^{j}(U,M;\mh{M})
\\
\mr{Alt}^{\bullet}(U,M;\mh{M})^{c}
&\coloneqq
\bigoplus_{j\in[0,n]\cap\Z}
\mr{Alt}^{j}(U,M;\mh{M})^{c}
\\
\mr{Alt}^{\bullet}(U,M;\mh{M})_{c}
&\coloneqq
\bigoplus_{j\in[0,n]\cap\Z}
\mr{Alt}^{j}(U,M;\mh{M})_{c}
\\
\mr{Alt}^{\bullet}(U,M;\mh{P})_{[c]}
&\coloneqq
\bigoplus_{j\in[0,n]\cap\Z}
\mr{Alt}^{j}(U,M;\mh{P})_{[c]}
\\
\Omega^{\bullet}(U,M;\mh{M})
&\coloneqq
\bigoplus_{j\in[0,n]\cap\Z}
\Omega^{j}(U,M;\mh{M})
\\
\Omega^{\bullet}(U,M;\mh{M})_{c}
&\coloneqq
\bigoplus_{j\in[0,n]\cap\Z}
\Omega^{j}(U,M;\mh{M})_{c}
\\
\Omega^{\bullet}(U,M;\mh{P})_{[c]}
&\coloneqq
\bigoplus_{j\in[0,n]\cap\Z}
\Omega^{j}(U,M;\mh{P})_{[c]}.
\end{aligned}
\end{equation*}
We call the elements of $\mr{Alt}^{\bullet}(U,M;\mh{M})$ $\mh{M}$-valued alternating tensors of $M$ defined on $U$ and those of 
$\Omega^{\bullet}(U,M;\mh{M})$ $\mh{M}$-valued forms of $M$ defined on $U$.
\end{definition}  
\begin{proposition}
\label{11141945}
We have 
\begin{equation*}
\begin{aligned}
\mf{Q}_{\bullet}^{\bullet}(U,M;\mh{P})_{[c]}
&=
\mf{Q}_{\bullet}^{\bullet}(U,M;\mh{P})_{c}
\\
&=
\mf{Q}_{\bullet}^{\bullet}(U,M;\mh{P}_{c})_{c}
\\
&=
\mf{Q}_{\bullet}^{\bullet}(U,M;\mh{P}_{c})^{c},
\end{aligned}
\end{equation*}
moreover 
\begin{equation*}
\begin{aligned}
\mr{Alt}^{\bullet}(U,M;\mh{P})_{[c]}
&=
\mr{Alt}^{\bullet}(U,M;\mh{P})_{c}
\\
&=
\mr{Alt}^{\bullet}(U,M;\mh{P}_{c})_{c}
\\
&=
\mr{Alt}^{\bullet}(U,M;\mh{P}_{c})^{c},
\end{aligned}
\end{equation*}
and 
\begin{equation*}
\begin{aligned}
\Omega^{\bullet}(U,M;\mh{P})_{[c]}
&=
\Omega^{\bullet}(U,M;\mh{P})_{c}
\\
&=
\Omega^{\bullet}(U,M;\mh{P}_{c})_{c}.
\end{aligned}
\end{equation*}
\end{proposition}
\begin{proof}
Since Lemma \ref{11140700}.
\end{proof}
If the ring $A$ present in \cite[II.61(23)]{BourA1} is assumed to be commutative, then all the $\Z$-bilinear maps, 
employed in proving that the map $g$ in \cite[II.61(23)]{BourA1} is a morphism in the category of $\Z$-modules, are also 
$A$-bilinear maps, therefore $g$ in this special case is also a morphism in the category of $A$-modules, in particular 
according to \cite[II.61 Prp.7]{BourA1} $g$ is an isomorphism of $A$-modules. Therefore if $D\ni\alpha\mapsto\mh{M}_{\alpha}$ and 
$E\ni\beta\mapsto\mh{N}_{\beta}$ are families of $\mc{A}(U)$-modules we deduce that 
\begin{equation}
\label{11131552}
\bigoplus_{(\beta,\alpha)\in E\times D}\mf{Q}(\mh{N}_{\beta},\mh{M}_{\alpha})
\simeq_{\mc{A}(U)-\mr{mod}}
\mf{Q}\bigl(\bigoplus_{\beta\in E}\mh{N}_{\beta},\bigoplus_{\alpha\in D}\mh{M}_{\alpha}\bigr).
\end{equation}
where $a\simeq_{\mc{A}(U)-\mr{mod}}b$ means that $a,b$ are isomorphic $\mc{A}(U)$-modules. 
\begin{proposition}
\label{11131559} 
We have 
\begin{equation*}
\begin{aligned}
\mf{Q}_{\bullet}^{\bullet}(U,M;\mh{M})
&\simeq_{\mc{A}(U)-\mr{mod}}
\mf{Q}(\mf{T}_{\bullet}^{\bullet}(U,M),\mh{M})
\\
\mf{Q}_{\bullet}^{\bullet}(U,M;\mh{M})^{c}&\simeq_{\mc{A}(U)-\mr{mod}}\mf{Q}(\mf{T}_{\bullet}^{\bullet}(\Gamma(U,TM),\mf{a}(U)),\mh{M})
\\
\mf{Q}_{\bullet}^{\bullet}(U,M;\mh{M})_{c}&\simeq_{\mc{A}(U)-\mr{mod}}\mf{Q}(\mf{T}_{\bullet}^{\bullet}(U,M)_{c},\mh{M})
\\
\mf{Q}_{\bullet}^{\bullet}(U,M;\mh{P})_{[c]}
&\simeq_{\mc{A}(U)-\mr{mod}}
\mf{Q}(\mf{T}_{\bullet}^{\bullet}(U,M),\mh{P}_{c})
\\
\mr{Alt}^{\bullet}(U,M;\mh{M})&\simeq_{\mc{A}(U)-\mr{mod}}\mf{Q}(\mr{Alt}^{\bullet}(U,M),\mh{M})
\\
\mr{Alt}^{\bullet}(U,M;\mh{M})^{c}&\simeq_{\mc{A}(U)-\mr{mod}}\mf{Q}(\mr{Alt}^{\bullet}(\Gamma(U,TM),\mf{a}(U)),\mh{M})
\\
\mr{Alt}^{\bullet}(U,M;\mh{M})_{c}&\simeq_{\mc{A}(U)-\mr{mod}}\mf{Q}(\mr{Alt}^{\bullet}(U,M)_{c},\mh{M})
\\
\mr{Alt}^{\bullet}(U,M;\mh{P})_{[c]}&\simeq_{\mc{A}(U)-\mr{mod}}\mf{Q}(\mr{Alt}^{\bullet}(U,M),\mh{P}_{c})
\\
\Omega^{\bullet}(U,M;\mh{M})&\simeq_{\mc{A}(U)-\mr{mod}}\mf{Q}(\Omega^{\bullet}(U,M),\mh{M}).
\\
\Omega^{\bullet}(U,M;\mh{M})_{c}&\simeq_{\mc{A}(U)-\mr{mod}}\mf{Q}(\Omega_{c}^{\bullet}(U,M),\mh{M})
\\
\Omega^{\bullet}(U,M;\mh{P})_{[c]}&\simeq_{\mc{A}(U)-\mr{mod}}\mf{Q}(\Omega^{\bullet}(U,M),\mh{P}_{c}).
\end{aligned}
\end{equation*}
\end{proposition}
\begin{proof}
By \eqref{11131552}.
\end{proof}
\begin{definition}
\label{11151621}
Define for every $(r,s)\in\Z_{\geq 0}\times\Z_{\geq 0}$ the map 
\begin{equation*}
\ast_{\mh{M}}^{(r,s)}:\mh{M}\times\mf{T}_{s}^{r}(U,M)\to\mr{Q}_{s}^{r}(U,M;\mh{M}),
\end{equation*}
such that 
\begin{equation*}
(\forall f\in\mh{M})(\forall T\in\mf{T}_{s}^{r}(U,M))(\forall x\in\prod[\Gamma(U,TM),r,s])
((f\ast_{\mh{M}}^{(r,s)}T)(x)\coloneqq T(x)\cdot f)
\end{equation*}
where $\cdot$ is the $\mc{A}(U)$-action over $\mh{M}$.
\end{definition}
\begin{lemma}
\label{11151339}
Def.\ref{11151339} is well-set and $\ast_{\mh{M}}^{(r,s)}$ is $\mc{A}(U)$-bilinear.
\end{lemma}
\begin{proof}
$f\ast_{\mh{M}}^{(r,s)}T$ is clearly valued in $\mh{M}$, while the $\mc{A}(U)$-multilinearity of $T$ ensures that 
$f\ast_{\mh{M}}^{(r,s)}T$ is $\mc{A}(U)$-multilinear so, Def.\ref{11151339} is well-set. The $\mc{A}(U)$-bilinearity of 
$\ast_{\mh{M}}^{(r,s)}$ follows since the $\mc{A}(U)$-module structure of $\mf{T}_{s}^{r}(U,M)$ and of $\mh{M}$.
\end{proof}
\begin{definition}
For every $(r,s)\in\Z_{\geq 0}\times\Z_{\geq 0}$ define 
\begin{equation*}
\Pr_{(r,s)}^{\mr{Q}}:\mr{Q}_{\bullet}^{\bullet}(U,M;\mh{M})\to\mr{Q}_{s}^{r}(U,M;\mh{M}),
\end{equation*}
to be the restriction at the domain of the projector map, and define
\begin{equation*}
\imath_{(r,s)}^{\mr{Q}}:\mr{Q}_{s}^{r}(U,M;\mh{M})\to\mr{Q}_{\bullet}^{\bullet}(U,M;\mh{M}),
\end{equation*}
to be the restriction at the codomain of the injection map.
\end{definition}
\begin{lemma}
\label{11151619}
There exists a unique $\mc{A}(U)$-bilinear map 
\begin{equation*}
\ast_{\mh{M}}:\mh{M}\times\mf{T}_{\bullet}^{\bullet}(U,M)\to\mr{Q}_{\bullet}^{\bullet}(U,M;\mh{M}),
\end{equation*}
such that 
\begin{equation*}
(\forall(r,s)\in\Z_{\geq 0}\times\Z_{\geq 0})
(\Pr_{(r,s)}^{\mr{Q}}\circ(\ast_{\mh{M}})=\ast_{\mh{M}}^{(r,s)}\circ(\mr{Id}_{\mh{M}}\times\Pr_{(r,s)}^{\mf{T}})).
\end{equation*}
\end{lemma}
In other words 
\begin{equation*}
(\forall f\in\mh{M})(\forall T\in\mf{T}_{\bullet}^{\bullet}(U,M))
(\Pr_{(r,s)}^{\mr{Q}}(\ast_{\mh{M}}(f,T))=\ast_{\mh{M}}^{(r,s)}(f,T_{s}^{r})),
\end{equation*}
where $T_{s}^{r}=\Pr_{(r,s)}^{\mf{T}}(T)$. 
\begin{proof}[Proof of Lemma \ref{11151619}]
For every $f\in\mh{M}$ let 
\begin{equation*}
\ast_{\mh{M}}^{(r,s)}(f,\cdot):\mf{T}_{s}^{r}(U,M)\to\mr{Q}_{s}^{r}(U,M;\mh{M}),\,T\mapsto\ast_{\mh{M}}^{(r,s)}(f,T), 
\end{equation*}
then since Lemma \ref{11151339} and \cite[II.74 Cor. to Prp.1(b)]{BourA1} we deduce that 
\begin{multline}
\label{11151609}
\left(\exists\,!u^{(r,s)}\in
\mr{mor}_{\mc{A}(U)-\mr{mod}}\bigl(\mh{M},\mr{mor}_{\mc{A}(U)-\mr{mod}}(\mf{T}_{s}^{r}(U,M),\mr{Q}_{s}^{r}(U,M;\mh{M}))\bigr)\right)
\\
(\forall f\in\mh{M})(u_{f}^{(r,s)}=\ast_{\mh{M}}^{(r,s)}(f,\cdot)).
\end{multline}
If $f\in\mh{M}$, then $\imath_{(r,s)}^{\mr{Q}}\circ u_{f}^{(r,s)}$ is $\mc{A}(U)$-linear by \eqref{11151609} therefore, 
since \cite[II.12 Prp.6]{BourA1} we deduce that 
\begin{multline}
\label{11171622}
\left(\exists\,!\mf{u}_{f}\in\mr{mor}_{\mc{A}(U)-\mr{mod}}\bigl(\mf{T}_{\bullet}^{\bullet}(U,M),\mr{Q}_{\bullet}^{\bullet}(U,M;\mh{M})
\bigr)\right)
\\
(\forall(r,s)\in\Z_{\geq 0}\times\Z_{\geq 0})(\mf{u}_{f}\circ\imath_{(r,s)}^{\mf{T}}=\imath_{(r,s)}^{\mr{Q}}\circ u_{f}^{(r,s)}),
\end{multline}
next the $\mc{A}(U)$-linearity of $\imath_{(r,s)}^{\mr{Q}}$, the $\mc{A}(U)$-linearity of $u^{(r,s)}$ by \eqref{11151609} and the 
uniqueness of $\mf{u}_{f}$ stated in \eqref{11171622}, imply 
\begin{equation}
\label{11180631}
(f\mapsto\mf{u}_{f})\in\mr{mor}_{\mc{A}(U)-\mr{mod}}
\left(\mh{M},\mr{mor}_{\mc{A}(U)-\mr{mod}}(\mf{T}_{\bullet}^{\bullet}(U,M),\mr{Q}_{\bullet}^{\bullet}(U,M;\mh{M}))\right),
\end{equation}
then 
\begin{equation*}
\ast_{\mh{M}}:\mh{M}\times\mf{T}_{\bullet}^{\bullet}(U,M)\to\mr{Q}_{\bullet}^{\bullet}(U,M;\mh{M}),
\end{equation*}
such that 
\begin{equation*}
\ast_{\mh{M}}(f,T)=\mf{u}_{f}(T),
\end{equation*}
is a bilinear map of $\mc{A}(U)$-modules since \cite[II.74 Cor. to Prp. 1]{BourA1} which satisfies the property in the 
statement since $T=\sum_{(r,s)\in\Z_{\geq 0}\times\Z_{\geq 0}}(\imath_{(r,s)}^{\mf{T}}\circ\Pr_{(r,s)}^{\mf{T}})(T)$ for every 
$T\in\mf{T}_{\bullet}^{\bullet}(U,M)$. Finally the uniqueness of $\ast_{\mh{M}}$ follows since the uniqueness of $\mf{u}_{f}$
stated in \eqref{11171622}.
\end{proof}
\begin{proof}[Alternative proof of \eqref{11180631}]
The 1th expression in \eqref{11171622} permits to define the following map
\begin{equation*}
\begin{aligned}
\mf{h}:\mh{M}&\to
\prod_{((r,s),(t,j))\in(\Z_{\geq 0}\times\Z_{\geq 0})^{2}}\mr{mor}_{\mc{A}(U)-\mr{mod}}(\mf{T}_{s}^{r}(U,M),\mr{Q}_{j}^{t}(U,M;\mh{M})),
\\
f&\mapsto\bigl(((r,s),(t,j))\mapsto\Pr_{(t,j)}^{\mr{Q}}\circ\mf{u}_{f}\circ\imath_{(r,s)}^{\mf{T}}\bigr),
\end{aligned}
\end{equation*}
while the 2th expression in \eqref{11171622} and \eqref{11151609} imply that 
\begin{equation}
\label{11171712}
\mf{h}\in\mr{mor}_{\mc{A}(U)-\mr{mod}}\left(\mh{M},
\prod_{((r,s),(t,j))\in(\Z_{\geq 0}\times\Z_{\geq 0})^{2}}\mr{mor}_{\mc{A}(U)-\mr{mod}}(\mf{T}_{s}^{r}(U,M),\mr{Q}_{j}^{t}(U,M;\mh{M})),\right).
\end{equation}
Next in the special case in which the ring $A$ in \cite[II.13 Cor.1]{BourA1} is commutative, we have that 
the $\Z$-isomorphism in \cite[II.13 Cor.1]{BourA1} as well its inverse say $\eta_{A}$ is also a $A$-isomorphism in particular
$A$-linear. Therefore by \eqref{11171712} we deduce that $\eta_{\mc{A}(U)}\circ\mf{h}$ is $\mc{A}(U)$-linear, namely 
\eqref{11180631}.
\end{proof}
$\mf{Q}_{\bullet}^{\bullet}(U,M;\mh{M})$ and $\mf{Q}(\mf{T}_{\bullet}^{\bullet}(U,M),\mh{M})$ are isomorphic $\mc{A}(U)$-modules 
since Prp.\ref{11131559} therefore, if we identify them, then we can establish the following unique representation of 
$\mf{Q}_{\bullet}^{\bullet}(U,M;\mh{M})$ in terms of tensors on $\Gamma(U,TM)$ at values in $\mh{M}$.
\begin{theorem}
[\textbf{Representation of $\mf{Q}_{\bullet}^{\bullet}(U,M;\mh{M})$ in terms of $\mr{Q}_{\bullet}^{\bullet}(U,M;\mh{M})$}]
\label{11151400}
We have that 
\begin{multline*}
\left(\exists\,!\Upphi\in\mr{mor}_{\mc{A}(U)-\mr{mod}}
\bigl(\mf{Q}_{\bullet}^{\bullet}(U,M;\mh{M}),\mr{Q}_{\bullet}^{\bullet}(U,M;\mh{M})\bigr)\right)
\\
(\forall(f,T)\in\mh{M}\times\mf{T}_{\bullet}^{\bullet}(U,M))
(\Upphi(f\otimes T)=f\ast_{\mh{M}}T).
\end{multline*}
\end{theorem}
\begin{proof}
Since Lemma \ref{11151619} and the universal property of tensor product.
\end{proof}
\begin{definition}
\label{11051504}
Define 
\begin{equation*}
\begin{aligned}
\mf{Q}_{\bullet}^{\bullet}(U,M;G,\lambda)
&\coloneqq
\mf{Q}_{\bullet}^{\bullet}(U,M;\mf{L}_{loc}^{1}(U,G,\lambda)),
\,
\mf{Q}_{\bullet}^{\bullet}(U,M;G,\lambda;m)
\coloneqq
\mf{Q}_{\bullet}^{\bullet}(U,M;\mh{L}_{loc}^{1}(U,G,\lambda))
\\
\mr{Q}_{\bullet}^{\bullet}(U,M;G,\lambda)
&\coloneqq
\mr{Q}_{\bullet}^{\bullet}(U,M;\mf{L}_{loc}^{1}(U,G,\lambda)),
\,
\mr{Q}_{\bullet}^{\bullet}(U,M;G,\lambda;m)
\coloneqq
\mr{Q}_{\bullet}^{\bullet}(U,M;\mh{L}_{loc}^{1}(U,G,\lambda))
\\
\mf{I}_{\bullet}^{\bullet}(U,M;G,\lambda)
&\coloneqq
\mf{Q}_{\bullet}^{\bullet}(U,M;\mf{L}_{c}^{1}(U,G,\lambda)),
\,
\mf{I}_{\bullet}^{\bullet}(U,M;G,\lambda;m)
\coloneqq
\mf{Q}_{\bullet}^{\bullet}(U,M;\mh{L}_{c}^{1}(U,G,\lambda))
\\
\mf{I}_{\bullet}^{\bullet}(U,M;G,\lambda)^{c}
&\coloneqq
\mf{Q}_{\bullet}^{\bullet}(U,M;\mf{L}_{c}^{1}(U,G,\lambda))^{c},
\,
\mf{I}_{\bullet}^{\bullet}(U,M;G,\lambda;m)^{c}
\coloneqq
\mf{Q}_{\bullet}^{\bullet}(U,M;\mh{L}_{c}^{1}(U,G,\lambda))^{c}
\\
\mf{I}_{\bullet}^{\bullet}(U,M;G,\lambda)_{c}
&\coloneqq
\mf{Q}_{\bullet}^{\bullet}(U,M;\mf{L}_{c}^{1}(U,G,\lambda))_{c},
\,
\mf{I}_{\bullet}^{\bullet}(U,M;G,\lambda;m)_{c}
\coloneqq
\mf{Q}_{\bullet}^{\bullet}(U,M;\mh{L}_{c}^{1}(U,G,\lambda))_{c}
\\
\mr{I}_{\bullet}^{\bullet}(U,M;G,\lambda)
&\coloneqq
\mr{Q}_{\bullet}^{\bullet}(U,M;\mf{L}_{c}^{1}(U,G,\lambda)),
\,
\mr{I}_{\bullet}^{\bullet}(U,M;G,\lambda;m)
\coloneqq
\mr{Q}_{\bullet}^{\bullet}(U,M;\mh{L}_{c}^{1}(U,G,\lambda))
\\
\mf{T}[p]_{\bullet}^{\bullet}(U,M;G)
&\coloneqq
\mf{Q}_{\bullet}^{\bullet}(U,M;\mc{C}^{p}(U,G))
\\
\mr{T}[p]_{\bullet}^{\bullet}(U,M;G)
&\coloneqq
\mr{Q}_{\bullet}^{\bullet}(U,M;\mc{C}^{p}(U,G));
\end{aligned}
\end{equation*}
and 
\begin{equation*}
\begin{aligned}
\Gamma_{loc}(\mc{V};G;\lambda)
&\coloneqq
\Gamma(\mc{V};\mf{L}_{loc}^{1}(U,G,\lambda)),
\,
\Gamma_{loc}(\mc{V};G;\lambda;m)
\coloneqq
\Gamma(\mc{V};\mh{L}_{loc}^{1}(U,G,\lambda))
\\
\Gamma(\mc{V};G;\lambda)
&\coloneqq
\Gamma(\mc{V};\mf{L}_{c}^{1}(U,G,\lambda)),
\,
\Gamma(\mc{V};G;\lambda;m)
\coloneqq
\Gamma(\mc{V};\mh{L}_{c}^{1}(U,G,\lambda))
\\
\Gamma[p](\mc{V};G)
&\coloneqq
\Gamma(\mc{V};\mc{C}^{p}(U,G)),
\end{aligned}
\end{equation*}
and for every $p\in\Z_{\geq 0}\cup\{+\infty\}$ define
\begin{equation*}
\begin{aligned}
\mr{Alt}_{loc}^{\bullet}(U,M;G,\lambda)
&\coloneqq
\mr{Alt}^{\bullet}(U,M;\mf{L}_{loc}^{1}(U,G,\lambda)),
\,
\mr{Alt}_{loc}^{\bullet}(U,M;G,\lambda;m)
\coloneqq
\mr{Alt}^{\bullet}(U,M;\mh{L}_{loc}^{1}(U,G,\lambda))
\\
\Omega_{loc}^{\bullet}(U,M;G,\lambda)
&\coloneqq
\Omega^{\bullet}(U,M;\mf{L}_{loc}^{1}(U,G,\lambda)),
\,
\Omega_{loc}^{\bullet}(U,M;G,\lambda;m)
\coloneqq
\Omega^{\bullet}(U,M;\mh{L}_{loc}^{1}(U,G,\lambda))
\\
\mr{Alt}^{\bullet}(U,M;G,\lambda)
&\coloneqq
\mr{Alt}^{\bullet}(U,M;\mf{L}_{c}^{1}(U,G,\lambda)),
\,
\mr{Alt}^{\bullet}(U,M;G,\lambda;m)
\coloneqq
\mr{Alt}^{\bullet}(U,M;\mh{L}_{c}^{1}(U,G,\lambda))
\\
\Omega^{\bullet}(U,M;G,\lambda)
&\coloneqq
\Omega^{\bullet}(U,M;\mf{L}_{c}^{1}(U,G,\lambda)),
\,
\Omega^{\bullet}(U,M;G,\lambda;m)
\coloneqq
\Omega^{\bullet}(U,M;\mh{L}_{c}^{1}(U,G,\lambda))
\\
\mr{Alt}[p]^{\bullet}(U,M;G)
&\coloneqq
\mr{Alt}[p]^{\bullet}(U,M;\mc{B}^{p}(U,G))
\\
\mr{Alt}[p]^{\bullet}(U,M;G)^{c}
&\coloneqq
\mr{Alt}[p]^{\bullet}(U,M;\mc{B}^{p}(U,G))^{c}
\\
\mr{Alt}[p]^{\bullet}(U,M;G)_{c}
&\coloneqq
\mr{Alt}[p]^{\bullet}(U,M;\mc{B}^{p}(U,G))_{c}
\\
\mr{Alt}[p]^{\bullet}(U,M;G)_{[c]}
&\coloneqq
\mr{Alt}[p]^{\bullet}(U,M;\mc{B}^{p}(U,G))_{[c]}
\\
\Omega[p]^{\bullet}(U,M;G)
&\coloneqq
\Omega[p]^{\bullet}(U,M;\mc{B}^{p}(U,G))
\\
\Omega[p]^{\bullet}(U,M;G)_{c}
&\coloneqq
\Omega[p]^{\bullet}(U,M;\mc{B}^{p}(U,G))_{c}
\\
\Omega[p]^{\bullet}(U,M;G)_{[c]}
&\coloneqq
\Omega[p]^{\bullet}(U,M;\mc{B}^{p}(U,G))_{[c]}.
\end{aligned}
\end{equation*}
\end{definition}
The elements of $\mf{Q}_{\bullet}^{\bullet}(U,M;G,\lambda)$ (respectively $\mf{Q}_{\bullet}^{\bullet}(U,M;G,\lambda;m)$) are called 
$G$-valued (respectively measurable) scalarly locally $\lambda$-integrable tensor fields of $M$ defined on $U$, those of 
$\mf{I}_{\bullet}^{\bullet}(U,M;G,\lambda)$ (respectively $\mf{I}_{\bullet}^{\bullet}(U,M;G,\lambda;m)$)
are called $G$-valued compactly supported (respectively measurable) scalarly $\lambda$-integrable tensor fields of $M$ defined 
on $U$. The elements of $\mr{Alt}_{loc}^{\bullet}(U,M;G,\lambda)$ (respectively $\mr{Alt}_{loc}^{\bullet}(U,M;G,\lambda;m)$)
are called $G$-valued (respectively measurable) scalarly locally $\lambda$-integrable alternating tensor fields of $M$ defined 
on $U$. The elements of $\mr{Alt}^{\bullet}(U,M;G,\lambda)$ (respectively $\mr{Alt}^{\bullet}(U,M;G,\lambda;m)$)
are called compactly supported $G$-valued (respectively measurable) scalarly $\lambda$-integrable alternating tensor fields of 
$M$ defined on $U$. The elements of $\mr{Alt}[p]^{\bullet}(U,M;G)$ are called $G$-valued $p$-times continuously differentiable 
alternating tensor fields of $M$ defined on $U$, while those of $\Omega[p]^{\bullet}(U,M;G)$ are called $G$-valued $p$-times 
continuously differentiable forms of $M$ defined on $U$.
\begin{proposition}
\label{09161155}
We have for any $\gamma\in\{\emptyset,m\}$
\begin{equation*}
\mf{Q}_{\bullet}^{\bullet}(U,M;\mf{L}_{c,loc}^{1}(U,G,\lambda;\gamma))^{c}
=\mf{Q}_{\bullet}^{\bullet}(U,M;G,\lambda;\gamma)_{c}
=\mf{Q}_{\bullet}^{\bullet}(U,M;G,\lambda;\gamma)_{[c]},
\end{equation*}
and
\begin{equation*}
\mf{I}_{\bullet}^{\bullet}(U,M;G,\lambda;\gamma)=\mf{I}_{\bullet}^{\bullet}(U,M;G,\lambda;\gamma)^{c}
=\mf{I}_{\bullet}^{\bullet}(U,M;G,\lambda;\gamma)_{c}
=\mf{I}_{\bullet}^{\bullet}(U,M;G,\lambda;\gamma)_{[c]}.
\end{equation*}
\end{proposition}  
\begin{proof}
Since Prp.\ref{11141945} with the position $\mh{P}=\mf{L}_{loc}^{1}(U,G,\lambda)$ if $\gamma=\emptyset$, and 
$\mh{P}=\mh{L}_{loc}^{1}(U,G,\lambda)$ if $\gamma=m$ for the 1th sequence of equalities, while with the position 
$\mh{P}=\mf{L}_{c}^{1}(U,G,\lambda)=\mh{P}_{c}$ if $\gamma=\emptyset$, and 
$\mh{P}=\mh{L}_{c}^{1}(U,G,\lambda)=\mh{P}_{c}$ if $\gamma=m$ for the 2th sequence.
\end{proof}
\begin{corollary}
[Unique decomposition of the elements of $\mf{Q}_{s}^{r}(U,M;\mh{M})$]
\label{10011800}
Let $(U,\phi)$ be a chart of $M$ and $r,s\in\Z_{\geq 0}$, then  
\begin{equation*}
(\forall T\in\mf{Q}_{s}^{r}(U,M;\mh{M}))
(\exists\,!f\in\mh{M}^{\Xi(b^{r,s,\phi})})
\bigl(T=\sum_{\alpha\in\Xi(b^{r,s,\phi})}f_{\alpha}\otimes(\otimes(b^{r,s,\phi})^{\ast})_{\alpha}\bigr).
\end{equation*}
\end{corollary}
\begin{proof}
By \cite[II.62 Cor.1]{BourA1}.
\end{proof}  
\begin{definition}
[Bar operators on and onto $\mf{Q}(\mf{T}_{\bullet}^{\bullet}(U,M);\mh{M})$]
\label{09101743}
Define 
\begin{equation*}
\begin{aligned}
\mf{t}^{\mh{M},U}
&\in\mr{mor}_{\mc{A}(U)-\mr{mod}}\bigl(
\mh{M}\otimes_{\mc{A}(U)}\mf{T}_{\bullet}^{\bullet}(U,M),
\mh{M}\otimes_{\mc{A}(U)}\Gamma(\mf{T}_{\bullet}^{\bullet}(TM)\up U)
\bigr)
\\
\mf{t}^{\mh{M},U}&\coloneqq\mr{Id}_{\mh{M}}\otimes\mf{t}_{\R};
\end{aligned}
\end{equation*}
and 
\begin{equation*}
\begin{aligned}
\mf{r}^{\mh{M},U}
&\in\mr{mor}_{\mc{A}(U)-\mr{mod}}\bigl(
\mh{M}\otimes_{\mc{A}(U)}\Gamma(\mf{T}_{\bullet}^{\bullet}(TM)\up U),
\mh{M}\otimes_{\mc{A}(U)}\mf{T}_{\bullet}^{\bullet}(U,M)
\bigr)
\\
\mf{r}^{\mh{M},U}&\coloneqq\mr{Id}_{\mh{M}}\otimes\mf{r}_{\R}.
\end{aligned}
\end{equation*}
\end{definition}
\begin{proposition}
\label{08281918int}
$\mf{t}^{\mh{M},U}$ and $\mf{r}^{\mh{M},U}$ are isomorphisms one the inverse of the other in the category $\mc{A}(U)-\mr{mod}$.
\end{proposition}
\begin{remark}
\label{09261218}
Since Rmk. \ref{09251444int} the gluing lemma via a covering of charts extends to 
$\Gamma(\mf{T}_{\bullet}^{\bullet}(TM)\up U;\mh{M})$.
\end{remark}
\begin{convention}
Let $E[U]=E\times U$ be the trivial vector bundle with base $U$ and fiber $E$. 
\end{convention}
\begin{proposition}
\label{11181317}
We have 
\begin{equation*}
\exists!\,\mf{s}_{\mc{V}}^{\mh{P}}
\in
\mr{mor}_{\mc{A}(U)-\mr{mod}}\bigl(\Gamma(\mc{V};\mh{P}),\Gamma_{0}(U,E[U]\otimes_{\K}\mc{V})\bigr),
\end{equation*}
such that
\begin{equation*}
(\forall f\in\mh{P})(\forall\beta\in\Gamma(\mc{V}))
\left(\mf{s}_{\mc{V}}^{\mh{P}}(f\otimes\beta)=\otimes\circ(f\boxtimes\beta)\right),
\end{equation*}
where $\otimes$ in $\otimes\circ(f\boxtimes\beta)$ is the usual extension of the tensor product on the product of vector bundles
$E[U]\times\mc{V}$ and valued in the vector bundle $E[U]\otimes_{\K}\mc{V}$.
\end{proposition}
\begin{proof}
The map defined on $\mh{P}\times\Gamma(\mc{V})$ and valued in $\Gamma_{0}(U,E[U]\otimes_{\K}\mc{V})$
defined by $(f,\beta)\mapsto\otimes\circ(f\boxtimes\beta)$ is $\mc{A}(U)$-bilinear thus, the statement follows by the universal 
property of the tensor product of modules over a commutative ring.  
\end{proof}
\begin{definition}
\label{08211232}
If $Z$ is an $n$-dimensional manifold, $O$ is an open set of $Z$ and $L$ is a Hausdorff locally convex space over $\K$, then 
for every $r\in\{c,loc,mc,mloc\}\cup\Z_{\geq 0}\cup\{+\infty\}$ we let $\mh{Q}[r]_{O}^{Z}(L)$ be the following $\mc{A}(O)$-module 
\begin{equation*}
\mh{Q}[r]_{O}^{Z}(L)
\coloneqq
\begin{cases}
\mf{L}_{loc}^{1}(O,L,\lambda),\,\text{if }r=loc,
\\
\mf{L}_{c}^{1}(O,L,\lambda),\,\text{if }r=c,
\\
\mh{L}_{loc}^{1}(O,L,\lambda),\,\text{if }r=mloc,
\\
\mh{L}_{c}^{1}(O,L,\lambda),\,\text{if }r=mc,
\\
\mc{B}^{r}(U,L),\,\text{if }r\in\Z_{\geq 0}\cup\{+\infty\}.
\end{cases}
\end{equation*}
\end{definition}
\begin{convention}
\label{09241142}
If $Z$ is an $n$-dimensional manifold, $O$ is an open set of $Z$ and $L$ is a Hausdorff locally convex space over $\K$, then 
Prp.\ref{11131559} legitimates what follows. In the remaining of the paper and for every 
$r\in\{c,loc,mc,mloc\}\cup\Z_{\geq 0}\cup\{+\infty\}$ often and possibly without mentioning it, we identify
the $\mc{A}(O)$-module $\mh{Q}[r]_{O}^{Z}(L)\otimes_{\mc{A}(O)}\mf{T}_{\bullet}^{\bullet}(O,Z)$ with 
the $\mc{A}(O)$-module $\bigoplus_{(i,j)\in\Z_{\geq 0}\times\Z_{\geq 0}}\mh{Q}[r]_{O}^{Z}(L)\otimes_{\mc{A}(O)}\mf{T}_{j}^{i}(O,Z)$,
and the $\mc{A}(O)$-module $\mh{Q}[r]_{O}^{Z}(L)\otimes_{\mc{A}(O)}\mr{Alt}^{\bullet}(O,Z)$ with 
the $\mc{A}(O)$-module $\bigoplus_{k\in[0,n]\cap\Z}\mh{Q}[r]_{O}^{Z}(L)\otimes_{\mc{A}(O)}\mr{Alt}^{k}(O,Z)$.
\end{convention}
Now we are able to define the support as follows.
\begin{definition}
[\textbf{Support}]
\label{09161456}
Define 
\begin{equation*}
\begin{aligned}  
&\mr{supp}_{\mc{V}}^{\mh{P}}:\Gamma(\mc{V};\mh{P})\to\mr{Cl}(U);
\\
&\mr{supp}_{\mc{V}}^{\mh{P}}\coloneqq\mr{supp}\circ\mf{s}_{\mc{V}}^{\mh{P}}.
\end{aligned}
\end{equation*}
Furthermore define
\begin{equation*}
\begin{aligned}  
&\mr{supp}_{U}^{\mh{P}}:\mf{Q}_{\bullet}^{\bullet}(U,M;\mh{P})\to\mr{Cl}(U);
\\
&\mr{supp}_{U}^{\mh{P}}\coloneqq\mr{supp}_{\mf{T}_{\bullet}^{\bullet}(TM)\up U}^{\mh{P}}\circ\mf{t}^{\mh{P},U}.
\end{aligned}
\end{equation*}
\end{definition}
\begin{lemma}
\label{11181347}
We have 
\begin{equation*}
(\forall f\in\mh{P})(\forall\beta\in\Gamma(\mc{V}))
(\mr{supp}_{\mc{V}}^{\mh{P}}(f\otimes\beta)\subseteq\mr{supp}(f)\cap\mr{supp}(\beta)).
\end{equation*}
\end{lemma}
\begin{proof}
Let $\ze_{\mc{V}}\in\mc{V}$ be such that $\ze_{\mc{V}}(x)=\ze_{x}$ for every $x\in U$, where $\ze_{x}$ is the zero of the vector 
space $\mc{V}_{x}$, while let $\ze_{x}^{\otimes}$ be the zero of the vector space $(E[U]\otimes\mc{V})_{x}$. 
Let $f\in\mh{P}$ and $\beta\in\Gamma(\mc{V})$. Next $\ze_{E_{\R}}\otimes q=\ze_{x}^{\otimes}$ and $v\otimes\ze_{x}=\ze_{x}^{\otimes}$, 
for every $x\in U$, $q\in\mc{V}_{x}$ and $v\in E_{\R}$. Therefore since Prp.\ref{11181317} we deduce that 
\begin{equation*}
\overset{-1}{f}(\{\ze_{E_{\R}}\})\cup\overset{-1}{\beta}(\ze_{\mc{V}}(U))
\subseteq
\{x\in U\,\vert\,\mf{s}_{\mc{V}}^{\mh{P}}(f\otimes\beta)(x)=\ze_{x}\}
\end{equation*}
therefore 
\begin{equation*}
\begin{aligned}
\complement_{U}
\{x\in U\,\vert\,\mf{s}_{\mc{V}}^{\mh{P}}(f\otimes\beta)(x)=\ze_{x}\}
&\subseteq
\complement_{U}
(\overset{-1}{f}(\{\ze_{E_{\R}}\}))
\cap
\complement_{U}(\overset{-1}{\beta}(\ze_{\mc{V}}(U))),
\\
&
\subseteq
\mr{cl}_{U}(\complement_{U}(\overset{-1}{f}(\{\ze_{E_{\R}}\})))
\cap
\mr{cl}_{U}(\complement_{U}(\overset{-1}{\beta}(\ze_{\mc{V}}(U))))
\\
&=
\mr{supp}(f)\cap\mr{supp}(\beta),
\end{aligned}
\end{equation*}
where $\mr{cl}_{U}$ is the closure set operation with respect to the topological space $U$, then the statement follows.
\end{proof}
\begin{remark}
\label{11181348}
Lemma \ref{11181347} yields that $\mh{P}=\mh{P}_{c}$ implies $\mr{supp}_{\mc{V}}^{\mh{P}}(\eta)\in\mr{Cmp}(U)$ for every 
$\eta\in\Gamma(\mc{V};\mh{P})$ and then also $\mr{supp}_{U}^{\mh{P}}(\theta)\in\mr{Cmp}(U)$ for every 
$\theta\in\mf{Q}_{\bullet}^{\bullet}(U,M;\mh{P})$.
\end{remark}
We will employ the next result in order to construct in Thm.\ref{09171005} a vector-measure; it is in this context that 
locally integrability rather than just integrability comes into play.
\begin{proposition}
\label{09170948}
For any $\gamma\in\{\emptyset,m\}$, there exists a unique $\mc{A}(U)$-bilinear map $(g,\theta)\mapsto g\cdot\theta$ defined 
on $\mc{H}(U)\times\mf{Q}_{\bullet}^{\bullet}(U,M;G,\lambda;\gamma)$, valued in $\mf{I}_{\bullet}^{\bullet}(U,M;G,\lambda;\gamma)$ and 
such that for every $g\in\mc{H}(U)$, every $f\in\mf{L}_{loc}^{1}(U,G,\lambda)$ if $\gamma=\emptyset$ and 
$f\in\mh{L}_{loc}^{1}(U,G,\lambda)$ if $\gamma=m$, and every $T\in\mf{T}_{\bullet}^{\bullet}(U,M)$ we have 
$g\cdot(f\otimes T)=(gf)\otimes T$.
\end{proposition}
\begin{proof}
We shall prove the statement for $\gamma=\emptyset$, being the case $\gamma=m$ provable mutatis mutandis in the same way.
To this end let $R=\mf{I}_{\bullet}^{\bullet}(U,M;G,\lambda)$ and $g\in\mc{H}(U)$ so, the map 
$\mf{L}_{loc}^{1}(U,G,\lambda)\to\mf{L}_{c}^{1}(U,G,\lambda)$, $f\mapsto gf$ is well-set since Lemma \ref{10111124}
and clearly a morphism of $\mc{A}(U)$-modules therefore 
the map $\mf{L}_{loc}^{1}(U,G,\lambda)\times\mf{T}_{\bullet}^{\bullet}(U,M)\to R$, $(f,T)\mapsto(gf)\otimes T$ is not only
well-defined but also $\mc{A}(U)$-bilinear since the $\mc{A}(U)$-module structure of $R$.
Then, by the universal property of tensor product 
\begin{multline}
\label{11181952}
\left(\exists\,!\mf{q}(g)\in\mr{mor}_{\mc{A}(U)\mr{mod}}\bigl(\mf{Q}_{\bullet}^{\bullet}(U,M;G,\lambda),R \bigr)\right)
\\
(\forall f\in\mf{L}_{loc}^{1}(U,G,\lambda))(\forall T\in\mf{T}_{\bullet}^{\bullet}(U,M)) 
\bigl(\mf{q}(g)(f\otimes T)=(gf)\otimes T\bigr).
\end{multline}
Now the above uniqueness implies that 
\begin{equation}
\mf{q}\in
\mr{mor}_{\mc{A}(U)\mr{mod}}\left(\mc{H}(U),\mr{mor}_{\mc{A}(U)\mr{mod}}\bigl(\mf{Q}_{\bullet}^{\bullet}(U,M;G,\lambda),R \bigr)\right).
\end{equation}
Therefore the map $\mc{H}(U)\times\mf{Q}_{\bullet}^{\bullet}(U,M;G,\lambda)\to R$, $(g,\eta)\mapsto\mf{q}(g)(\eta)$ is 
$\mc{A}(U)$-bilinear since \cite[II.74 Cor. to Prp. 1]{BourA1} moreover it is the unique map satisfying the property in the 
statement since the uniqueness of the map $\mf{q}(g)$ stated in \eqref{11181952}.
\end{proof}
\begin{assumption}
Let $N$ be an $n$-dimensional differential manifold, $W$ be an open set of $N$, $F\in\mc{C}^{\infty}(W,U)$ be a 
\textbf{diffeomorphism} moreover let $F^{\ast}:\mf{T}_{\bullet}^{0}(U,M)\to\mf{T}_{\bullet}^{0}(W,N)$ be the usual pullback via $F$ on 
tensors and let $F^{\ast,\gamma}:\Gamma(\mf{T}_{\bullet}^{0}(TM)\up U)\to\Gamma(\mf{T}_{\bullet}^{0}(TN)\up W)$ be the usual pullback 
via $F$ defined on sections. By abuse of language we let $F_{\star}$ denote $(F^{\ast}\vert_{\mc{A}(U)}^{\mc{A}(W)})_{\star}$.
\end{assumption}
\begin{definition}
\label{11111716}
For every $r\in\{c,loc,mc,mloc\}\cup\Z_{\geq 0}\cup\{+\infty\}$ define
\begin{equation*}
\begin{aligned}
\overset{\ast,r,G}{F}:\mh{Q}[r]_{U}^{M}(G)&\to\mh{Q}[r]_{W}^{N}(G),
\\
f&\mapsto f\circ F.
\end{aligned}
\end{equation*}
\end{definition}
\begin{lemma}
\label{09251753}
Def.\ref{11111716} is well-defined more specifically for every $r\in\{c,loc,mc,mloc\}\cup\Z_{\geq 0}\cup\{+\infty\}$ 
we have $f\in\mh{Q}[r]_{U}^{M}(G)\Rightarrow f\circ F\in\mh{Q}[r]_{W}^{N}(G)$. Furthermore 
\begin{equation*}
\overset{\ast,r,G}{F}\in\mr{mor}_{\mc{A}(U)-\mr{mod}}\bigl(\mh{Q}[r]_{U}^{M}(G),F_{\star}(\mh{Q}[r]_{W}^{N}(G))\bigr).
\end{equation*}
\end{lemma}
\begin{proof}
Let $(V,\psi)$ be a chart of the manifold $W$. Now 
\begin{equation*}
F\circ\imath_{V}^{W}\circ\psi^{-1}
=
\imath_{F(V)}^{U}\circ(F\circ\imath_{V}^{W})_{\natural}\circ\psi^{-1}
=
\imath_{F(V)}^{U}\circ\phi^{-1},
\end{equation*}
where $\phi=\psi\circ(F\circ\imath_{V}^{W})_{\natural}^{-1}:F(V)\to\psi(V)$, notice that $\psi(V)$ is an open set of $\R^{n}$, 
while $F(V)$ is an open set of $U$ since $F$ is a homeomorphism being a diffeomorphism thus, $(F(V),\phi)$ is a chart of $U$ 
and the 1th part of the statement follows for the case $r\in\{c,loc\}\cup\Z_{\geq 0}\cup\{+\infty\}$. The case $r\in\{mc,mloc\}$ 
follows by the case above proven and by the fact that for every $f\in\mh{Q}[r]_{U}^{M}(G)$ we have that 
$f\circ\imath_{F(V)}^{U}\circ\phi^{-1}$ is $\lambda_{\psi(V)}$-measurable since $(F(V),\phi)$ is a chart of $U$. 
Finally the 2th part of the statement trivially follows by the 1th one.
\end{proof}
Note that
\begin{equation*}
F^{\ast}\in\mr{mor}_{\mc{A}(U)-\mr{mod}}(\mf{T}_{\bullet}^{0}(U,M),F_{\star}(\mf{T}_{\bullet}^{0}(W,N))),
\end{equation*}
 while
\begin{equation*}
F^{\ast,\gamma}\in\mr{mor}_{\mc{A}(U)-\mr{mod}}
\left(\Gamma(\mf{T}_{\bullet}^{0}(TM)\up U),F_{\star}(\Gamma(\mf{T}_{\bullet}^{0}(TN)\up W))\right).
\end{equation*}
\begin{convention}
For every $r\in\{c,loc,mc,mloc\}\cup\Z_{\geq 0}\cup\{+\infty\}$ let 
\begin{equation*}
\uptheta^{F,r,G}:F_{\star}(\mh{Q}[r]_{W}^{N}(G))\otimes_{\mc{A}(U)}F_{\star}(\mf{T}_{\bullet}^{0}(W,N))
\to
F_{\star}\left(\mh{Q}[r]_{W}^{N}(G)\otimes_{\mc{A}(W)}\mf{T}_{\bullet}^{0}(W,N)\right),
\end{equation*}
be the canonical map, and let 
\begin{equation*}
\uptheta^{F,r,G,\gamma}:F_{\star}(\mh{Q}[r]_{W}^{N}(G))\otimes_{\mc{A}(U)}F_{\star}(\Gamma(\mf{T}_{\bullet}^{0}(TN)\up W))
\to
F_{\star}\left(\mh{Q}[r]_{W}^{N}(G)\otimes_{\mc{A}(W)}\Gamma(\mf{T}_{\bullet}^{0}(TN)\up W)\right),
\end{equation*}
be the canonical map 
\end{convention}
\begin{definition}
[Pullback]
\label{11271807}
Define for every $r\in\{c,loc,mc,mloc\}\cup\Z_{\geq 0}\cup\{+\infty\}$ 
\begin{equation}
\label{09180901}
\begin{aligned}
\overset{\times,r,G}{F}
&\in
\mr{mor}_{\mc{A}(U)-\mr{mod}}\left(\mh{Q}[r]_{U}^{M}(G)\otimes_{\mc{A}(U)}\mf{T}_{\bullet}^{0}(U,M),
F_{\star}\bigl(\mh{Q}[r]_{W}^{N}(G)\otimes_{\mc{A}(W)}\mf{T}_{\bullet}^{0}(W,N)\bigr)\right)
\\
\overset{\times,r,G}{F}&\coloneqq\uptheta^{F,r,G}\circ(\overset{\ast,r,G}{F}\otimes F^{\ast}),
\end{aligned}
\end{equation}
and define 
\begin{equation}
\label{09180902}
\begin{aligned}
\overset{\times,r,G,\gamma}{F}&\in
\mr{mor}_{\mc{A}(U)-\mr{mod}}\left(\mh{Q}[r]_{U}^{M}(G)\otimes_{\mc{A}(U)}\Gamma(\mf{T}_{\bullet}^{0}(TM)\up U),
F_{\star}\bigl(\mh{Q}[r]_{W}^{N}(G)\otimes_{\mc{A}(W)}\Gamma(\mf{T}_{\bullet}^{0}(TN)\up W)\bigr)\right)
\\
\overset{\times,r,G,\gamma}{F}
&\coloneqq\uptheta^{F,r,G,\gamma}\circ(\overset{\ast,r,G}{F}\otimes F^{\ast,\gamma}).
\end{aligned}
\end{equation}
\end{definition}
\begin{remark}
\label{11272118}
If $\rho:B\to A$ is a morphism of rings, $Y$ is a $A$-module, $Z$ is a $A$-module and 
$\zeta:\rho_{\star}(Y)\otimes_{B}\rho_{\star}(Z)\to\rho_{\star}(Y\otimes_{A}Z)$ is the canonical map, then
$\zeta$ is a $B$-module morphism \cite[II.57 Rmk.1]{BourA1} and surjective \cite[II.53 Prp.2]{BourA1}. 
Furthermore if $\gamma:C\to B$ is a morphism of rings, then easily we see that 
$(\rho\circ\gamma)_{\star}(Z)=\gamma_{\star}(\rho_{\star}(Z))$.
\end{remark}
\begin{proposition}
[Pullback is contravariant]
\label{11272114}
Let $Q$ be a smooth manifold, $V$ be an open set of $Q$, $P\in\mc{C}^{\infty}(V,W)$ and 
$r\in\{c,loc,mc,mloc\}\cup\Z_{\geq 0}\cup\{+\infty\}$. Then 
\begin{equation*}
F_{\star}(P^{\times,r,G})\circ F^{\times,r,G}=(F\circ P)^{\times,r,G},
\end{equation*} 
and 
\begin{equation*}
F_{\star}(P^{\times,r,G,\gamma})\circ F^{\times,r,G,\gamma}=(F\circ P)^{\times,r,G,\gamma}.
\end{equation*} 
\end{proposition}
\begin{proof}
The statement is well-set since Rmk.\ref{11272118} and it follows since $(F\circ P)^{\ast,r,G}=P^{\ast,r,G}\circ F^{\ast,r,G}$,
$(F\circ P)^{\ast}=P^{\ast}\circ F^{\ast}$, and since $(F\circ P)^{\ast,\gamma}=P^{\ast,\gamma}\circ F^{\ast,\gamma}$.
\end{proof}
\begin{remark}
\label{11261838}
$\uprho^{U}$ is a morphism of rings, then by Def.\ref{11271807}
\begin{multline*}
\uprho_{\star}^{U}(\overset{\times,r,G}{F})\in
\mr{mor}_{\R-\mr{mod}}\bigl(
\uprho_{\star}^{U}\bigl(\mh{Q}[r]_{U}^{M}(G)\otimes_{\mc{A}(U)}\mf{T}_{\bullet}^{0}(U,M)\bigr),
\\
\uprho_{\star}^{U}\bigl(F_{\star}\bigl(\mh{Q}[r]_{W}^{N}(G)\otimes_{\mc{A}(W)}\mf{T}_{\bullet}^{0}(W,N)\bigr)\bigr)
\bigr)
\end{multline*}
and 
\begin{multline*}
\uprho_{\star}^{U}(\overset{\times,r,G,\gamma}{F})\in
\mr{mor}_{\R-\mr{mod}}\bigl(
\uprho_{\star}^{U}\bigl(\mh{Q}[r]_{U}^{M}(G)\otimes_{\mc{A}(U)}\Gamma(\mf{T}_{\bullet}^{0}(TM)\up U)\bigr),
\\
\uprho_{\star}^{U}\bigl(F_{\star}\bigl(\mh{Q}[r]_{W}^{N}(G)\otimes_{\mc{A}(W)}\Gamma(\mf{T}_{\bullet}^{0}(TN)\up W)\bigr)\bigr)
\bigr).
\end{multline*}
\end{remark}
Next we prepare for the definition of pushforward.
\begin{definition}
\label{12141346}
Let $\uppsi\in\mc{L}(G,H)$ and $r\in\{c,loc,mc,mloc\}\cup\Z_{\geq 0}\cup\{+\infty\}$, set 
$\updelta_{\star}^{\K,r}(\uppsi)=\updelta_{\star}^{\K}(\uppsi)$ if $r\in\Z_{>0}\cup\{+\infty\}$, while 
$\updelta_{\star}^{\K,r}(\uppsi)=\uppsi$ if $r\in\{0,c,loc,mc,mloc\}$. Next define
\begin{equation*}
\uppsi_{\ast,r,U,M,G}:\mh{Q}[r]_{U}^{M}(G)\ni f\mapsto\updelta_{\star}^{\K,r}(\uppsi)\circ f\in\mh{Q}[r]_{U}^{M}(H).
\end{equation*}
\end{definition}
Notice that $\updelta_{\star}^{\K}(\uppsi)\in\mc{L}(G_{\R},H_{\R})$ since $\uppsi$ is linear and continuous therefore the above 
definition is trivially well-set in the cases $r\in\{c,loc\}\cup\Z_{\geq 0}\cup\{+\infty\}$ and then it is well-set also in the 
case $r\in\{mc,mloc\}$ as a consequence of \cite[IV.63 Thm.1]{IntBourb}.
Clearly we have 
\begin{lemma}
\label{08281522int}
If $\uppsi\in\mc{L}(G,H)$, then for every $r\in\{c,loc,mc,mloc\}\cup\Z_{\geq 0}\cup\{+\infty\}$ we have 
$\uppsi_{\ast,r,U,M,G}\in\mr{mor}_{\mc{A}(U)-\mr{mod}}(\mh{Q}[r]_{U}^{M}(G),\mh{Q}[r]_{U}^{M}(H))$.
\end{lemma}
The above result permits to give the following
\begin{definition}
[Pushforward]
\label{08281845int}
Let $\uppsi\in\mc{L}(G,H)$ and $r\in\{c,loc,mc,mloc\}\cup\Z_{\geq 0}\cup\{+\infty\}$, define  
\begin{equation*}
\begin{aligned}
\uppsi_{\times,r,U,M,G;\mh{N}}
&\in\mr{mor}_{\mc{A}(U)-\mr{mod}}\left(\mh{Q}[r]_{U}^{M}(G)\otimes_{\mc{A}(U)}\mh{N},\mh{Q}[r]_{U}^{M}(H)\otimes_{\mc{A}(U)}\mh{N}\right),
\\
\uppsi_{\times,r,U,M,G;\mh{N}}&\coloneqq\uppsi_{\ast,r,U,M,G}\otimes\mr{Id}_{\mh{N}}.
\end{aligned}
\end{equation*}
Set
\begin{equation*}
\begin{aligned}
\uppsi_{\times,r,U,M,G}&\coloneqq\uppsi_{\times,r,U,M,G;\mf{T}_{\bullet}^{\bullet}(U,M)},
\\
\uppsi_{\times,r,U,M,G,\gamma}&\coloneqq\uppsi_{\times,r,U,M,G;\Gamma(\mf{T}_{\bullet}^{\bullet}(TM)\up U)}.
\end{aligned}
\end{equation*}
\end{definition}
Then easily we find that
\begin{proposition}
[Commutation properties]  
\label{09170918}
Let $r\in\{c,loc,mc,mloc\}\cup\Z_{\geq 0}\cup\{+\infty\}$, 
then for every $x\in\mh{Q}[r]_{U}^{M}(G)\otimes_{\mc{A}(U)}\mf{T}_{\bullet}^{0}(U,M)$ we have 
\begin{equation}
\label{09170918a}
F_{\star}(\mf{t}^{\mh{Q}[r]_{W}^{N}(G),W})(\overset{\times,r,G}{F}(x))
=
\overset{\times,r,G,\gamma}{F}(\mf{t}^{\mh{Q}[r]_{U}^{M}(G),U}(x)).
\end{equation}
Furthermore if $\uppsi\in\mc{L}(G,H)$, then we have 
\begin{equation*}
\begin{aligned}
\uppsi_{\times,r,U,M,G}\circ\mf{r}^{\mh{Q}[r]_{U}^{M}(G),U}
&=
\mf{r}^{\mh{Q}[r]_{U}^{M}(H),U}\circ\uppsi_{\times,r,U,M,G,\gamma}
\\
\uppsi_{\times,r,U,M,G,\gamma}\circ\mf{t}^{\mh{Q}[r]_{U}^{M}(G),U}
&=
\mf{t}^{\mh{Q}[r]_{U}^{M}(H),U}\circ\uppsi_{\times,r,U,M,G}
\end{aligned}
\end{equation*}
and for every $y\in\mh{Q}[r]_{U}^{M}(G)\otimes_{\mc{A}(U)}\mf{T}_{\bullet}^{0}(U,M)$ we have 
\begin{equation}
\label{11271705}
F_{\star}(\uppsi_{\times,r,W,N,G})(\overset{\times,r,G}{F}(y))
=
\overset{\times,r,H}{F}(\uppsi_{\times,r,U,M,G}(y)).
\end{equation}
\end{proposition}
\begin{remark}
\label{09261218diff}
Since Rmk. \ref{09251444} the gluing lemma via a covering of charts extends to $\Gamma(\mf{T}_{\bullet}^{\bullet}(TM)\up U;G)$.
\end{remark}
\begin{convention}
Let $r\in\{c,loc,mc,mloc\}\cup\Z_{\geq 0}\cup\{+\infty\}$, whenever it will not cause confusion we convein to let $\uppsi_{\ast}$ 
denote $\uppsi_{\ast,r,U,M,G}$, $\uppsi_{\times}$ denote $\uppsi_{\times,r,U,M,G;\mh{N}}$, $\overset{\times}{F}$ denote 
$\overset{\times,r,G,\gamma}{F}$ and $\overset{\times,r,G}{F}$.
\end{convention}
\begin{proposition}
\label{12031900}
Let $A$ be a ring, $\{E_{i}\}_{i\in I}$ a family of $A$-modules, $X$ be a nonempty set moreover let $E=\bigoplus_{i\in I}E_{i}$, 
and $F=\bigoplus_{i\in I}E_{i}^{X}$. Let $E[\omega]^{X}$ be the direct sum of $|X|$ copies of $E$ and for every 
$g\in E[\omega]^{X}$ let $P(g)$ be the support of $g$ namely 
\begin{equation*}
\begin{aligned}
E[\omega]^{X}&\coloneqq\{g\in E^{X}\,\vert\,(\exists\,P\in\mc{P}_{\omega}(I))(\forall i\in\complement_{I}P)
(\Pr_{i}^{E}\circ g=\ze_{X,i})\},
\\
P(g)&\coloneqq\bigcap\{P\in\mc{P}_{\omega}(I)\,\vert\,(\forall i\in\complement_{I}P)(\Pr_{i}^{E}\circ g=\ze_{X,i})\};
\end{aligned}
\end{equation*}
where $\mc{P}_{\omega}(I)$ is the set of all finite parts of $I$, while $\ze_{X,i}\in E_{i}^{X}$ is the constant map on $X$ with 
value $\ze_{i}\in E_{i}$. For every $j\in I$ let 
\begin{equation*}
\begin{aligned}
\psi_{j}:E[\omega]^{X}&\to E_{j}^{X},\,g\mapsto\Pr_{j}^{E}\circ g
\\
\phi_{j}:E_{j}^{X}&\to E[\omega]^{X},\,f\mapsto\mf{i}_{j}^{E}\circ f
\end{aligned}
\end{equation*}
and let
\begin{equation*}
\phi:F\to E[\omega]^{X},\,\phi\coloneqq\bigoplus_{i\in I}\phi_{i}.
\end{equation*}
Then $\phi$ is an isomorphism in the category $A-\mr{mod}$ with inverse equal to $\psi$, where 
\begin{equation*}
\psi:E[\omega]^{X}\to F,\,g\mapsto\sum_{k\in I}(\mf{i}_{k}^{F}\circ\psi_{k})(g)=\sum_{k\in P(g)}(\mf{i}_{k}^{F}\circ\psi_{k})(g).
\end{equation*}
In particular $F$ and $E[\omega]^{X}$ are isomorphic objects in the category $A-\mr{mod}$. 
\end{proposition}
\begin{proof}[Proof of Proposition \ref{12031900}]
Let $j,s\in I$ with $s\neq j$, then $\psi_{j}$ and $\phi_{j}$ are well-defined $A$-linear maps such that 
$\psi_{j}\circ\phi_{j}=\mr{Id}_{E_{j}^{X}}$ since $\Pr_{j}^{E}\circ\mf{i}_{j}^{E}=\mr{Id}_{E_{j}}$. 
Next $\psi_{s}\circ\phi_{j}=\ze$ since $\Pr_{s}^{E}\circ\mf{i}_{j}^{E}=\ze_{s}$. Let $g\in E[\omega]^{X}$ so, 
the map $I\ni i\mapsto\psi_{i}(g)\in E_{i}^{X}$ has finite support equal to $P(g)$, then the map 
$I\ni i\mapsto(\phi_{i}\circ\psi_{i})(g)\in E[\omega]^{X}$ has finite support equal to $P(g)$ since $\mf{i}_{i}^{E}$ is linear and 
injective moreover for every $x\in X$ we have 
\begin{equation*}
\mr{ev}_{x}\bigl(\sum_{k\in I}(\phi_{k}\circ\psi_{k})(g)\bigr)
=
\mr{ev}_{x}\bigl(\sum_{k\in I}\mf{i}_{k}^{E}\circ\Pr_{k}^{E}\circ g\bigr)
=
\mr{ev}_{x}(g),
\end{equation*}
the last equality following by \cite[II.12(20)]{BourA1}, so $g=\sum_{k\in I}(\phi_{k}\circ\psi_{k})(g)$.
Therefore, the hypotheses of \cite[II.13 Cor.2]{BourA1} are fulfilled and then the statement follows. 
\end{proof}
Now the $A$-isomorphism provided in Prp.\ref{12031900} pemits to establish the following result.
\begin{corollary}
\label{12032252}
Assume $\K=\R$ and that there exists a finite family $\mr{G}=\{G_{i}\}_{i=1}^{m}$ of Hausdorff locally convex spaces on $\R$ such 
that $G=\bigoplus_{i=1}^{m}G_{i}$ topological direct sum of the family $\mr{G}$. 
Let $r\in\{c,loc,mc,mloc\}\cup\Z_{\geq 0}\cup\{+\infty\}$, and for every $i\in[1,m]\cap\Z$ let  
\begin{equation*}
\begin{aligned}
\mh{E}&=\uprho_{\star}^{U}(\mh{Q}[r]_{U}^{M}(G))
\\
\mh{F}_{i}&=\uprho_{\star}^{U}(\mh{Q}[r]_{U}^{M}(G_{i})). 
\end{aligned}
\end{equation*}
Then, by letting $a\simeq_{\R}b$ stand for $a,b$ are isomorphic $\R$-linear spaces, we have 
\begin{equation*}
\mh{E}\simeq_{\R}\bigoplus_{i=1}^{m}\mh{F}_{i}.
\end{equation*}
Furthermore the map 
\begin{equation*}
\mh{E}\to\mh{F},\,g\mapsto\sum_{k=1}^{m}(\mf{i}_{k}^{\mh{F}}\circ(\Pr_{k}^{G})_{\times})(g)
\end{equation*}
is an $\R$-linear isomorphism.
\end{corollary}
\begin{proof}
The statement follows since Prp.\ref{12031900} applied for $A=\R$, $I=[1,n]\cap\Z$, $X=U$ and $E_{k}$ equal to the $\R$-linear 
space underlying $G_{k}$ for every $k\in I$, and it follows since for every $g\in\mh{E}$ we have $\Pr_{k}^{G}\circ g\in\mh{F}_{k}$ 
by the fact that $\Pr_{k}^{G}$ is linear and continuous.
\end{proof}
\section{Operations on $\mr{Alt}^{k}(U,M;\mh{M})$}
\label{09172013}
\begin{corollary}
[Unique decomposition of $\mr{Alt}^{k}(U,M;\mh{M})$ with $U$ domain of a chart]  
\label{08262041}
If $(U,\phi)$ is a chart of $M$ and $k\in[0,n]\cap\Z$ then,  
\begin{equation*}
(\forall\theta\in\mr{Alt}^{k}(U,M;\mh{M}))
(\exists\,!f\in\mh{M}^{M(k,n,<)})
\bigl(\theta=\sum_{I\in M(k,n,<)}f_{I}\otimes\mc{E}_{dx^{\phi}}(I)\bigr).
\end{equation*}
\end{corollary}
\begin{proof}
By \cite[II.62 Cor.1]{BourA1}.
\end{proof}  
\begin{lemma}
\label{12031657}
If $T\in\mf{T}_{s}^{r}(U,M)$, then \eqref{12031657a}$\Leftrightarrow$\eqref{12031657b}, where
\begin{enumerate}
\item
$(\forall x\in U)(\mf{t}_{\R}(T)(x)\neq\ze)$, 
\label{12031657a}
\item
\begin{equation*}
(\forall x\in U)(\exists\,X\in\Gamma(U,TM)^{s})(\exists\,\theta\in(\Gamma(U,TM)^{\ast})^{r})
(T(\theta_{1},\dots,\theta_{s},X_{1},\dots,X_{r})(x)\neq\ze).
\end{equation*}
\label{12031657b}
\end{enumerate}
\end{lemma}
\begin{proof}
\eqref{12031657b}$\Rightarrow$\eqref{12031657a} follows directly by the definition of the operator $\mf{t}_{\R}$.
Next \eqref{12031657a} is equivalent to state that 
\begin{multline*}
(\forall x\in U)(\exists\,X\in\Gamma(U,TM)^{s})(\exists\,\theta\in(\Gamma(U,TM)^{\ast})^{r})
\\
(\mf{t}_{\R}(T)(x)(\theta_{1}(x),\dots,\theta_{s}(x),X_{1}(x),\dots,X_{r}(x))\neq 0),
\end{multline*}
where as usual above we have identified the isomorphic $\mc{A}(U)$-modules $\Gamma(U,TM)^{\ast}$ and $\Gamma(U,(TM)^{\ast})$.
Therefore, \eqref{12031657a}$\Rightarrow$\eqref{12031657b} follows by the definition of the opearator $\mf{r}_{\R}$ and the fact 
that $T=\mf{r}_{\R}(\mf{t}_{\R}(T))$.
\end{proof}
\begin{definition}
We call $\omega\in\mr{Alt}^{n}(U,M)$ a volume form iff $\mf{t}_{\R}(\omega)$ is a volume form namely iff 
the equivalent properties established in Lemma \ref{12031657} are satisfied for $T=\omega$.
\end{definition}
\begin{corollary}
[Unique decomposition of $\mr{Alt}^{n}(U,M;\mh{M})$ with respect to a volume form]  
\label{12011455}
Let $\omega\in\mr{Alt}^{n}(U,M)$ be a volume form, then
\begin{equation*}
(\exists\,!\mf{f}_{\omega}^{\mh{M}}:\mr{Alt}^{n}(U,M;\mh{M})\to\mh{M})
(\forall\theta\in\mr{Alt}^{n}(U,M;\mh{M})(\theta=\mf{f}_{\omega}^{\mh{M}}(\theta)\otimes\omega).
\end{equation*}
\end{corollary}
\begin{proof}
$\{\omega\}$ is a basis of the $\mc{A}_{U}$-module $\mr{Alt}^{n}(U,M)$, then the statement follows since 
\cite[II.62 Cor.1]{BourA1}.
\end{proof}  
\begin{notations}
\label{12011532}
Let $\omega\in\mr{Alt}^{n}(U,M)$ be a volume form, let $r\in\{c,loc\}\cup\Z_{\geq 0}\cup\{+\infty\}$, then we let 
$\mf{f}_{\omega}^{r,U,M,G}$ denote $\mf{f}_{\omega}^{\mh{Q}[r]_{U}^{M}(G)}$ and $\mf{f}_{\omega}^{r,M,G}$ denote $\mf{f}_{\omega}^{r,M,M,G}$.
\end{notations}
Next we start the sequence of results required to define the extension of the wedge product and the extension of the insertion 
operator. If $\otimes:G\times H\to G\otimes H$ is the bilinear and continuous tensor product map defined on $G\times H$ and 
valued in the projective tensor product $G\otimes H$, 
then by recalling that 
$\mf{j}_{G\otimes H}\in\mc{L}(G\otimes H,G\widehat{\otimes}H)$ denotes the canonical map of $G\otimes H$ into its completion, 
we provide the following  
\begin{definition}
Let $\ov{\otimes}:G\times H\to G\widehat{\otimes}H$ be the bilinear and continuous map defined by 
$\ov{\otimes}=\mf{j}_{G\otimes H}\circ\otimes$.
\end{definition}
In the following important result Def.\ref{07152211} and Def.\ref{07152212} are essential.
\begin{theorem}
\label{09251928}
Assume that there exist $\K$-linear subspaces $X$ of $G^{\ast}$ and $Y$ of $H^{\ast}$ such that the topology on $G$ and $H$ are 
$\sigma(G,X)$ and $\sigma(H,Y)$ respectively. Then for every $p\in\Z_{\geq 0}\cup\{+\infty\}$ the following
\begin{equation*}
\begin{aligned}
\mh{L}_{c}^{1}(U,G,\lambda)\times\mc{B}^{p}(U,H)&\to\mh{L}_{c}^{1}(U,G\widehat{\otimes}H,\lambda),
\\
(f,g)&\mapsto\ov{\otimes}\circ(f\boxtimes g);
\end{aligned}
\end{equation*}
and 
\begin{equation*}
\begin{aligned}
\mh{L}_{loc}^{1}(U,G,\lambda)\times\mc{B}^{p}(U,H)&\to\mh{L}_{loc}^{1}(U,G\widehat{\otimes}H,\lambda),
\\
(f,g)&\mapsto\ov{\otimes}\circ(f\boxtimes g);
\end{aligned}
\end{equation*}
are well-defined $\mc{A}(U)$-bilinear maps.
\end{theorem}
\begin{proof}
It is sufficient to prove that the maps in the statement are well-defined since then their $\mc{A}(U)$-bilinearity would be 
trivial. Since the topological dual of a Hausdorff topological linear space is as linear space $\K$-isomorphic to the 
topological dual of its completion, we deduce by \cite[Prp.2, statement 1, pg. 30]{gro} that $(G\widehat{\otimes}H)^{\prime}$ is as
linear space $\K$-isomorphic via the universal property to the space of bilinear continuous $\K$-forms on $G\times H$. Therefore
given any continuous bilinear $\K$-form $b$ on $G\times H$ we have $\widehat{b}\in(G\widehat{\otimes}H)^{\prime}$, where
$\widehat{b}$ is the continuous extension at $G\widehat{\otimes}H$ of the linearization of $b$ via the universal property 
moreover any element of $(G\widehat{\otimes}H)^{\prime}$ arises uniquely in this way finally there exist $a>0$, 
$\uppsi\in X=G^{\prime}$ and $\upphi\in Y=H^{\prime}$ such that for every $(u,v)\in G\times H$ we have 
\begin{equation}
\label{09251928b}
|\widehat{b}(u\otimes v)|\leq a|\uppsi(u)|\,|\upphi(v)|.
\end{equation}
Indeed by the choice of the topologies of $G$ and $H$ we have that $\{G\ni x\mapsto|f(x)|\in\R_{\geq 0}\,\vert\,f\in X\}$ and 
$\{H\ni y\mapsto|g(y)|\in\R_{\geq 0}\,\vert\,g\in Y\}$ are fundamental set of seminorms of $G$ and $H$ respectively, so 
\eqref{09251928b} follows since $b$ is continuous and by the property of the projective tensor topoology on $G\otimes H$ 
provided in \cite[Prp.2, statement 2, pg. 31]{gro}. 
Next let $f\in\mh{L}_{c}^{1}(U,G,\lambda)$, $g\in\mc{B}^{p}(U,H)$ thus, $\uppsi_{\times}(f)\in\mh{L}_{c}^{1}(U,\K,\lambda)$ and 
$\upphi_{\times}(g)\in\mc{B}^{p}(U)$. Let $(V,\phi)$ be a chart of $U$ and let $\mu$ denote $\lambda_{\phi(V)}$.
Now $\widehat{b}\circ\ov{\otimes}$ is continuous, while $f\circ\phi^{-1}$ and $g\circ\phi^{-1}$ are $\mu$-measurable the first 
because of definition\footnote{It is exactly here that we require the space $\mh{L}_{c}^{1}(U,G,\lambda)$ rather than just 
$\mf{L}_{c}^{1}(U,G,\lambda)$.} the second because continuous, therefore by letting 
$h=\widehat{b}\circ\ov{\otimes}\circ(f\boxtimes g)\circ\phi^{-1}$ we deduce that $h$ is $\mu$-measurable since 
\cite[IV.63 Thm.1]{IntBourb}. Furthermore $\mh{L}_{c}^{1}(U,\K,\lambda)$ is a $\mc{B}^{p}(U)$-module
since Cor.\ref{10111859} in particular $\mu^{\ast}|(\uppsi_{\times}(f)\cdot\upphi_{\times}(g))\circ\phi^{-1}|<\infty$, then 
$\mu^{\ast}|h|<\infty$ by \eqref{09251928b}, therefore $h$ is $\mu$-integrable since \cite[IV.71 Thm.5]{IntBourb}. 
Next $\ov{\otimes}$ is continuous, and $f\circ\phi^{-1}$ and $g\circ\phi^{-1}$ are $\mu$-measurable so 
$\ov{\otimes}\circ(f\boxtimes g)\circ\phi^{-1}$ is $\mu$-measurable since \cite[IV.63 Thm.1]{IntBourb}.
Thus the 1th part of the statement follows. 
Next let $e\in\mh{L}_{loc}^{1}(U,G,\lambda)$, let $j=\widehat{b}\circ\ov{\otimes}\circ(e\boxtimes g)\circ\phi^{-1}$ and 
$k\in\mc{H}(\phi(V))$ thus, $e\circ\phi^{-1}\in\mh{L}_{loc}^{1}(\phi(V),G,\mu)$ and $g\circ\phi^{-1}\in\mc{B}^{p}(\phi(V),H)$
as a direct consequence of the respective definitions, then since Lemma \ref{10111124} we obtain
\begin{equation}
\label{11241842}
k(e\circ\phi^{-1})\in\mh{L}_{c}^{1}(\phi(V),G,\mu).
\end{equation}
Next 
\begin{equation}
\label{11241912}
\begin{aligned}
kj
&=
k(\widehat{b}\circ\ov{\otimes}\circ((e\circ\phi^{-1})\boxtimes(g\circ\phi^{-1})))
\\
&=
\widehat{b}\circ\ov{\otimes}\circ(k(e\circ\phi^{-1})\boxtimes(g\circ\phi^{-1}))
\end{aligned}
\end{equation}
but by \eqref{11241842} and the 1th part of the statement above proven we deduce that 
\begin{equation*}
\ov{\otimes}\circ(k(e\circ\phi^{-1})\boxtimes(g\circ\phi^{-1}))\in\mh{L}_{c}^{1}(\phi(V),G\widehat{\otimes}H,\mu),
\end{equation*}
and then by recalling that Def.\ref{12141346} is well-set we obtain 
\begin{equation*}
\widehat{b}\circ\ov{\otimes}\circ(k(e\circ\phi^{-1})\boxtimes(g\circ\phi^{-1}))\in\mh{L}_{c}^{1}(\phi(V),\K,\mu),
\end{equation*}
so $kj\in\mh{L}_{c}^{1}(\phi(V),\K,\mu)$ by \eqref{11241912} and then $j\in\mf{L}_{loc}^{1}(\phi(V),\K,\mu)$ namely 
\begin{equation}
\label{19070728}
\ov{\otimes}\circ(e\boxtimes g)\in\mf{L}_{loc}^{1}(U,G\widehat{\otimes}H,\lambda)
\end{equation}
but $\ov{\otimes}$ is continuous, while $e\circ\phi^{-1}$ and $g\circ\phi^{-1}$ are $\mu$-measurable so 
$\ov{\otimes}\circ(e\boxtimes g)\circ\phi^{-1}$ is $\mu$-measurable since \cite[IV.63 Thm.1]{IntBourb} and then by 
\eqref{19070728} we conclude 
\begin{equation*}
\ov{\otimes}\circ(e\boxtimes g)\in\mh{L}_{loc}^{1}(U,G\widehat{\otimes}H,\lambda)
\end{equation*}
and the 2th part of the statement follows.
\end{proof}
\begin{lemma}
\label{09251928diff}
Let $q,p\in\Z_{\geq 0}\cup\{+\infty\}$, then the following map
\begin{equation*}
\begin{aligned}
\mc{B}^{q}(U,G)\times\mc{B}^{p}(U,H)&\to\mc{B}^{\min\{q,p\}}(U,G\widehat{\otimes}H),
\\
(f,g)&\mapsto\ov{\otimes}\circ(f\boxtimes g),
\end{aligned}
\end{equation*}
is well-defined and $\mc{A}(U)$-bilinear.
\end{lemma}
\begin{proof}
$\otimes$ is $\K$-bilinear and continuous, then the statement follows since properties of the product of differentiable maps.
\end{proof}
\begin{notation}
Let $r\in\{mc,mloc\}\cup\Z_{\geq 0}\cup\{+\infty\}$ and $p\in\Z_{\geq 0}\cup\{+\infty\}$, set 
$x(r,p)\coloneqq r$ if $r\in\{mc,mloc\}$ and $x(r,p)\coloneqq\min\{r,p\}$ if $r\in\Z_{\geq 0}\cup\{+\infty\}$.
\end{notation}
\begin{assumption}
\label{07191725}
From Cor.\ref{09260817} till Prp.\ref{11300106} included we let $r\in\{mc,mloc\}\cup\Z_{\geq 0}\cup\{+\infty\}$ and 
$p\in\Z_{\geq 0}\cup\{+\infty\}$ moreover if $r\in\{mc,mloc\}$, then assume that there exist $\K$-linear subspaces $X$ of 
$G^{\ast}$ and $Y$ of $H^{\ast}$ such that the topology on $G$ and $H$ are $\sigma(G,X)$ and $\sigma(H,Y)$ respectively. 
\end{assumption}
\begin{corollary}
\label{09260817}
We have 
\begin{multline*}
\left(\exists\,!\uptau_{r}^{p,G,H}\in\mr{mor}_{\mc{A}(U)-\mr{mod}}
\bigl(\mh{Q}[r]_{U}^{M}(G)\otimes_{\mc{A}(U)}\mc{B}^{p}(U,H),\mh{Q}[x(r,p)]_{U}^{M}(G\widehat{\otimes}H)\bigr)\right)
\\
(\forall f\in\mh{Q}[r]_{U}^{M}(G))(\forall g\in\mc{B}^{p}(U,H))
\bigl(\uptau_{r}^{p,G,H}(f\otimes g)=\ov{\otimes}\circ(f\boxtimes g)\bigr).
\end{multline*}
\end{corollary}
\begin{proof}
Since Thm.\ref{09251928}, Lemma \ref{09251928diff} and the universal property of tensor product.
\end{proof}
\begin{assumption}
Till Prp.\ref{12051732} included we assume that $\mh{X},\mh{Y},\mh{Z}$ are $\mc{A}(U)$-modules and that 
$\upeta\in\mr{mor}_{\mc{A}(U)-\mr{mod}}(\mh{X}\otimes_{\mc{A}(U)}\mh{Y},\mh{Z})$.
\end{assumption}
\begin{corollary}
\label{12020741}
For every $(g,a)\in\mc{B}^{p}(U,H)\times\mh{X}$ there exists a unique $\upeta_{r,g,a,p}^{G,H}$ such that
\begin{equation*}
\begin{aligned}
\upeta_{r,g,a,p}^{G,H}:
\mh{Q}[r]_{U}^{M}(G)\times\mh{Y}&\to\mh{Q}[x(r,p)]_{U}^{M}(G\widehat{\otimes}H)\otimes_{\mc{A}(U)}\mh{Z},
\\
(f,\zeta)&\mapsto\left(\uptau_{r}^{p,G,H}(f\otimes g)\right)\otimes(\upeta(a\otimes\zeta)).
\end{aligned}
\end{equation*}
\end{corollary}
\begin{proof}
Since Cor.\ref{09260817}.
\end{proof}
\begin{proposition}
\label{120020748}
For every $(g,a)\in\mc{B}^{p}(U,H)\times\mh{X}$ the unique map $\upeta_{r,g,a,p}^{G,H}$ established in Cor.\ref{12020741} is 
$\mc{A}(U)$-bilinear in particular 
\begin{multline*}
\left(\exists\,!\widehat{\upeta}_{r,g,a,p}^{G,H}\in
\mr{mor}_{\mc{A}(U)-\mr{mod}}
\bigl(\mh{Q}[r]_{U}^{M}(G)\otimes_{\mc{A}(U)}\mh{Y},\mh{Q}[x(r,p)]_{U}^{M}(G\widehat{\otimes}H)\otimes_{\mc{A}(U)}\mh{Z}\bigr)\right)
\\
(\forall(f,\zeta)\in\mh{Q}[r]_{U}^{M}(G)\times\mh{Y})
\bigl(\widehat{\upeta}_{r,g,a,p}^{G,H}(f\otimes\zeta)=\upeta_{r,g,a,p}^{G,H}(f,\zeta)\bigr).
\end{multline*}
\end{proposition}
\begin{proof}
The 1th part of the statement follows since the tensor product is $\mc{A}(U)$-bilinear, since $\upeta$ is $\mc{A}(U)$-linear, 
and since $\uptau_{r}^{p,G,H}$ is $\mc{A}(U)$-linear by Cor.\ref{09260817}.
The 2th part follows by the 1th one and by the universal property of tensor product. 
\end{proof}
\begin{proposition}
[\textbf{First $(r,p,G,H)$-extension of the operator $\upeta$}]
\label{12020810}
The map $(g,a)\mapsto\widehat{\upeta}_{r,g,a,p}^{G,H}$, defined on $\mc{B}^{p}(U,H)\times\mh{X}$ and where
$\widehat{\upeta}_{r,g,a,p}^{G,H}$ is the unique map established in Prp.\ref{120020748}, is $\mc{A}(U)$-bilinear, 
in particular 
\begin{multline*}
\bigl(\exists\,!\widehat{\upeta}_{r,p}^{G,H}
\in
\\
\mr{mor}_{\mc{A}(U)-\mr{mod}}
\bigl(\mc{B}^{p}(U,H)\otimes_{\mc{A}(U)}\mh{X},
\mr{mor}_{\mc{A}(U)-\mr{mod}}
\bigl(\mh{Q}[r]_{U}^{M}(G)\otimes_{\mc{A}(U)}\mh{Y},\mh{Q}[x(r,p)]_{U}^{M}(G\widehat{\otimes}H)\otimes_{\mc{A}(U)}\mh{Z}\bigr)
\bigr)\bigr)
\\
(\forall(g,a)\in\mc{B}^{p}(U,H)\times\mh{X})
\bigl(\widehat{\upeta}_{r,p}^{G,H}(g\otimes a)=\widehat{\upeta}_{r,g,a,p}^{G,H}\bigr).
\end{multline*}
We call $\widehat{\upeta}_{r,p}^{G,H}$ the \emph{first $(r,p,G,H)$-extension of the operator $\upeta$}.
\end{proposition}
\begin{proof}
The 1th part of the statement follows since the tensor product is $\mc{A}(U)$-bilinear, since $\upeta$ is $\mc{A}(U)$-linear, 
and since $\uptau_{r}^{p,G,H}$ is $\mc{A}(U)$-linear by Cor.\ref{09260817}.
The 2th part follows by the 1th one and by the universal property of tensor product. 
\end{proof}
\begin{corollary}
[\textbf{Second $(r,p,G,H)$-extension of the operator $\upeta$}]
\label{12020811}
We have
\begin{multline*}
\bigl(\exists\,!\upeta_{r,p}^{G,H}
\in
\mr{bil}_{\mc{A}(U)}
\bigl(
\mc{B}^{p}(U,H)\otimes_{\mc{A}(U)}\mh{X},\mh{Q}[r]_{U}^{M}(G)\otimes_{\mc{A}(U)}\mh{Y};
\mh{Q}[x(r,p)]_{U}^{M}(G\widehat{\otimes}H)\otimes_{\mc{A}(U)}\mh{Z}\bigr)\bigr)
\\
(\forall a\in\mc{B}^{p}(U,H)\otimes_{\mc{A}(U)}\mh{X})(\forall b\in\mh{Q}[r]_{U}^{M}(G)\otimes_{\mc{A}(U)}\mh{Y})
\bigl(\upeta_{r,p}^{G,H}(a,b)=\widehat{\upeta}_{r,p}^{G,H}(a)(b)\bigr).
\end{multline*}
We call $\upeta_{r,p}^{G,H}$ the \emph{second $(r,p,G,H)$-extension of the operator $\upeta$}.
\end{corollary}
\begin{proof}
Since Prp.\ref{12020810}.
\end{proof}
\begin{remark}
\label{12011040}
We have that 
\begin{multline*}
(\forall(g,a)\in\mc{B}^{p}(U,H)\times\mh{X})(\forall(f,\zeta)\in\mh{Q}[r]_{U}^{M}(G)\times\mh{Y})
\\
\left(\upeta_{r,p}^{G,H}(g\otimes a,f\otimes\zeta)=\widehat{\upeta}_{r,p}^{G,H}(g\otimes a)(f\otimes\zeta)=
\left(\uptau_{r}^{p,G,H}(f\otimes g)\right)\otimes(\upeta(a\otimes\zeta))\right).
\end{multline*}
\end{remark}
\begin{proposition}
[\textbf{Pushforward commutes with $\upeta_{r,p}^{G,H}$}]
\label{12051732}
Assume that there exist $\K$-linear subspaces $X_{1}$ of $G_{1}^{\ast}$ and $Y_{1}$ of $H_{1}^{\ast}$ such that the topology on 
$G_{1}$ and $H_{1}$ are $\sigma(G_{1},X_{1})$ and $\sigma(H_{1},Y_{1})$ respectively.
Furthermore let $a\in\mc{B}^{p}(U,H)\otimes_{\mc{A}(U)}\mh{X}$ and $b\in\mh{Q}[r]_{U}^{M}(G)\otimes_{\mc{A}(U)}\mh{Y}$.
If $\uppsi\in\mc{L}(G,G_{1})$, and $\upphi\in\mc{L}(H,H_{1})$, then 
\begin{equation*}
\begin{aligned}
(\uppsi\widehat{\otimes}\upphi)_{\times}\bigl(\upeta_{r,p}^{G,H}(a,b)\bigr)
&=
\upeta_{r,p}^{G_{1},H_{1}}(\upphi_{\times}(a),\uppsi_{\times}(b))
\\
(\uppsi\widehat{\otimes}\upphi)_{\times}\bigl(\widehat{\upeta}_{r,p}^{G,H}(a)(b)\bigr)
&=
\widehat{\upeta}_{r,p}^{G_{1},H_{1}}(\upphi_{\times}(a))(\uppsi_{\times}(b)).
\end{aligned}
\end{equation*}
\end{proposition}
\begin{proof}
The statement is well-set since $\uppsi\widehat{\otimes}\upphi\in\mc{L}(G\widehat{\otimes}H,G_{1}\widehat{\otimes}H_{1})$, 
then the statement follows since Rmk.\ref{12011040} and since the following general property of tensor 
product: if $A$ is a commutative ring, $E$ and $F$ are $A$-modules, then the set $\{x\otimes y\,\vert\,(x,y)\in E\times F\}$ is 
a system of generators of the $A$-module $E\otimes_{A}F$.
\end{proof}
Now we are in the position of defining in Def.\ref{12022149} the extension of the insertion operator via Prp.\ref{12020810} and 
defining in Def.\ref{07191816} the extension of the wedge product via Cor.\ref{12020811}. More exactly 
let $\uptheta$ be the $\mc{A}(U)$-linearization via the universal property of tensor product of the following 
$\mc{A}(U)$-bilinear map 
\begin{equation*}
\Gamma(U,TM)\times\mr{Alt}^{\bullet}(U,M)\to\mr{Alt}^{\bullet}(U,M),\,(x,\zeta)\mapsto\mf{i}_{x}(\zeta),
\end{equation*}
where $\mf{i}$ is the standard insertion operator on $\mr{Alt}^{\bullet}(U,M)$. 
Then Prp.\ref{12020810} legitimates the following definition. 
\begin{definition}
[\textbf{$(r,p,G,H)$-extension of the insertion operator}]
\label{12022149}
Define 
\begin{multline*}
\mf{i}_{r,p}^{U,M,G,H}
\in
\mr{mor}_{\mc{A}(U)-\mr{mod}}
\bigl(\mc{B}^{p}(U,H)\otimes_{\mc{A}(U)}\Gamma(U,TM),
\\
\mr{mor}_{\mc{A}(U)-\mr{mod}}
\bigl(\mr{Alt}^{\bullet}(U,M;\mh{Q}[r]_{U}^{M}(G)),\mr{Alt}^{\bullet}(U,M;\mh{Q}[x(r,p)]_{U}^{M}(G\widehat{\otimes}H))\bigr)
\bigr),
\end{multline*}
such that 
\begin{equation*}
\mf{i}_{r,p}^{U,M,G,H}\coloneqq\widehat{\uptheta}_{r,p}^{G,H}.
\end{equation*}
By abuse of language we call $\mf{i}_{r,p}^{U,M,G,H}$ the \emph{$(r,p,G,H)$-extension of the insertion operator}.
\end{definition}
\begin{remark}
\label{12011040ins}
By Rmk.\ref{12011040} we have that 
\begin{multline*}
(\forall(g,x)\in\mc{B}^{p}(U,H)\times\Gamma(U,TM))(\forall(f,\zeta)\in\mh{Q}[r]_{U}^{M}(G)\times\mr{Alt}^{\bullet}(U,M))
\\
\left(\mf{i}_{r,p}^{U,M,G,H}(g\otimes x)(f\otimes\zeta)
=
\left(\uptau_{r}^{p,G,H}(f\otimes g)\right)\otimes(\mf{i}_{x}(\zeta))\right).
\end{multline*}
\end{remark}
\begin{proposition}
[\textbf{Pushforward commutes with $(r,p,G,H)$-extension of the insertion operator}]
\label{12051823}
Assume that there exist $\K$-linear subspaces $X_{1}$ of $G_{1}^{\ast}$ and $Y_{1}$ of $H_{1}^{\ast}$ such that the topology on 
$G_{1}$ and $H_{1}$ are $\sigma(G_{1},X_{1})$ and $\sigma(H_{1},Y_{1})$ respectively.
Let $a\in\mc{B}^{p}(U,H)\otimes_{\mc{A}(U)}\Gamma(U,TM)$ and $b\in\mr{Alt}^{\bullet}(U,M;\mh{Q}[r]_{U}^{M}(G))$.
If $\uppsi\in\mc{L}(G,G_{1})$, and $\upphi\in\mc{L}(H,H_{1})$, then 
\begin{equation*}
(\uppsi\widehat{\otimes}\upphi)_{\times}\bigl(\mf{i}_{r,p}^{U,M,G,H}(a)(b)\bigr)=\mf{i}_{r,p}^{U,M,G_{1},H_{1}}(\upphi_{\times}(a))
(\uppsi_{\times}(b)).
\end{equation*}
\end{proposition}
\begin{proof}
Since Prp.\ref{12051732}
\end{proof}
Next let $\upgamma$ be the $\mc{A}(U)$-linearization via the universal property of tensor product of the following 
$\mc{A}(U)$-bilinear map 
\begin{equation*}
\mr{Alt}^{\bullet}(U,M)\times\mr{Alt}^{\bullet}(U,M)\to\mr{Alt}^{\bullet}(U,M),\,(z,\zeta)\mapsto z\wedge\zeta,
\end{equation*}
where $\wedge$ is the standard wedge product on $\mr{Alt}^{\bullet}(U,M)$. Then Cor.\ref{12020811} legitimates the following 
definition. 
\begin{definition}
[\textbf{$(r,p,G,H)$-extension of the wedge product}]
\label{07191816}
Define 
\begin{multline*}
\wedge_{r,p}^{U,M,G,H}\in
\\
\mr{bil}_{\mc{A}(U)}\left(
\mr{Alt}^{\bullet}(U,M;\mh{Q}[p]_{U}^{M}(H)),\mr{Alt}^{\bullet}(U,M;\mh{Q}[r]_{U}^{M}(G));
\mr{Alt}^{\bullet}(U,M;\mh{Q}[x(r,p)]_{U}^{M}(G\widehat{\otimes}H))\right),
\end{multline*}
such that 
\begin{equation*}
\wedge_{r,p}^{U,M,G,H}\coloneqq\upgamma_{r,p}^{G,H}.
\end{equation*}
By abuse of language we call $\wedge_{r,p}^{U,M,G,H}$ the \emph{$(r,p,G,H)$-extension of the wedge product}.
We convein for every $a\in\mr{Alt}^{\bullet}(U,M;\mh{Q}[p]_{U}^{M}(H))$ and $b\in\mr{Alt}^{\bullet}(U,M;\mh{Q}[r]_{U}^{M}(G))$ 
to let $a\wedge_{r,p}^{U,M,G,H}b$ denote $\wedge_{r,p}^{U,M,G,H}(a,b)$. 
\end{definition}
\begin{remark}
\label{12011040wedge}
By Rmk.\ref{12011040} we have that 
\begin{multline*}
(\forall(g,a)\in\mc{B}^{p}(U,H)\times\mr{Alt}^{\bullet}(U,M))(\forall(f,\zeta)\in\mh{Q}[r]_{U}^{M}(G)\times\mr{Alt}^{\bullet}(U,M))
\\
\left((g\otimes a)\wedge_{r,p}^{U,M,G,H}(f\otimes\zeta)=\left(\uptau_{r}^{p,G,H}(f\otimes g)\right)\otimes(a\wedge\zeta)\right).
\end{multline*}
\end{remark}
\begin{proposition}
[\textbf{Pushforward commutes with $(r,p,G,H)$-extension of the wedge product}]
\label{09290547}
Assume there exist $\K$-linear subspaces $X_{1}$ of $G_{1}^{\ast}$ and $Y_{1}$ of $H_{1}^{\ast}$ such that the topology on $G_{1}$ and 
$H_{1}$ are $\sigma(G_{1},X_{1})$ and $\sigma(H_{1},Y_{1})$ respectively.Let $\ep\in\mr{Alt}^{\bullet}(U,M;\mh{Q}[p]_{U}^{M}(H))$ and 
$\theta\in\mr{Alt}^{\bullet}(U,M;\mh{Q}[r]_{U}^{M}(G))$. If $\uppsi\in\mc{L}(G,G_{1})$, and $\upphi\in\mc{L}(H,H_{1})$, then 
\begin{equation*}
(\uppsi\widehat{\otimes}\upphi)_{\times}(\ep\wedge_{r,p}^{G,H}\theta)=\upphi_{\times}(\ep)\wedge_{r,p}^{G_{1},H_{1}}\uppsi_{\times}(\theta).
\end{equation*}
\end{proposition}  
\begin{proof}
Since Prp.\ref{12051732}.
\end{proof}
\begin{proposition}
[\textbf{Pullback commutes with $(r,p,G,H)$-extension of the wedge product}]
\label{11300106}
For every $\ep\in\mr{Alt}^{\bullet}(U,M;\mh{Q}[p]_{U}^{M}(H))$ and $\theta\in\mr{Alt}^{\bullet}(U,M;\mh{Q}[r]_{U}^{M}(G))$ we have 
\begin{equation*}
F^{\times,x(r,p),G\widehat{\otimes}H}
(\ep\wedge_{r,p}^{G,H}\theta)
=
F^{\times,p,H}(\ep)\wedge_{r,p}^{G,H}F^{\times,r,G}(\theta)
\end{equation*}
\end{proposition}
\begin{proof}
For every $(g,a)\in\mc{B}^{p}(U,H)\times\mr{Alt}^{\bullet}(U,M)$ and every 
$(f,\zeta)\in\mh{Q}[r]_{U}^{M}(G)\times\mr{Alt}^{\bullet}(U,M))$ we have 
\begin{equation*}
\begin{aligned}
F^{\times,x(r,p),G\widehat{\otimes}H}\bigl((g\otimes a)\wedge_{r,p}^{U,M,G,H}(f\otimes\zeta)\bigr)
&=
\left(\uptau_{r}^{p,G,H}(f\otimes g)\circ F\right)\otimes F^{\ast}(a\wedge\zeta)
\\
&=
\left(\uptau_{r}^{p,G,H}((f\circ F)\otimes(g\circ F))\right)\otimes(F^{\ast}(a)\wedge F^{\ast}(\zeta))
\\
&=
\bigl(F^{\times,p,H}(g\otimes a)\wedge_{r,p}^{W,N,G,H}F^{\times,r,G}(f\otimes\zeta)\bigr),
\end{aligned}
\end{equation*}
where the 1th and 3th equalities follow since Rmk.\ref{12011040wedge} and \eqref{09180901}, while 
the 2th one follows since the well-know property $F^{\ast}(a\wedge\zeta)=F^{\ast}(a)\wedge F^{\ast}(\zeta)$.
Then the statement follows since the following general property of tensor product: if $A$ is a commutative ring, $E$ and $F$ are
$A$-modules, then the set $\{u\otimes v\,\vert\,(u,v)\in E\times F\}$ is a system of generators of the $A$-module $E\otimes_{A}F$.
\end{proof}
Next we shall use Cor. \ref{08262041} and Convention \ref{09241142} to define the differential 
of $G$-valued differentiable alternating tensors defined on a chart of $M$.
\begin{definition}
[\textbf{Differential of $G$-valued $p$-times continuously differentiable alternating tensor fields on a chart}]
\label{08281540}
If $(U,\phi)$ is a chart of $M$ and $p\in\Z_{>0}\cup\{+\infty\}$, define 
\begin{equation*}
d_{p}^{U,G}:\mr{Alt}[p]^{\bullet}(U,M;G)\to\mr{Alt}[p-1]^{\bullet}(U,M;G),
\end{equation*}
as the unique $\R$-linear map such that for every $\theta\in\bigcup_{k\in[0,n]}\mr{Alt}[p]^{k}(U,M;G)$ we have 
\begin{equation*}
d_{p}^{U,G}\theta\coloneqq\sum_{I\in M(\mr{ord}\theta,n,<)}\sum_{j=1}^{n}
\partial_{j}^{\phi,G}(f_{I})\otimes(dx_{j}^{\phi}\wedge\mc{E}_{dx^{\phi}}(I)),
\end{equation*}
where $f:M(\mr{ord}\theta,n,<)\to\mc{B}^{p}(U,G)$ is the unique map in the decomposition of $\theta$ established in 
Cor. \ref{08262041} for the special case $\mh{M}=\mc{B}^{p}(U,G)$.
\end{definition}  
Notice that the above definition is well-set as an application of \cite[II.12 Prp.6]{BourA1}.
\begin{theorem}
[\textbf{Differential of $G$-valued $p$-times continuously differentiable alternating tensor fields}]
\label{08281401}
Let $p\in\Z_{>0}\cup\{+\infty\}$ and $\{U_{\alpha}\}_{\alpha\in D}$ be a collection of domains of charts of $M$ which are subsets of 
$U$ covering $U$. Thus there exists a unique 
\begin{equation*}
d_{p}^{[U,G]}:\mr{Alt}[p]^{\bullet}(U,M;G)\to\mr{Alt}[p-1]^{\bullet}(U,M;G),
\end{equation*}
called the \emph{exterior $G$-differential on $p$-times differentiable alternating tensors}, such that for all $k\in\Z_{\geq 0}$ 
we have $d_{p}^{[U,G]}(\mr{Alt}[p]^{k}(U,M;G))\subseteq\mr{Alt}[p-1]^{k+1}(U,M;G)$ and 
\begin{equation*}
\xymatrix{
\mr{Alt}[p]^{\bullet}(U,M;G)
\ar[rr]^{d_{p}^{[U,G]}}
\ar[dd]_{(\imath_{U_{\alpha}}^{U})^{\times,p,G}}
& &
\mr{Alt}[p-1]^{\bullet}(U,M;G)
\ar[dd]^{(\imath_{U_{\alpha}}^{U})^{\times,p-1,G}}
\\
& &
\\
(\imath_{U_{\alpha}}^{U})_{\star}(\mr{Alt}[p]^{\bullet}(U_{\alpha},M;G))
\ar[rr]^{(\imath_{U_{\alpha}}^{U})_{\star}(d_{p}^{U_{\alpha},G})}
& &
(\imath_{U_{\alpha}}^{U})_{\star}(\mr{Alt}[p-1]^{\bullet}(U_{\alpha},M;G)).}
\end{equation*}
\end{theorem}
\begin{proof}
Since Prp.\ref{08281918int} and since the gluing lemma via charts that is legitimate by Rmk. \ref{09261218diff}
where the compatibility is ensured by the uniqueness of the decomposition established in Cor. \ref{08262041}, by 
Prp.\ref{11272114} and by the trivial fact that
$\imath_{U_{\alpha}\cap U_{\beta}}^{M}=\imath_{U_{\alpha}}^{M}\circ\imath_{U_{\alpha}\cap U_{\beta}}^{U_{\alpha}}=
\imath_{U_{\beta}}^{M}\circ\imath_{U_{\alpha}\cap U_{\beta}}^{U_{\beta}}$.
\end{proof}
\begin{proposition}
[\textbf{Differential of $G$-valued $p$-times continuously differentiable forms}]
\label{09111902}
Let $p\in\Z_{>0}\cup\{+\infty\}$, then there exists a unique map 
\begin{equation*}
\widehat{d}_{p}^{[U,G]}:\Omega[p]^{\bullet}(U,M;G)\to\Omega[p-1]^{\bullet}(U,M;G),
\end{equation*}
such that 
\begin{equation*}
(\forall\theta\in\mr{Alt}[p]^{\bullet}(U,M;G))
(\widehat{d}_{p}^{[U,G]}(\mf{t}^{\mc{B}^{p}(U,G),U}(\theta))=\mf{t}^{\mc{B}^{p-1}(U,G),U}(d_{p}^{[U,G]}(\theta)))
\end{equation*}
and called \emph{exterior $G$-differential on $p$-times differentiable forms}. 
\end{proposition}
\begin{proof}
Since Thm. \ref{08281401} and Prp. \ref{08281918int}.
\end{proof}
\begin{remark}
[$d_{p}^{[U,G]}$ is $\R$-linear]
\label{11271815}
If $p\in\Z_{>0}\cup\{+\infty\}$, then we easily find that 
\begin{equation*}
d_{p}^{[U,G]}\in\mr{mor}_{\R-\mr{mod}}
\left(\uprho_{\star}^{U}(\mr{Alt}[p]^{\bullet}(U,M;G)),\uprho_{\star}^{U}(\mr{Alt}[p-1]^{\bullet}(U,M;G))\right)
\end{equation*}
and
\begin{equation*}
\widehat{d}_{p}^{[U,G]}\in\mr{mor}_{\R-\mr{mod}}
\left(\uprho_{\star}^{U}(\Omega[p]^{\bullet}(U,M;G)),\uprho_{\star}^{U}(\Omega[p-1]^{\bullet}(U,M;G))\right).
\end{equation*}
\end{remark}
\begin{proposition}
[Pushforward commutes with $d_{p}^{[U,G]}$]
\label{08281542}
If $p\in\Z_{>0}\cup\{+\infty\}$ and $\uppsi\in\mc{L}(G,H)$, then $\uppsi_{\times}\circ d_{p}^{[U,G]}=d_{p}^{[U,H]}\circ\uppsi_{\times}$ 
and $\uppsi_{\times}\circ\widehat{d}_{p}^{[U,G]}=\widehat{d}_{p}^{[U,H]}\circ\uppsi_{\times}$.
\end{proposition}
\begin{proof}
The statement follows by \eqref{11271705}, since Thm. \ref{08281401} and by 
\begin{equation}
\label{12042052}
\partial_{i}^{\phi,H}\circ\uppsi_{\ast}=\uppsi_{\ast}\circ\partial_{i}^{\phi,G},
\end{equation}
which follows since $\uppsi$ is linear and continuous.
\end{proof}
\begin{remark}
\label{11281505}
Let $p\in\Z_{\geq 0}\cup\{+\infty\}$, $r\in\{mc,mloc\}\cup\Z_{\geq 0}\cup\{+\infty\}$  
and assume that there exist two possibly infinite and uncountable families 
$\mr{G}=\{G_{i}\}_{i\in I}$  and $\mr{H}=\{H_{j}\}_{j\in J}$ of Hausdorff locally convex spaces over $\R$ such that 
\begin{equation*}
\begin{cases}
G_{\R}=\bigoplus_{i\in I}G_{i},
\\
H_{\R}=\bigoplus_{j\in J}H_{j},
\end{cases}
\end{equation*}
topological direct sum of the family $\mr{G}$ and $\mr{H}$ respectively. Let $k\in I$ and $j\in J$. 
Then since Prp.\ref{08281542} and \eqref{11271705} both applied to $G=G_{\R}$, since Prp.\ref{12051823} and Prp.\ref{09290547} 
both applied to $G=G_{\R}$, $H=H_{\R}$, $G_{1}=G_{k}$ and $H_{1}=H_{j}$, we have for every 
$\ep\in\mr{Alt}^{\bullet}(U,M;\mh{Q}[p]_{U}^{M}(H_{\R}))$, $\theta\in\mr{Alt}^{\bullet}(U,M;\mh{Q}[r]_{U}^{M}(G_{\R}))$;
$a\in\mc{B}^{p}(U,H_{\R})\otimes_{\mc{A}(U)}\Gamma(U,TM)$ and $b\in\mr{Alt}^{\bullet}(U,M;\mh{Q}[r]_{U}^{M}(G_{\R}))$,
and $y\in\mh{Q}[r]_{U}^{M}(G_{\R})\otimes_{\mc{A}(U)}\mf{T}_{\bullet}^{0}(U,M)$ that 
\begin{equation*}
\begin{aligned}
(\Pr_{k}^{G_{\R}})_{\times}\circ d_{p}^{[U,G_{\R}]}
&=
d_{p}^{[U,G_{k}]}\circ(\Pr_{k}^{G_{\R}})_{\times}
\\
(\Pr_{k}^{G_{\R}}\widehat{\otimes}\Pr_{j}^{H_{\R}})_{\times}(\ep\wedge_{r,p}^{G_{\R},H_{\R}}\theta)
&=
(\Pr_{j}^{H_{\R}})_{\times}(\ep)\wedge_{r,p}^{G_{k},H_{j}}(\Pr_{k}^{G_{\R}})_{\times}(\theta)
\\
(\Pr_{k}^{G_{\R}}\widehat{\otimes}\Pr_{j}^{H_{\R}})_{\times}\bigl(\mf{i}_{r,p}^{U,M,G_{\R},H_{\R}}(a)(b)\bigr)
&=
\mf{i}_{r,p}^{U,M,G_{k},H_{j}}((\Pr_{j}^{H_{\R}})_{\times}(a))((\Pr_{k}^{G_{\R}})_{\times}(b))
\\
F_{\star}((\Pr_{k}^{G_{\R}})_{\times})(\overset{\times,r,G_{\R}}{F}(y))
&=
\overset{\times,r,G_{k}}{F}((\Pr_{k}^{G_{\R}})_{\times}(y)).
\end{aligned}
\end{equation*}
Next since the well-known relations between $G_{\R}^{\prime}$ and $G^{\prime}$ see for instance \cite[II.65(1)]{EVT} if 
$r\in\{mc,mloc\}$, while directly by definition if $r\in\Z_{\geq 0}\cup\{+\infty\}$ we have that 
\begin{equation*}
\mh{Q}[r]_{U}^{M}(G)=\mh{Q}[r]_{U}^{M}(G_{\R}).
\end{equation*}
Therefore, 
\begin{equation*}
\begin{cases}
d_{p}^{[U,G]}=d_{p}^{[U,G_{\R}]}
\\
\wedge_{r,p}^{G,H}=\wedge_{r,p}^{G_{\R},H_{\R}}
\\
\mf{i}_{r,p}^{U,M,G,H}=\mf{i}_{r,p}^{U,M,G_{\R},H_{\R}}
\\
F^{\times,r,G}=F^{\times,r,G_{\R}}
\end{cases} 
\end{equation*}
moreover if $\psi\in\mc{L}(G,H)$, then $\psi_{\times,p,U,M,G}=(\updelta_{\star}^{\K}(\psi))_{\times,p,U,M,G_{\R}}$.
Then we conclude that for every 
$\ep\in\mr{Alt}^{\bullet}(U,M;\mh{Q}[p]_{U}^{M}(H))$, $\theta\in\mr{Alt}^{\bullet}(U,M;\mh{Q}[r]_{U}^{M}(G))$
$a\in\mc{B}^{p}(U,H)\otimes_{\mc{A}(U)}\Gamma(U,TM)$ and $b\in\mr{Alt}^{\bullet}(U,M;\mh{Q}[r]_{U}^{M}(G))$,
and $y\in\mh{Q}[r]_{U}^{M}(G)\otimes_{\mc{A}(U)}\mf{T}_{\bullet}^{0}(U,M)$ we have 
\begin{equation}
\label{12042150}
\begin{aligned}
(\Pr_{k}^{G_{\R}})_{\times}\circ d_{p}^{[U,G]}
&=
d_{p}^{[U,G_{k}]}\circ(\Pr_{k}^{G_{\R}})_{\times}
\\
(\Pr_{k}^{G_{\R}}\widehat{\otimes}\Pr_{j}^{H_{\R}})_{\times}(\ep\wedge_{r,p}^{G,H}\theta)
&=
(\Pr_{j}^{H_{\R}})_{\times}(\ep)\wedge_{r,p}^{G_{k},H_{j}}(\Pr_{k}^{G_{\R}})_{\times}(\theta)
\\
(\Pr_{k}^{G_{\R}}\widehat{\otimes}\Pr_{j}^{H_{\R}})_{\times}\bigl(\mf{i}_{r,p}^{U,M,G,H}(a)(b)\bigr)
&=
\mf{i}_{r,p}^{U,M,G_{k},H_{j}}((\Pr_{j}^{H_{\R}})_{\times}(a))((\Pr_{k}^{G_{\R}})_{\times}(b))
\\
F_{\star}((\Pr_{k}^{G_{\R}})_{\times})(\overset{\times,r,G}{F}(y))
&=
\overset{\times,r,G_{k}}{F}((\Pr_{k}^{G_{\R}})_{\times}(y)).
\end{aligned}
\end{equation}
A very simple but fundamental particular case is given by the direct sum $\C_{\R}=\R^{2}$, more exactly 
$G=H=\C$, $I=J=\{1,2\}$ and $G_{k}=H_{k}=\R$, $k\in I$.
\end{remark}
\begin{definition}
Define $\alpha_{G}$ as the unique linear and continuous map such that the following diagram is commutative, 
\begin{equation*}
\xymatrix{
G\otimes\K
\ar[rr]^{\alpha_{G}}
& &
G
\\
& &
\\
G\times\K
\ar[uu]^{\otimes}
\ar[rruu]_{(u,\lambda)\mapsto\lambda u}
& &
}
\end{equation*}
define in addition the linear and continuous map $\beta_{G}:G\to G\otimes\K$, $u\mapsto u\otimes 1$.
Furthermore define $\mf{z}_{H,G}:H\otimes G\to G\otimes H$, to be the flip map namely 
the unique linearization of the $\K$-bilinear and continuous map $H\times G\to G\otimes H$, $(a,b)\mapsto b\otimes a$.
Finally define $\gamma_{G}\coloneqq\alpha_{G}\circ\mf{z}_{\K,G}$ and $\delta_{G}\coloneqq\mf{z}_{G,\K}\circ\beta_{G}$, 
\end{definition}
\begin{lemma}
\label{12101042}
The map $\mf{z}_{H,G}$ is a well-defined linear and continuous isomorphism whose inverse equals $\mf{z}_{G,H}$. 
Furthermore $\alpha_{G}$, $\beta_{G}$, $\gamma_{G}$, $\delta_{G}$ are well-defined linear and continuous isomorphisms such that 
$\alpha_{G}^{-1}=\beta_{G}$ and $\gamma_{G}^{-1}=\delta_{G}$.
\end{lemma}
\begin{proof}
The map $q:H\times G\to G\otimes H$, $(a,b)\mapsto b\otimes a$ is clearly $\K$-bilinear moreover but it is also continuous
since $\otimes:G\times H\to G\otimes H$ it is so, then by construction of the projective topology on the tensor product, the 
linearization of the map $q$ is continuous. Next trivially $\mf{z}_{G,H}$ is the inverse of $\mf{z}_{H,G}$. Similarly $\alpha_{G}$ 
is well-defined since $(u,\lambda)\mapsto\lambda u$ is $\K$-bilinear and continuous, while $\beta_{G}$ it is so since 
$\otimes:G\times\K\to G\otimes\K$ is $\K$-bilinear and continuous. Clearly $\alpha_{G}^{-1}=\beta_{G}$ and then 
$\gamma_{G}^{-1}=\delta_{G}$ since for what above proven $\mf{z}_{G,\K}$ is a linear continuous isomorphism whose inverse equals 
$\mf{z}_{\K,G}$.
\end{proof} 
\begin{lemma}
\label{12061627}
The set $\{\phi\widehat{\otimes}\mr{Id}_{\K}\,\vert\,\phi\in G^{\prime}\}$ separates the points of $G\widehat{\otimes}\K$.
\end{lemma}
\begin{proof}
Since Lemma \ref{12101042} we deduce that $G\otimes\K$ and $G$ are topologically isomorphic
and that $\widehat{\alpha}_{G}$ and $\widehat{\beta}_{G}$ are linear and continuous isomorphisms one inverse of the other, then 
in particular $G\widehat{\otimes}\K$ and $\widehat{G}$ are topologically isomorphic. 
Next let $\phi\in G^{\prime}$. We have that the following diagram is commutative in the category $\mr{tls}(\K)$ of linear 
topological spaces over $\K$ and linear and continuous maps
\begin{equation*}
\xymatrix{
G\otimes\K
\ar[rr]^{\phi\otimes\mr{Id}_{\K}}
\ar[dd]^{\alpha_{G}}
& &
\K\otimes\K
\ar[dd]^{\alpha_{\K}}
\\
& &
\\
G
\ar[rr]^{\phi}
& &
\K,
}
\end{equation*}
and then also the next diagram is commutative in the category $\mr{tls}(\K)$ 
\begin{equation*}
\xymatrix{
G\widehat{\otimes}\K
\ar[rr]^{\phi\widehat{\otimes}\mr{Id}_{\K}}
\ar[dd]^{\widehat{\alpha}_{G}}
& &
\K\widehat{\otimes}\K
\ar[dd]^{\widehat{\alpha}_{\K}}
\\
& &
\\
\widehat{G}
\ar[rr]^{\widehat{\phi}}
& &
\K.
}
\end{equation*}
Now $(\widehat{G})^{\prime}$ separates the points of $G$ since $G$ is Hausdorff, while 
$(\widehat{G})^{\prime}=\{\widehat{\phi}\,\vert\,\phi\in G^{\prime}\}$ by the Hahn-Banach theorem for locally convex spaces, 
finally as above proven $\widehat{\alpha}_{G}$ and $\widehat{\alpha}_{\K}$ are linear and continuous isomorphisms. 
Therefore the statement follows by the above commutative diagram. 
\end{proof}
\begin{lemma}
\label{12061953}
Let $r\in\{mc,mloc\}\cup\Z_{\geq 0}\cup\{+\infty\}$ and $p\in\Z_{\geq 0}\cup\{+\infty\}$.
If $r\in\{mc,mloc\}$, then assume that there exist a $\K$-linear subspace $X$ of $G^{\ast}$ such that the topology on $G$ is 
$\sigma(G,X)$. Furthermore let $f\in\mh{Q}[r]_{U}^{M}(G)$ and $g\in\mc{B}^{p}(U,\K)$. 
Then, by letting $(g\cdot f)(x)\coloneqq g(x)f(x)$ for every $x\in U$, we have 
\begin{equation*}
(\widehat{\alpha}_{G}\circ\uptau_{r}^{p,G,\K})(f\otimes g)=g\cdot f.
\end{equation*}
\end{lemma}
\begin{proof}
By Cor.\ref{09260817} and by construction of the map $\alpha_{G}$.
\end{proof}
\begin{definition}
\label{12062042}
Let $(g,T)\mapsto g\cdot T$ the $\mc{A}(U)$-action on $\mh{N}$. 
Define $\mf{a}_{\mh{N}}$ be the unique $\mc{A}(U)$-linear morphism such that the following diagram in the category 
$\mc{A}(U)-\mr{mod}$ is commutative
\begin{equation*}
\xymatrix{
\mc{A}(U)\otimes_{\mc{A}(U)}\mh{N}
\ar[rr]^{\mf{a}_{\mh{N}}}
& &
\mh{N}
\\
& &
\\
\mc{A}(U)\times\mh{N}
\ar[uu]^{\otimes}
\ar[rruu]_{(g,T)\mapsto g\cdot T}.
& &
}
\end{equation*}
Furtheremore define $\mf{b}_{\mh{N}}:\mh{N}\to\mc{A}(U)\otimes_{\mc{A}(U)}\mh{N}$, $T\mapsto\un_{U}\otimes T$.
\end{definition}
Immediately we have that 
\begin{lemma}
\label{12081849}
$\mf{b}_{\mh{N}}$ is a $\mc{A}(U)$-module morphism moreover $\mf{a}_{\mh{N}}$ and $\mf{b}_{\mh{N}}$ are each the inverse of the other.
\end{lemma}
\begin{proposition}
\label{12062000}
Let $r\in\{mc,mloc\}\cup\Z_{\geq 0}\cup\{+\infty\}$ and $p\in\Z_{\geq 0}\cup\{+\infty\}$.
If $r\in\{mc,mloc\}$, then assume that there exists a $\K$-linear subspace $X$ of $G^{\ast}$ such that the topology on $G$ is 
$\sigma(G,X)$. Then
\begin{multline}
\label{12062000a}
(\forall(g,x)\in\mc{B}^{p}(U,\K)\times\Gamma(U,TM))(\forall(f,\zeta)\in\mh{Q}[r]_{U}^{M}(G)\times\mr{Alt}^{\bullet}(U,M))
\\
\left(\bigl((\widehat{\alpha}_{G})_{\times}\circ(\mf{i}_{r,p}^{U,M,G,\K}(g\otimes x))\bigr)(f\otimes\zeta)
=
(g\cdot f)\otimes(\mf{i}_{x}(\zeta))
\right),
\end{multline}
in addition 
\begin{equation}
\label{12062000b}
g\in\mc{A}(U)\Rightarrow(g\cdot f)\otimes(\mf{i}_{x}(\zeta))=(f\otimes(\mf{i}_{gx}(\zeta))=f\otimes(\mf{i}_{x}(g\zeta)).
\end{equation}
Finally if $\K=\R$, then 
\begin{multline}
\label{12062000c}
(\forall y\in\mc{A}(U)\otimes_{\mc{A}(U)}\Gamma(U,TM))(\forall\gamma\in\mh{Q}[\infty]_{U}^{M}(\R)\otimes_{\mc{A}(U)}
\mr{Alt}^{\bullet}(U,M))
\\
\left(\bigl(\mf{a}_{\mr{Alt}^{\bullet}(U,M)}\circ(\widehat{\alpha}_{\R})_{\times}\circ(\mf{i}_{r,p}^{U,M,\R,\R}(y))\bigr)(\gamma)
=
\mf{i}_{\mf{a}_{\Gamma(U,TM)}(y)}(\mf{a}_{\mr{Alt}^{\bullet}(U,M)}(\gamma)).
\right).
\end{multline}
\end{proposition}
\begin{proof}
\eqref{12062000a} follows since Rmk.\ref{12011040ins} and Lemma \ref{12061953}. 
Next \eqref{12062000b} follows by \eqref{12062000a} applied for $p=\infty$, since clearly $\mc{A}(U)=\mc{B}^{\infty}(U,\R)$, while 
$\mc{A}(U)\subset\mc{B}^{\infty}(U,\C)$, and by the $\mc{A}(U)$-bilinearity of the map $(x,\zeta)\mapsto\mf{i}_{x}(\zeta)$.
Finally since \eqref{12062000a}, by Def.\ref{12062042} and by the $\mc{A}(U)$-bilinearity of the map 
$(x,\zeta)\mapsto\mf{i}_{x}(\zeta)$ we obtain 
\begin{multline}
\label{12091849}
(\forall(g,x)\in\mc{A}(U)\times\Gamma(U,TM))(\forall(f,\zeta)\in\mh{Q}[\infty]_{U}^{M}(\R)\times\mr{Alt}^{\bullet}(U,M))
\\
\left(\bigl(\mf{a}_{\mr{Alt}^{\bullet}(U,M)}\circ(\widehat{\alpha}_{\R})_{\times}\circ(\mf{i}_{r,p}^{U,M,\R,\R}(g\otimes x))\bigr)
(f\otimes\zeta)
=
gf\mf{i}_{x}(\zeta)
=
\mf{i}_{gx}(f\zeta)
\right),
\end{multline}
\eqref{12062000c} follows by \eqref{12091849} and since the following general property of tensor 
product: if $A$ is a commutative ring, $E$ and $F$ are $A$-modules, then the set $\{x\otimes y\,\vert\,(x,y)\in E\times F\}$ is 
a system of generators of the $A$-module $E\otimes_{A}F$.
\end{proof}
\begin{proposition}
\label{12062133}
Let $r\in\{mc,mloc\}\cup\Z_{\geq 0}\cup\{+\infty\}$ and $p\in\Z_{\geq 0}\cup\{+\infty\}$.
If $r\in\{mc,mloc\}$, then assume that there exists a $\K$-linear subspace $X$ of $G^{\ast}$ such that the topology on $G$ is 
$\sigma(G,X)$. Then
\begin{multline}
\label{12062133a}
(\forall(g,x)\in\mc{B}^{p}(U,\K)\times\mr{Alt}^{\bullet}(U,M))(\forall(f,\zeta)\in\mh{Q}[r]_{U}^{M}(G)\times\mr{Alt}^{\bullet}(U,M))
\\
\left((\widehat{\alpha}_{G})_{\times}\bigl((g\otimes x)\wedge_{r,p}^{U,M,G,\K}(f\otimes\zeta)\bigr)
=(g\cdot f)\otimes(x\wedge\zeta)\right),
\end{multline}
in addition 
\begin{equation}
\label{12062133b}
g\in\mc{A}(U)\Rightarrow(g\cdot f)\otimes(x\wedge\zeta)=f\otimes(gx\wedge\zeta)=f\otimes(x\wedge g\zeta).
\end{equation}
Finally if $\K=\R$, then 
\begin{multline}
\label{12062133c}
(\forall y\in\mc{A}(U)\otimes_{\mc{A}(U)}\mr{Alt}^{\bullet}(U,M))
(\forall\gamma\in\mh{Q}[\infty]_{U}^{M}(\R)\otimes_{\mc{A}(U)}\mr{Alt}^{\bullet}(U,M))
\\
\left((\mf{a}_{\mr{Alt}^{\bullet}(U,M)}\circ(\widehat{\alpha}_{\R})_{\times})
\bigl(y\wedge_{r,p}^{U,M,\R,\R}\gamma\bigr)=\mf{a}_{\mr{Alt}^{\bullet}(U,M)}(y)\wedge\mf{a}_{\mr{Alt}^{\bullet}(U,M)}(\gamma)\right).
\end{multline}
\end{proposition}
\begin{proof}
\eqref{12062133a} follows since Rmk.\ref{12011040wedge} and Lemma \ref{12061953}. 
Next \eqref{12062133b} follows by \eqref{12062133a} applied for $p=\infty$, since clearly $\mc{A}(U)=\mc{B}^{\infty}(U,\R)$, while 
$\mc{A}(U)\subset\mc{B}^{\infty}(U,\C)$, and by the $\mc{A}(U)$-bilinearity of the map $(x,\zeta)\mapsto x\wedge\zeta$.
Finally 
since \eqref{12062133a}, by Def.\ref{12062042} and by the $\mc{A}(U)$-bilinearity of the map $(x,\zeta)\mapsto x\wedge\zeta$ 
we obtain
\begin{multline}
\label{12091853}
(\forall(g,x)\in\mc{A}(U)\times\mr{Alt}^{\bullet}(U,M))(\forall(f,\zeta)\in\mh{Q}[\infty]_{U}^{M}(\R)\times\mr{Alt}^{\bullet}(U,M))
\\
\left((\mf{a}_{\mr{Alt}^{\bullet}(U,M)}\circ(\widehat{\alpha}_{\R})_{\times})
\bigl((g\otimes x)\wedge_{r,p}^{U,M,\R,\R}(f\otimes\zeta)\bigr)=gf(x\wedge\zeta)=(gx\wedge f\zeta)
\right).
\end{multline}
\eqref{12062133c} follows by \eqref{12091853} and since the following general property of tensor 
product: if $A$ is a commutative ring, $E$ and $F$ are $A$-modules, then the set $\{x\otimes y\,\vert\,(x,y)\in E\times F\}$ is 
a system of generators of the $A$-module $E\otimes_{A}F$.
\end{proof}
\begin{theorem}
[\textbf{Properties of the differential}]
\label{09211057}
If $p=q=+\infty$, then
\begin{enumerate}
\item
For all $\omega\in\mr{Alt}[p]^{\bullet}(U,M;\K)$ and $\eta\in\mr{Alt}[q]^{\bullet}(U,M;G)$  
\begin{equation*}
d_{\min\{p,q\}}^{[U,G\widehat{\otimes}\K]}(\omega\wedge_{q,p}^{U,M,G,\K}\eta)
=
(d_{p}^{[U,\K]}\omega)\wedge_{q,p-1}^{U,M,G,\K}\eta+(-1)^{\mr{ord}(\omega)}\omega\wedge_{q-1,p}^{U,M,G,\K}d_{q}^{[U,G]}\eta;
\end{equation*}
\label{09211057st2}
\item
$d_{p}^{[U,G]}\circ d_{p+1}^{[U,G]}=\ze$;
\label{09211057st3}
\end{enumerate}
Moreover 
\begin{equation}
\label{12061805}
(\forall\theta\in\mr{Alt}[p]^{\bullet}(U,M;G))
\bigl(F_{\star}(d_{p}^{[W,G]})(\overset{\times,p,G}{F}\theta)=\overset{\times,p-1,G}{F}(d_{p}^{[U,G]}\theta)\bigr).
\end{equation}
\end{theorem}
\begin{proof}
Let $\omega\in\mr{Alt}[p]^{\bullet}(U,M;\K)$, $\eta\in\mr{Alt}[q]^{\bullet}(U,M;G)$ and $\phi\in G^{\prime}$, then since 
Prp.\ref{09290547} applied to the positions $G_{1}=H=H_{1}=\K$ and 
to the elements of $\{\widehat{\phi\otimes\mr{Id}_{\K}}\,\vert\,\phi\in G^{\prime}\}$, and since Prp.\ref{08281542}
by replacing $G$ by $G\widehat{\otimes}\K$, $H$ by $\K\widehat{\otimes}\K$ and $\psi$ by $\phi\otimes\mr{Id}_{\K}$ where 
$\phi\in G^{\prime}$, we obtain
\begin{equation}
\label{12061827a}
(\phi\widehat{\otimes}\mr{Id}_{\K})_{\times}
\bigl(d_{\min\{p,q\}}^{[U,G\widehat{\otimes}\K]}(\omega\wedge_{q,p}^{U,M,G,\K}\eta)\bigr)
=
d_{\min\{p,q\}}^{[U,\K\widehat{\otimes}\K]}(\omega\wedge_{q,p}^{U,M,\K,\K}\phi_{\times}(\eta)),
\end{equation}
while since Prp.\ref{09290547} and Prp.\ref{08281542} with the above position we have
\begin{multline}
\label{12061827b}
(\phi\widehat{\otimes}\mr{Id}_{\K})_{\times}
\bigl((d_{p}^{[U,\K]}\omega)\wedge_{q,p-1}^{U,M,G,\K}\eta+(-1)^{\mr{ord}(\omega)}\omega\wedge_{q-1,p}^{U,M,G,\K}d_{q}^{[U,G]}\eta\bigl)
=
\\
(d_{p}^{[U,\K]}\omega)\wedge_{q,p-1}^{U,M,\K,\K}\phi_{\times}(\eta)+(-1)^{\mr{ord}(\omega)}\omega\wedge_{q-1,p}^{U,M,\K,\K}
d_{q}^{[U,\K]}\phi_{\times}(\eta).
\end{multline}
If $\K=\C$, then we need in addition to evaluate \eqref{12061827a} and \eqref{12061827b} over the following functional 
\begin{equation}
\label{12070102}
\begin{cases}
(\widehat{\uppsi})_{\times}
\\
\uppsi=(\Pr_{k}^{C_{\R}}\widehat{\otimes}\Pr_{j}^{C_{\R}})\in\mc{L}(\C_{\R}\widehat{\otimes}_{\R}\C_{\R},\R\widehat{\otimes}_{\R}\R),
\end{cases}
\end{equation}
for every $k,j\in\{1,2\}$.
Now statement \eqref{09211057st2} is satisfied in the special case $\K=\R$ and $G=\K$ since \eqref{12062133c}, 
Thm.\ref{08281401}, Def.\ref{08281540}, and since the well-known property of the standard exterior differential.
Therefore, 
if $\K=\R$, then statement \eqref{09211057st2} follows since Lemma \ref{12061627}, by \eqref{12061827a} and \eqref{12061827b}.
Instead 
if $\K=\C$, then statement \eqref{09211057st2} follows since Lemma \ref{12061627}, by \eqref{12061827a} and \eqref{12061827b} 
modified as above stated, and since \eqref{12042150} by replacing $G$ and $H$ by $\C$.
The proof of statement \eqref{09211057st3} runs similarly to that of statement \eqref{09211057st2}. 
Next let $\theta\in\mr{Alt}[p]^{\bullet}(U,M;G)$, then since Prp.\ref{08281542} and \eqref{11271705} by replacing $H$ by $\K$ 
we obtain
\begin{equation}
\label{12070100a}
\phi_{\times}\bigl(F_{\star}(d_{p}^{[W,G]})(\overset{\times,p,G}{F}\theta)\bigr)
=
d_{p}^{[W,\K]}(\overset{\times,p,\K}{F}\phi_{\times}(\theta))
\end{equation}
and 
\begin{equation}
\label{12070100b}
\phi_{\times}\bigl(\overset{\times,p-1,G}{F}(d_{p}^{[U,G]}\theta)\bigr)
=
\overset{\times,p-1,\K}{F}(d_{p}^{[U,\K]}\phi_{\times}(\theta)).
\end{equation}
If $\K=\C$, then we need in addition to evaluate \eqref{12070100a} and \eqref{12070100b} over the functional in \eqref{12070102}
for every $k,j\in\{1,2\}$.
Next by letting $d_{p}$ the usual exterior differential over $\mr{Alt}[p]^{\bullet}(U,M)$, then 
since Thm.\ref{08281401} and Def.\ref{08281540} we deduce that 
\begin{equation}
\label{12091753}
\mf{a}_{\mr{Alt}^{\bullet}(U,M)}\circ d_{p}^{[U,\R]}=d_{p}\circ\mf{a}_{\mr{Alt}^{\bullet}(U,M)}.
\end{equation}
Furthermore easily we see that 
\begin{equation*}
\mf{a}_{\mr{Alt}^{\bullet}(U,M)}\circ F^{\times,p,\R}=F^{\ast}\circ\mf{a}_{\mr{Alt}^{\bullet}(U,M)}.
\end{equation*}
Therefore \eqref{12061805} is satisfied in the special case $\K=\R$ and $G=\K$.
Thus, if $\K=\R$, then \eqref{12061805} follows since $G^{\prime}$ separates the points of $G$, by \eqref{12070100a} and 
\eqref{12070100b}.
Instaed if $\K=\C$, then \eqref{12061805} follows since $G^{\prime}$ separates the points of $G$, by \eqref{12070100a} and 
\eqref{12070100b} modified as above stated, and since \eqref{12042150} by replacing $G$ and $H$ by $\C$.
\end{proof}
\begin{corollary}
[\textbf{$H$-valued divergence}]
\label{12011515}
Let $\omega\in\mr{Alt}^{n}(U,M)$ be a volume form and $p\in\Z_{>0}\cup\{+\infty\}$. Thus 
$\mf{b}_{\mr{Alt}^{\bullet}(U,M)}(\omega)\in\mr{Alt}^{\bullet}(U,M;\mh{Q}[\infty]_{U}^{M}(\R))$. Furthermore there exists a unique map 
\begin{equation*}
\mr{div}_{\omega,p,H_{\R}}:\mc{B}^{p}(U,H_{\R})\otimes_{\mc{A}(U)}\Gamma(U,TM)\to\mc{B}^{p-1}(U,\R\widehat{\otimes}H_{\R}),
\end{equation*}
such that for every $y\in\mc{B}^{p}(U,H_{\R})\otimes_{\mc{A}(U)}\Gamma(U,TM)$ we have 
\begin{equation*}
\bigl(d_{p}^{[U,\R\widehat{\otimes}H_{\R}]}\circ(\mf{i}_{\infty,p}^{U,M,\R,H_{\R}}(y))\bigr)(\mf{b}_{\mr{Alt}^{\bullet}(U,M)}(\omega))
=(\mr{div}_{\omega,p,H_{\R}}(y))\otimes\omega.
\end{equation*}
Furthermore we have 
\begin{equation*}
\mr{div}_{\omega,p,H_{\R}}(y)
=
\bigl(\mf{f}_{\omega}^{p-1,U,M,\R\widehat{\otimes}H_{\R}}\circ d_{p}^{[U,\R\widehat{\otimes}H_{\R}]}\circ(\mf{i}_{\infty,p}^{U,M,\R,H_{\R}}(y))\bigr)
(\mf{b}_{\mr{Alt}^{\bullet}(U,M)}(\omega)).
\end{equation*}
\end{corollary}
\begin{proof}
The 1th part of the statement follows since Def.\ref{12062042} and the fact that by definition
$\mc{A}(U)=\mh{Q}[\infty]_{U}^{M}(\R)$. The 2th part of the statement follows since Def.\ref{12022149}, Thm.\ref{08281401} and 
Cor.\ref{12011455}.
\end{proof}
\begin{lemma}
\label{12091916}
Let $\omega\in\mr{Alt}^{n}(U,M)$ be a volume form, $z\in\mc{B}^{\infty}(U,H_{\R})\otimes_{\mc{A}(U)}\Gamma(U,TM)$ and 
$\phi\in(H_{\R})^{\prime}$. Then
\begin{equation*}
\left(\mf{a}_{\mr{Alt}^{\bullet}(U,M)}\circ(\widehat{\alpha}_{\R})_{\times}\circ(\mr{Id}_{\R}\widehat{\otimes}\phi)_{\times}\right)
\bigl((\mr{div}_{\omega,\infty,H_{\R}}(z))\otimes\omega\bigr)
=\mr{div}_{\omega}\bigl(\mf{a}_{\Gamma(U,TM)}(\phi_{\times}(z))\bigr)\omega.
\end{equation*}
\end{lemma}
\begin{proof}
Since Cor.\ref{12011515}, Prp.\ref{08281542} and Prp.\ref{12051823} we obtain 
\begin{equation*}
(\mr{Id}_{\R}\widehat{\otimes}\phi)_{\times}\bigl((\mr{div}_{\omega,\infty,H_{\R}}(z))\otimes\omega\bigr)
=
\bigl(d_{\infty}^{[U,\R\widehat{\otimes}\R]}\circ(\mf{i}_{\infty,\infty}^{U,M,\R,\R}(\phi_{\times}(z)))\bigr)(\mf{b}_{\mr{Alt}^{\bullet}(U,M)}(\omega)),
\end{equation*}
therefore since Prp.\ref{08281542} we have 
\begin{equation*}
\left((\widehat{\alpha}_{\R})_{\times}\circ
(\mr{Id}_{\R}\widehat{\otimes}\phi)_{\times}\right)\bigl((\mr{div}_{\omega,\infty,H_{\R}}(z))\otimes\omega\bigr)
=
\bigl(d_{\infty}^{[U,\R]}\circ(\widehat{\alpha}_{\R})_{\times}\circ
(\mf{i}_{\infty,\infty}^{U,M,\R,\R}(\phi_{\times}(z)))\bigr)(\mf{b}_{\mr{Alt}^{\bullet}(U,M)}(\omega)),
\end{equation*}
then we have 
\begin{multline*}
\left(\mf{a}_{\mr{Alt}^{\bullet}(U,M)}\circ(\widehat{\alpha}_{\R})_{\times}\circ(\mr{Id}_{\R}\widehat{\otimes}\phi)_{\times}\right)
\bigl((\mr{div}_{\omega,\infty,H_{\R}}(z))\otimes\omega\bigr)
=
\\
\bigl(d_{\infty}\circ\mf{a}_{\mr{Alt}^{\bullet}(U,M)}\circ(\widehat{\alpha}_{\R})_{\times}\circ
(\mf{i}_{\infty,\infty}^{U,M,\R,\R}(\phi_{\times}(z)))\bigr)(\mf{b}_{\mr{Alt}^{\bullet}(U,M)}(\omega))
=
\\ 
(d_{\infty}\circ\mf{i}_{\mf{a}_{\Gamma(U,TM)}(\phi_{\times}(z))})(\omega),
\end{multline*}
where the 1th equality follows by by \eqref{12091753}, while the 2th one follows by \eqref{12062000c} and since 
Lemma \ref{12081849}, then the statement follows by definition of usual divergence.
\end{proof}
\begin{notations}
If $g\in\mf{T}_{2}^{0}(M)$ is a pseudo-Riemannian metric on $M$, then we let $v_{g}$ denote the canonical volume form on $M$
induced by $g$. We tacitly assume the pseudo-Riemanniann manifold $(M,g)$ be oriented with orientation induced by $v_{g}$.
Furthermore we let $g[U]\in\mf{T}_{2}^{0}(U,M)$ be the pseudo-Riemannian metric on $U$ induced by $g$. 
Finally we let $\mr{div}_{g[U]}$ denote $\mr{div}_{v_{g[U]}}$ the divergence operator associated with the canonical volume form 
$v_{g[U]}$ on $U$ induced by $g[U]$. If $(\uppsi,U)$ is a chart of $M$, then we define 
\begin{equation*}
M^{\uppsi}(g):([1,n]\cap\Z)^{2}\ni(i,j)\mapsto\mf{r}^{U}(\mf{t}^{M}(g)\circ\imath_{U}^{M})(\partial_{i}^{\uppsi},\partial_{j}^{\uppsi}).
\end{equation*}
\end{notations}
\begin{lemma}
\label{12100851}
If $(U,\uppsi)$ be a chart of $M$, then  
\begin{equation*}
(\exists\,!\mf{g}^{\mh{M},\uppsi}:\mh{M}\otimes_{\mc{A}(U)}\Gamma(U,TM)\to\mh{M}^{n})
(\forall z\in\mh{M}\otimes_{\mc{A}(U)}\Gamma(U,TM))
\bigl(z=\sum_{i=1}^{n}\mf{g}_{i}^{\mh{M},\uppsi}(z)\otimes\partial_{i}^{\uppsi}\bigr).
\end{equation*}
\end{lemma}
\begin{proof}
$[1,n]\cap\Z\ni i\mapsto\partial_{i}^{\uppsi}$ is a basis of the $\mc{A}(U)$-module $\Gamma(U,TM)$ so, the statement follows since 
\cite[II.62 Cor.1]{BourA1}.
\end{proof}  
\begin{lemma}
\label{12101326}
Let $(U,\uppsi)$ be a chart of $M$, $r\in\{c,loc,mc,mloc\}\cup\Z_{\geq 0}\cup\{+\infty\}$ and $\uppi\in\mc{L}(G,H)$.
\begin{equation*}
\uppi_{\ast}\circ\mf{g}^{\mh{Q}[r]_{U}^{M}(G),\uppsi}=\mf{g}^{\mh{Q}[r]_{U}^{M}(H),\uppsi}\circ\uppi_{\times}.
\end{equation*}
\end{lemma}
\begin{proof}
Since Lemma \ref{12100851} applied to $z$ we have 
\begin{equation*}
\uppi_{\times}(z)
=
\sum_{i=1}^{n}\left((\uppi_{\ast}\circ\mf{g}_{i}^{\mh{Q}[r]_{U}^{M}(G),\uppsi})(z)\right)\otimes\partial_{i}^{\uppsi},
\end{equation*}
while since Lemma \ref{12100851} applied to $\uppi_{\times}(z)$ we have 
\begin{equation*}
\uppi_{\times}(z)
=
\sum_{i=1}^{n}\left((\mf{g}_{i}^{\mh{Q}[r]_{U}^{M}(H),\uppsi}\circ\uppi_{\times})(z)\right)\otimes\partial_{i}^{\uppsi},
\end{equation*}
then the statement follows since the uniqueness stated in Lemma \ref{12100851}.
\end{proof}
\begin{notations}
If $(U,\uppsi)$ be a chart of $M$, then let $\mf{g}^{G,\uppsi}\coloneqq\mf{g}^{\mc{B}^{\infty}(U,G),\uppsi}$ and 
$\mf{g}_{i}^{G,\uppsi}\coloneqq\Pr_{i}^{\mc{B}^{\infty}(U,G)^{n}}\circ\mf{g}^{G,\uppsi}$ 
moreover let $\mf{h}_{i}^{G,\uppsi}\coloneqq((\widehat{\delta_{G}}\circ\mf{j}_{G})_{\ast}\circ\mf{g}_{i}^{G,\uppsi})$
for any $i\in[1,n]\cap\Z$.
\end{notations}
\begin{theorem}
\label{12092034}
Let $g\in\mf{T}_{2}^{0}(M)$ be a pseudo-Riemannian metric of signature $s$, $(U,\uppsi)$ be a positively oriented chart of $M$, 
and $z\in\mc{B}^{\infty}(U,H_{\R})\otimes_{\mc{A}(U)}\Gamma(U,TM)$. Then we have 
\begin{equation*}
\mr{div}_{v_{g[U]},\infty,H_{\R}}(z)
=
\sum_{i=1}^{n}
\bigl(\partial_{i}^{\uppsi,\R\widehat{\otimes}H_{\R}}\circ\mf{h}_{i}^{H_{\R},\uppsi}\bigr)(z)
+
\frac{(-1)^{s}}{2}
\sum_{i=1}^{n}
\frac{\mf{h}_{i}^{H_{\R},\uppsi}(z)}{\mr{det}(M^{\uppsi}(g))}
\partial_{i}^{\uppsi}(\mr{det}(M^{\uppsi}(g))).
\end{equation*}
\end{theorem}
\begin{proof}
Let $D$ denote the right-hand side of the equality in the statement. Let $\phi\in(H_{\R})^{\prime}$ and $i\in[1,n]\cap\Z$. 
Since Lemma \ref{12101326} we have 
\begin{equation}
\label{12101655}
\bigl(\widehat{\alpha_{\R}}\circ(\mr{Id}_{\R}\widehat{\otimes}\phi)\bigr)_{\ast}\circ\mf{h}_{i}^{H_{\R},\uppsi}
=
\mf{g}_{i}^{\R,\uppsi}\circ
\bigl(\widehat{\alpha_{\R}}\circ(\mr{Id}_{\R}\widehat{\otimes}\phi)\circ\widehat{\delta_{H_{\R}}}\circ\mf{j}_{H_{\R}}\bigr)_{\times}.
\end{equation}
Next by direct calculation we obtain
\begin{equation}
\label{12101656}
\widehat{\alpha_{\R}}\circ(\mr{Id}_{\R}\widehat{\otimes}\phi)\circ\widehat{\delta_{H_{\R}}}\circ\mf{j}_{H_{\R}}
=
\phi
\end{equation}
moreover by \eqref{12042052} we have
\begin{equation}
\label{12101700}
\bigl(\widehat{\alpha_{\R}}\circ(\mr{Id}_{\R}\widehat{\otimes}\phi)\bigr)_{\ast}\circ\partial_{i}^{\uppsi,\R\widehat{\otimes}H_{\R}}\circ
\mf{h}_{i}^{H_{\R},\uppsi}
=
\partial_{i}^{\uppsi,\R}\circ\bigl(\widehat{\alpha_{\R}}\circ(\mr{Id}_{\R}\widehat{\otimes}\phi)\bigr)_{\ast}\circ\mf{h}_{i}^{H_{\R},\uppsi}.
\end{equation}
Thus, 
\begin{multline*}
\left(\mf{a}_{\mr{Alt}^{\bullet}(U,M)}\circ\bigl(\widehat{\alpha_{\R}}\circ(\mr{Id}_{\R}\widehat{\otimes}\phi)\bigr)_{\times}\right)
(D\otimes v_{g{U}})
=
\\
\sum_{i=1}^{n}
\left(
\bigl(\widehat{\alpha_{\R}}\circ(\mr{Id}_{\R}\widehat{\otimes}\phi)\bigr)_{\ast}\circ\partial_{i}^{\uppsi,\R\widehat{\otimes}H_{\R}}\circ
\mf{h}_{i}^{H_{\R},\uppsi}\right)(z)
+
\\
\frac{(-1)^{s}}{2}
\sum_{i=1}^{n}
\frac{\left(\bigl(\widehat{\alpha_{\R}}\circ(\mr{Id}_{\R}\widehat{\otimes}\phi)\bigr)_{\ast}\circ\mf{h}_{i}^{H_{\R},\uppsi}\right)(z)}
{\mr{det}(M^{\uppsi}(g))}
\partial_{i}^{\uppsi}\bigl(\mr{det}(M^{\uppsi}(g))\bigr)
v_{g[U]}.
\end{multline*}
Therefore by \eqref{12101700},\eqref{12101655} and \eqref{12101656} we deduce 
\begin{multline}
\label{12101725}
\left(\mf{a}_{\mr{Alt}^{\bullet}(U,M)}\circ\bigl(\widehat{\alpha_{\R}}\circ(\mr{Id}_{\R}\widehat{\otimes}\phi)\bigr)_{\times}\right)
(D\otimes v_{g{U}})
=
\\
\sum_{i=1}^{n}
\bigl(\partial_{i}^{\uppsi,\R}\circ\mf{g}_{i}^{\R,\uppsi}\circ\phi_{\times}\bigr)(z)v_{g[U]}
+
\\
\frac{(-1)^{s}}{2}
\sum_{i=1}^{n}
\frac{(\mf{g}_{i}^{\R,\uppsi}\circ\phi_{\times})(z)}{\mr{det}(M^{\uppsi}(g))}
\partial_{i}^{\uppsi}\bigl(\mr{det}(M^{\uppsi}(g))\bigr)
v_{g[U]}.
\end{multline}
Now since Lemma \ref{12100851} we have 
\begin{equation*}
\phi_{\times}(z)=\sum_{i=1}^{n}(\mf{g}_{i}^{\R,\uppsi}\circ\phi_{\times})(z)\otimes\partial_{i}^{\uppsi}
\end{equation*}
therefore,
\begin{equation}
\label{12101726}
(\mf{a}_{\mr{Alt}^{\bullet}(u,M)}\circ\phi_{\times})(z)
=
\sum_{i=1}^{n}(\mf{g}_{i}^{\R,\uppsi}\circ\phi_{\times})(z)\partial_{i}^{\uppsi}.
\end{equation}
Then, \eqref{12101725}, \eqref{12101726} and a standard result imply
\begin{equation*}
\left(\mf{a}_{\mr{Alt}^{\bullet}(U,M)}\circ\bigl(\widehat{\alpha_{\R}}\circ(\mr{Id}_{\R}\widehat{\otimes}\phi)\bigr)_{\times}\right)
(D\otimes v_{g{U}})
=
\mr{div}_{v_{g[U]}}\bigl((\mf{a}_{\mr{Alt}^{\bullet}(u,M)}\circ\phi_{\times})(z)\bigr)v_{g{U}}
\end{equation*}
therefore we obtain by Lemma \ref{12091916} that 
\begin{equation}
\label{12101740}
\left(\mf{a}_{\mr{Alt}^{\bullet}(U,M)}\circ\bigl(\widehat{\alpha_{\R}}\circ(\mr{Id}_{\R}\widehat{\otimes}\phi)\bigr)_{\times}\right)
(D\otimes v_{g{U}}-\mr{div}_{v_{g[U]},\infty,H_{\R}}(z)\otimes v_{g{U}})
=
\ze.
\end{equation}
Now Lemma \ref{12081849}, Lemma \ref{12101042} and \eqref{12101740} imply 
\begin{equation*}
(\mr{Id}_{\R}\widehat{\otimes}\phi)_{\times}
(D\otimes v_{g{U}}-\mr{div}_{v_{g[U]},\infty,H_{\R}}(z)\otimes v_{g{U}})
=
\ze
\end{equation*}
but then since Lemma \ref{12061627} we conclude that 
\begin{equation*}
D\otimes v_{g{U}}=\mr{div}_{v_{g[U]},\infty,H_{\R}}(z)\otimes v_{g{U}}
\end{equation*}
so, the statement follows since Cor.\ref{12011455}.
\end{proof}
\begin{corollary}
\label{12131724}
Let $g\in\mf{T}_{2}^{0}(M)$ be a pseudo-Riemannian metric of signature $s$, $(U,\uppsi)$ be a positively oriented chart of $M$, 
and $z\in\mc{B}^{\infty}(U,H_{\R})\otimes_{\mc{A}(U)}\Gamma(U,TM)$. Then we have 
\begin{multline*}
((\widehat{\gamma_{H_{\R}}})_{\ast}\circ\mr{div}_{v_{g[U]},\infty,H_{\R}})(z)
=
\\
\sum_{i=1}^{n}
\bigl(\partial_{i}^{\uppsi,\widehat{H_{\R}}}\circ(\mf{j}_{H_{\R}})_{\ast}\circ\mf{g}_{i}^{H_{\R},\uppsi}\bigr)(z)
+
\frac{(-1)^{s}}{2}
\sum_{i=1}^{n}
\frac{\bigl((\mf{j}_{H_{\R}})_{\ast}\circ\mf{g}_{i}^{H_{\R},\uppsi}\bigr)(z)}{\mr{det}(M^{\uppsi}(g))}
\partial_{i}^{\uppsi}(\mr{det}(M^{\uppsi}(g))).
\end{multline*}
\end{corollary}
\begin{proof}
The statement follows by Thm.\ref{12092034}, \eqref{12042052} and Lemma \ref{12101042}.
\end{proof}
\begin{lemma}
\label{12041839}
Assume that there exists a\footnote{possibly infinite and uncountable.} family $\mr{G}=\{G_{i}\}_{i\in I}$ of Hausdorff locally 
convex spaces over $\K$ such that $G=\bigoplus_{i\in I}G_{i}$ topological direct sum of the family $\mr{G}$. Furthermore assume 
that $(U,\phi)$ is a chart of $M$. Then by letting $\mc{P}_{\upomega}(I)$ be the set of all finite parts of $I$, there exists a map 
\begin{equation*}
\mr{I}:\mc{B}^{1}(U,G)\times U\to\mc{P}_{\upomega}(I),
\end{equation*}
such that 
\begin{multline*}
(\forall f\in\mc{B}^{1}(U,G))(\forall x\in U)(\forall\nu\in[1,n]\cap\Z)(\forall\psi\in\mc{L}(G,H))
\\
(\psi_{\ast}\circ\partial_{\nu}^{\phi,G})(f)(x)
=
\sum_{k\in\mr{I}(f,x)}\bigl(\partial_{\nu}^{\phi,H}\circ(\psi\circ\mf{i}_{k}^{G})_{\ast}\circ(\Pr_{k}^{G})_{\ast}\bigr)(f)(x).
\end{multline*}
\end{lemma}
Notice that 
\begin{equation*}
(\forall g\in\prod_{i\in I}\mc{L}(G_{i},H))(\exists\,!\psi\in\mc{L}(G,H))(\forall k\in I)(g_{k}=\psi\circ\mf{i}_{k}^{G}),
\end{equation*}
the uniqueness of the linear map $\psi$ follows since the characterization of the dual of the algebric direct sum, 
the continuity of $\psi$ follows by the properties of the topology on $G$.
\begin{proof}[Proof of Lemma \ref{12041839}]
We have 
\begin{equation*}
f(x)=\sum_{k\in\mr{I}(f,x)}(\mf{i}_{k}^{G}\circ\Pr_{k}^{G})(f(x)),
\end{equation*}
then the statement follows since \eqref{12042052}.
\end{proof}
Next the unique decomposition established in Cor. \ref{12011455} for the case $\mh{M}=\mf{L}_{c}^{1}(V,\K,\lambda)$ 
permits to define the integral of a maximal $\R$-valued integrable form defined on an open set $V$ of $\R^{n}$ as in the standard 
case.
\begin{definition}
\label{09140959a}
Let $V$ be an open set of $\R^{n}$, and for every $\omega\in\mr{Alt}^{n}(V,\R^{n};\K,\lambda)$ let $f_{\omega}$ be
the unique map in $\mf{L}_{c}^{1}(V,\K,\lambda)$ such that 
$\omega=f_{\omega}\otimes\bigwedge_{i=1}^{n}(\imath_{V}^{\R^{n}})^{\ast}(dx_{i})$
via the decomposition established in Cor. \ref{12011455} for the case $\mh{M}=\mf{L}_{c}^{1}(V,\K,\lambda)$. Define the map
\begin{equation*}
\int:\mr{Alt}^{n}(V,\R^{n};\K,\lambda)\to\K,\,\omega\mapsto\int f_{\omega}d\lambda_{V}.
\end{equation*}
\end{definition}
\begin{assumption}
We assume for the remaining of the present section that $M$ is oriented. 
\end{assumption}
\begin{notations}
For any $\#\in([0,n]\cap\Z)\cup\{\bullet\}$ and $z\in\{\emptyset,m\}$, we let $\mr{Alt}^{\#}(M;G,\lambda;z)$ denote 
$\mr{Alt}^{\#}(M,M;G,\lambda;z)$ and $\mr{Alt}_{loc}^{\#}(M;G,\lambda;z)$ denote $\mr{Alt}_{loc}^{\#}(M,M;G,\lambda;z)$.
Furthermore we identify $\mr{Alt}^{\#}(M;G,\lambda;m)$ with an $\mc{A}(M)$-submodule of $\mr{Alt}^{\#}(M;G,\lambda)$ 
and identify $\mr{Alt}_{loc}^{\#}(M;G,\lambda;m)$ with an $\mc{A}(M)$-submodule of $\mr{Alt}_{loc}^{\#}(M;G,\lambda)$. 
Finally if $(U,\phi)$ is an oriented chart of $M$, then let $\gamma_{\phi}\in\{1,-1\}$ be such that $\gamma_{\phi}=1$ if 
$(U,\phi)$ is positively oriented, otherwise $\gamma_{\phi}=-1$. 
\end{notations}
Def. \ref{09140959a} and the concept of support as introduced in Def. \ref{09161456} togheter Rmk.\ref{11181348}
permit to give the following definition as in the standard case.
\begin{definition}
\label{09140959b}
Let $\{(U_{\alpha},\phi_{\alpha})\}_{\alpha\in D}$ be a \emph{finite} family of oriented charts of $M$ such that 
$\{U_{\alpha}\}_{\alpha\in D}$ is a covering of $\mr{supp}(\omega)$ moreover by setting $D^{\dagger}=D\cup\{\dagger\}$ and 
$U_{\dagger}=\complement_{M}\mr{supp}(\omega)$, let $\{\psi_{\alpha}\}_{\alpha\in D^{\dagger}}$ be a smooth partition of unity subordinate 
to $\{U_{\alpha}\}_{\alpha\in D^{\dagger}}$. Define
\begin{equation*}
\int:\mr{Alt}^{n}(M;\K,\lambda)\to\K,\,
\omega\mapsto\sum_{\alpha\in D}\gamma_{\phi_{\alpha}}\int(\imath_{U_{\alpha}}^{M}\circ\phi_{\alpha}^{-1})^{\times}(\psi_{\alpha}\omega).
\end{equation*}
\end{definition}
Standard arguments as for instance \cite[13.1.9]{die2} permit to show that the above definition does not depend
by the choice of the covering and of the partition of unity subordinate to it.
\begin{definition}
[\textbf{Weak integral of compactly supported $G$-valued scalarly $\lambda$-integrable maximal forms}]
\label{09281748}
Let $\eta\in\mr{Alt}^{n}(M;G,\lambda)$. Define $\int\eta\in(G^{\prime})^{\ast}$ such that 
\begin{equation*}
\int\eta:G^{\prime}\ni\uppsi\mapsto\int\uppsi_{\times}(\eta)\in\K,
\end{equation*}
called the \emph{weak integral of $\eta$}. We say that $\int\eta$ \emph{belongs to $G$} or that $\int\eta\in G$ iff there 
exists a necessarily unique element $s\in G$ such that $\uppsi(s)=\int\uppsi_{\times}(\eta)$ for every $\uppsi\in G^{\prime}$, 
in such a case and whenever there is no confusion we let $\int\eta$ denote also the element $s$.
\end{definition}
Since $\uppsi_{\times}$ is a morphism of $\mc{A}(M)$-modules by construction, and by the $\K$-linearity of the integral we obtain 
that 
\begin{equation}
\label{11292015}
\int\in\mr{mor}_{\R-\mr{mod}}(\uprho_{\star}^{M}(\mr{Alt}^{n}(M;G,\lambda)),\updelta_{\star}^{\K}(\K)).
\end{equation}
\begin{proposition}
\label{08291022}
Assume that $G$ is quasi-complete and let $\eta\in\mr{Alt}[0]^{n}(M;G)_{[c]}$, then $\int\eta\in G$ namely
\begin{equation*}
(\exists\,!b\in G)(\forall\uppsi\in G^{\prime})\left(\uppsi(b)=\int\uppsi_{\times}(\eta)\right).
\end{equation*}
\end{proposition}
\begin{proof}
The statement is well set since $\mr{Alt}[0]^{N}(M;G)_{[c]}$ is a $\mc{A}(M)$-submodule of $\mr{Alt}^{N}(M;G,\lambda)$ by 
\eqref{11292215}. The statement then follows since the weak integral of any compactly supported continuous $G$-valued map 
against any measure belongs to $G$ as established in \cite[III.38 Cor. 2]{IntBourb}.
\end{proof}
\begin{definition}
Define $G^{\star}\coloneqq\lr{(G^{\prime})^{\ast}}{\sigma((G^{\prime})^{\ast},G^{\prime})}$.
\end{definition}
\begin{theorem}
[\textbf{Vector-measure associated with a scalarly locally integrable $G$-valued form}]  
\label{09171005}
There exists a unique map 
\begin{equation*}
\mf{m}\in\mr{mor}_{\R-\mr{mod}}\left(\uprho_{\star}(\mr{Alt}_{loc}^{n}(M;G,\lambda)),\mc{M}(M,G^{\star})\right),
\end{equation*}
such that 
\begin{equation*}
(\forall\eta\in\mr{Alt}_{loc}^{n}(M;G,\lambda))(\forall g\in\mc{H}(M))\left(\mf{m}_{\eta}(g)=\int g\cdot\eta\right),
\end{equation*}
where $(\cdot)$ is the $\mc{A}(M)$-bilinear map constructed in Prp. \ref{09170948}.
Finally if $G$ is quasi-complete, then 
\begin{equation*}
\mf{m}\in\mr{mor}_{\R-\mr{mod}}\left(\uprho_{\star}(\mr{Alt}_{loc}^{n}(M;G,\lambda)),\mc{M}(M,G)\right).
\end{equation*}
\end{theorem}
\begin{proof}
The uniqueness is trivially true. By construction we have $\mf{m}_{\eta}:\mc{H}(M)\to(G^{\prime})^{\ast}$ so, we need to prove that 
this map is continuous and that $\mf{m}$ is $\R$-linear. To this end let $E$ denote $G^{\star}$ and for every $\uppsi\in G^{\prime}$
let $\mr{b}_{\uppsi}:(G^{\prime})^{\ast}\to\K$, $z\mapsto z(\uppsi)$, thus by the properties of weak topologies
\begin{equation}
\label{10021531}  
E^{\prime}=\{\mr{b}_{\uppsi}\}_{\uppsi\in G^{\prime}}.
\end{equation}
Let $\eta\in\mr{Alt}_{loc}^{n}(M;G,\lambda)$, $g\in\mc{H}(M)$ and $\uppsi\in G^{\prime}$ thus, 
$\uppsi_{\times}(g\cdot\eta)=g\uppsi_{\times}(\eta)$ since the characterization of $(\cdot)$ provided in Prp. \ref{09170948},
therefore 
\begin{equation}
\label{11292159}
\int\uppsi_{\times}(g\cdot\eta)=\int g\uppsi_{\times}(\eta).
\end{equation}
Next for any $\K$-valued locally integrable form $\omega$ on $M$ we know that the map $\mc{H}(M)\ni g\mapsto\int g\omega$ 
belongs to $\mc{M}(M,\K)$ see for instance \cite[10.4]{bourvar}, then by \eqref{11292159} and the fact that 
$\uppsi_{\times}(\eta)\in\mr{Alt}_{loc}^{n}(M;\K,\lambda)$ by Def.\ref{08281845int} we deduce that
\begin{equation}
\label{10021534}  
\mr{b}_{\uppsi}\circ\mf{m}_{\eta}\in\mc{M}(M,\K),
\end{equation}
in particular $\mr{b}_{\uppsi}\circ\mf{m}_{\eta}$ is continuous. Therefore $\mf{m}_{\eta}:\mc{H}(M)\to E$ is continuous namely 
$\mf{m}_{\eta}\in\mc{M}(M,E)$ by \eqref{10021531}, by \eqref{10021534}, since the definition of weak topologies and since 
\cite[I.12 Prp. 4]{BourGT}. $\mf{m}$ is $\R$-linear since Prp. \ref{09170948} and by \eqref{11292015}.
The last part of the statement follows since Prp.\ref{08291022}.
\end{proof}
\begin{corollary}
\label{09141452}
Let $\eta\in\mr{Alt}^{n}(M;G,\lambda)$, $\{(U_{\alpha},\phi_{\alpha})\}_{\alpha\in D}$
be a \emph{finite} family of oriented charts of $M$ such that $\{U_{\alpha}\}_{\alpha\in D}$ is a covering of
$\mr{supp}(\eta)$ moreover by setting $D^{\dagger}=D\cup\{\dagger\}$ and $U_{\dagger}=\complement_{M}\mr{supp}(\eta)$,
let $\{\psi_{\alpha}\}_{\alpha\in D^{\dagger}}$ be a smooth partition of unity subordinate to $\{U_{\alpha}\}_{\alpha\in D^{\dagger}}$.
Thus 
\begin{equation*}
\int\eta=\sum_{\alpha\in D}\int\psi_{\alpha}\cdot\eta,
\end{equation*}
where $(\cdot)$ is the $\mc{A}(M)$-bilinear map constructed in Prp. \ref{09170948}.
\end{corollary}
\begin{proof}
Since $D$ is finite we can define in $\mc{A}(M)$ the map $g=\sum_{\alpha\in D}\psi_{\alpha}$, in particular $g\in\mc{H}(M)$, 
while $g\circ\imath_{\mr{supp}(\eta)}^{M}=\un_{\mr{supp}(\eta)}$ since $\psi_{\dagger}\circ\imath_{\mr{supp}(\eta)}^{M}=\ze_{\mr{supp}(\eta)}$ and 
since $g+\psi_{\dagger}=\un_{M}$ by definition of partition of unity. Therefore $\eta=g\cdot\eta$, then 
$\int\eta=\mf{m}_{\eta}(g)=\sum_{\alpha\in D}\mf{m}_{\eta}(\psi_{\alpha})$ where the second equality follows since Thm. \ref{09171005}.
\end{proof}  
Finally we can establish the following 
\begin{theorem}
[\textbf{Stokes theorem for compactly supported $G$-valued $p$-times continuously differentiable forms}]
\label{08281926}
Let $M$ be with or without boundary, $p\in\Z_{>0}\cup\{+\infty\}$ and $\theta\in\mr{Alt}[p]^{n-1}(M;G)_{[c]}$, thus
\begin{equation*}
\int d_{p}^{[M,G]}\theta=\int(\imath_{\partial M}^{M})^{\times,p,G}(\theta)
\end{equation*}
moreover if $G$ is quasi-complete, then the above integrals belong to $G$.
Here if $\partial M=\emptyset$, then the right-hand side of the equality has to be understood equal to $\ze$.
\end{theorem}
\begin{proof}
The statement is well set since for every $q\in\{p-1,p\}$ the $\mr{Alt}[q]^{\bullet}(M;G)_{[c]}$ is a $\mc{A}(M)$-submodule of 
$\mr{Alt}^{\bullet}(M;G,\lambda)$ by \eqref{11292215}. Let $\uppsi\in G^{\prime}$, thus 
\begin{equation*}
\begin{aligned}
\uppsi_{\times}(d_{p}^{[M,G]}\theta)&=d_{p}^{[M,\K]}(\uppsi_{\times}(\theta)),
\\
(\imath_{\partial M}^{M})_{\star}(\uppsi_{\times})
\bigl((\imath_{\partial M}^{M})^{\times,p,G}\theta\bigr)&=(\imath_{\partial M}^{M})^{\times,p,\K}(\uppsi_{\times}\theta),
\end{aligned}
\end{equation*}
since Prp.\ref{08281542} and \eqref{11271705} respectively. Henceforth the equality follows by \eqref{12042150} and 
since Stokes Theorem. The last sentence of the statement follows since Prp. \ref{08291022}.
\end{proof}
\begin{corollary}
[\textbf{Divergence theorem for compactly supported $H$-valued $p$-times continuously differentiable vector fields}]
\label{12022324}
Let $M$ be with or without boundary, $\omega\in\mr{Alt}^{n}(U,M)$ be a volume form and $p\in\Z_{>0}\cup\{+\infty\}$. 
Then for every $x\in\mc{B}_{c}^{p}(U,H_{\R})\otimes_{\mc{A}(U)}\Gamma(U,TM)$ we have 
\begin{equation*}
\int(\mr{div}_{\omega,p,H_{\R}}(x))\otimes\omega
=
\int(\imath_{\partial M}^{M})^{\times,p-1,\R\widehat{\otimes}H_{\R}}
\bigl(\mf{i}_{\infty,p}^{U,M,\R,H_{\R}}(x)(\mf{b}_{\mr{Alt}^{\bullet}(U,M)}(\omega))\bigr)
\end{equation*}
moreover being $\R\widehat{\otimes}H_{\R}$ complete we have that the above integrals belong to $\R\widehat{\otimes}H_{\R}$. 
Here if $\partial M=\emptyset$, then the right-hand side of the equality has to be understood equal to $\ze$.
\end{corollary}
\begin{proof}
Since Cor.\ref{12011515} and Thm.\ref{08281926}.
\end{proof}
Let $p,q\in\Z_{>0}\cup\{+\infty\}$, $\eta\in\mr{Alt}[q]^{\bullet}(M;G)$, and $h\in\mc{B}^{p}(M,\K)$ namely $h\in\mr{Alt}[p]^{0}(M)$, 
define $h\eta\coloneqq h\wedge_{q,p}^{U,M,G,\K}\eta$. 
\begin{corollary}
\label{09211037}
Let $M$ be with or without boundary. If either  $h\in\mc{B}^{\infty}(M,\K)$ and $\theta\in\mr{Alt}[\infty]^{n-1}(M;G)_{[c]}$, or 
$h\in\mc{B}_{c}^{\infty}(M,\K)$ and $\theta\in\mr{Alt}[\infty]^{n-1}(M;G)$, then we have 
\begin{equation*}
\int hd_{\infty}^{[M,G]}\theta=
\int(\imath_{\partial M}^{M})^{\ast}(h)(\imath_{\partial M}^{M})^{\times,p,G}(\theta)-\int(d_{\infty}^{[M,\R]}h)
\wedge_{\infty,\infty}^{U,M,G,\K}\theta.
\end{equation*}
\end{corollary}
\begin{proof}
$hd_{\infty}^{[M,G]}\theta=d_{\infty}^{[M,G]}(h\theta)-(d_{\infty}^{[M,\K]}h)\wedge_{\infty,\infty}^{U,M,G,\K}\theta$ since 
Thm.\ref{09211057}\eqref{09211057st2}, then the statement follows since Thm. \ref{08281926} applied to $h\theta$ and by 
Prp.\ref{11300106}.
\end{proof}

\bibliography{references}
\bibliographystyle{amsplain}

\end{document}